\documentclass[12pt]{article}
\usepackage[utf8]{inputenc}

\usepackage{graphicx}%
\usepackage{multirow}%
\usepackage{amsmath,amssymb,amsfonts}%
\usepackage{amsthm}%
\usepackage{mathrsfs}%
\usepackage[title]{appendix}%
\usepackage{xcolor}%
\usepackage{textcomp}%
\usepackage{manyfoot}%
\usepackage{booktabs}%
\usepackage{algorithm}%
\usepackage{algorithmicx}%
\usepackage{algpseudocode}%
\usepackage{listings}%

\usepackage{pgfplots}
\pgfplotsset{width=7cm,compat=1.16} 
\usepackage{tkz-graph}
\usepackage{tikz}
\usetikzlibrary{decorations.pathreplacing,shapes.misc,calc}
\usetikzlibrary{arrows.meta,decorations.markings,bending}
\usepackage{tkz-euclide}
\usetikzlibrary{arrows.meta}
\tikzset{>=latex} 
\usepackage{xcolor}
\tikzset{
  every point/.style = {circle, inner sep={.75\pgflinewidth}, opacity=1, draw, solid, fill=white},
  point/.style={insert path={node[every point, #1]{}}}, point/.default={},
  colored point/.style = {point={fill=#1}},
  point name/.style = {insert path={coordinate (#1)}},
  inherit/.style = {point/.style={insert path={node[circle, inner sep={.75\pgflinewidth}, draw, fill, #1]{}}}},
  ar3/.style={decoration={markings,mark=at position 0.46 with {\arrow{stealth}},mark=at position 0.97 with {\arrow{stealth}}}, postaction={decorate}},
  ar4/.style={decoration={markings,mark=at position 0.31 with {\arrow{stealth}},mark=at position 0.82 with {\arrow{stealth}}}, postaction={decorate}},
  ar4r/.style={decoration={markings,mark=at position 0.31 with {\arrow[>=stealth]{<}},mark=at position 0.82 with {\arrow[>=stealth]{<}}}, postaction={decorate}},
  ar5/.style={decoration={markings,mark=at position 0.33 with {\arrow{stealth}},mark=at position 0.86 with {\arrow{stealth}}}, postaction={decorate}},
  ->-/.style={decoration={markings,mark=at position 0.15 with {\arrow{stealth}},mark=at position 0.49 with {\arrow{stealth}},mark=at position 0.95 with {\arrow{stealth}}}, postaction={decorate}},
  -<-/.style={decoration={markings,mark=at position 0.18 with {\arrow[>=stealth]{<}},mark=at position 0.5 with {\arrow[>=stealth]{<}},mark=at position 0.88 with {\arrow[>=stealth]{<}}}, postaction={decorate}},
  ar1/.style={decoration={markings,mark=at position 0.13 with {\arrow{stealth}},mark=at position 0.5 with {\arrow{stealth}},mark=at position 0.97 with {\arrow{stealth}}},postaction={decorate}},
  ar2/.style={decoration={markings,mark=at position 0.11 with {\arrow{stealth}},mark=at position 0.51 with {\arrow{stealth}},mark=at position 0.99 with {\arrow{stealth}}},postaction={decorate}}
}

\colorlet{xcol}{blue!60!black}
\colorlet{myred}{red!80!black}
\colorlet{myblue}{blue!80!black}
\colorlet{mygreen}{green!40!black}
\colorlet{mypurple}{red!50!blue!90!black!80}
\colorlet{mydarkred}{myred!80!black}
\colorlet{mydarkblue}{myblue!80!black}
\tikzstyle{xline}=[xcol,thick,smooth]
\tikzstyle{width}=[{Latex[length=5,width=3]}-{Latex[length=5,width=3]},thick]
\tikzstyle{mydashed}=[dash pattern=on 1.7pt off 1.7pt]
\tikzset{
  traj/.style 2 args={xline,postaction={decorate},decoration={markings,
    mark=at position #1 with {\arrow{stealth}},
    mark=at position #2 with {\arrow{stealth}}}
  }
}

\usepackage{graphicx}
\usepackage{hyperref} 
\usepackage{amsmath}
\usepackage{amssymb}
\usepackage{amsthm}
\usepackage{amsfonts} 
\usepackage{bm}
\usepackage{cases}
\usepackage{xcolor}
\usepackage{color,soul}
\usepackage{caption} 
\usepackage{subcaption} 
\usepackage{graphicx}
\usepackage{epstopdf}
\usepackage{enumerate}
\usepackage{ifthen}
\usepackage{mathtools}  
\usepackage{notes2bib}
\usepackage{placeins} 
\usepackage{setspace}
\usepackage{sidecap} 
\usepackage{url}
\usepackage{ucs} 
\usepackage{float}

\def\epsilon{\varepsilon}
\newcommand{\BF}{\bf\boldmath }
\newcommand{\EM}{\BF }
\def\em{\bf}

\newcommand{\ie}{{\it{i.e.}}}

\newcommand{\eps}{\varepsilon}

\newcommand{\R}{{\mathbb{R}}}
\newcommand{\bN}{{\mathbb{N}}}
\newcommand{\bZ}{{\mathbb{Z}}}

\def\cA{{\cal A}}
\def\cC{{\cal C}}
\def\cD{{\cal D}}

\def\cI{{\cal I}}
\def\cL{{\cal L}}
\def\cM{{\cal M}}

\def\cO{{\cal O}}
\def\cP{{\cal P}}
\def\cR{{\cal R}}

\def\cS{{\cal S}}

\def\bR{{\mathbb R}}
\def\NW{{\cP}}

\newcommand{\beq}{\begin{linenomath}\begin{equation*}} 
\newcommand{\eeq}{\end{equation*}\end{linenomath}} 
\newcommand{\beqn}{\begin{linenomath}\begin{equation}} 
\newcommand{\eeqn}{\end{equation}\end{linenomath}} 
\usepackage[normalem]{ulem}

\newtheorem*{lemmA}{Lemma A}
\newtheorem{theorem}{Theorem}
\newtheorem{proposition}{Proposition}[subsection]
\newtheorem{definition}[proposition]{Definition}
\newtheorem{lemma}[proposition]{Lemma}

\newtheorem{corollary}[proposition]{Corollary}
\newtheorem{example}[proposition]{Example}

\newtheorem{theoremX}{Theorem}

\def\Dto{\overset{D}{\geq}}

\def\leadsto{\Sto}
\def\Ato{\succcurlyeq_A}
\def\Dto{\succcurlyeq_D}
\def\Fto{\succcurlyeq_F}

\def\CRtoI{\succcurlyeq_{\cC_{F,d_1}}}
\def\CRtoII{\succcurlyeq_{\cC_{F,d_2}}}
\def\NWto{\succcurlyeq_{\cP_F}}
\def\NWeq{=_{\cP_F}}
\def\Sto{\succcurlyeq_S}
\def\lra{\overset{S}{=}}

\def\Feq{\overset{F}{=}}
\def\Seq{\overset{S}{=}}

\definecolor{dgn}{rgb}{0.2, .6, 0.0}

\newcommand{\black}{\color{black}{}}
\newcommand{\blue}{\color{black}{}}

\definecolor{orange}{rgb}{0.8, .2, 0.0}

\definecolor{jim}{rgb}{0,.5,.5}
\definecolor{rob}{rgb}{0,0.4, 0.6}

\newcommand{\T}{{\mathbb{T}}}

\DeclareMathOperator{\Down}{Down}
\DeclareMathOperator{\Up}{Up}
\DeclareMathOperator{\Node}{Node}
\def\Om{\Omega}

\newcommand{\bydef}{\overset{\mathrm{def}}{=\joinrel=}}

\def\cT{{\cal T}}

\def\sR{\cR}

\usepackage{enumitem}

\newif\ifincludeXX
\includeXXtrue 

\usepackage{lineno}
\newif\ifincludeOLD
\includeOLDtrue 
\includeOLDfalse 
\begin{document}

\title{Streams and Graphs\\
of Dynamical Systems.}
\author{Roberto De Leo and James A. Yorke}

\date{\today}
\maketitle

{R. De Leo at
       Department of Mathematics, Howard University, Washington DC 20059,\\
       {roberto.deleo$@$howard.edu}  \\ 
      
      J.A. Yorke at
      Institute for Physical Science and Technology and the Departments of Mathematics and Physics, University of Maryland College Park, MD 20742,\\
      {yorke@umd.edu}
}

\begin{abstract}
    While studying gradient dynamical systems, Morse introduced the idea of encoding the qualitative behavior of a dynamical system into a graph.
    Smale later refined Morse's idea and extended it to Axiom-A diffeomorphisms on manifolds.
    In Smale's vision, nodes are indecomposable closed invariant subsets of the non-wandering set with a dense orbit and there is an edge from node $M$ to node $N$
    (we say that $N$ is {\em downstream} from $M$) if the unstable manifold of $M$ intersects the stable manifold of $N$.
    Since then, the decomposition of the non-wandering set was studied in many other settings, while the edges component of Smale's construction has been often overlooked.

    In the same years, more sophisticated generalizations of the non-wandering set, introduced by Birkhoff in 1920s, were elaborated first by Auslander in early 1960s, by Conley in early 1970s and later by Easton and other authors.
    In our language, each of these generalizations involves the introduction of a closed and transitive extension of the prolongational relation, that is closed but not transitive. 

    In the present article, we develop a theory that generalizes at the same time both these lines of research. We study the general properties of closed transitive relations (which we call {\em streams}) containing the space of orbits of a discrete-time or continuous-time semi-flow and we argue that these relations play a central role in the qualitative study of dynamical systems. 
    All most studied concepts of recurrence currently in literature can be defined in terms of our streams.
    Finally, we show how to associate to each stream a graph encoding its qualitative properties. 
    Our main general result is that each stream of a semi-flow with ``compact dynamics'' on a connected space has a connected graph.
    The range of semi-flows covered by our theorem goes from 1-dimensional discrete-time systems like the logistic map up to infinite-dimensional continuous-time systems like the semi-flow of quasilinear parabolic reaction-diffusion partial differential  equations.
\end{abstract}



\maketitle

\section{Introduction}
This article is about the following two fundamental ingredients in the qualitative description of a dynamical system: the invariant sets of the system and the dynamical relations between them.
The foundational example we have in mind, introduced and studied by Smale in~\cite{Sma61}, is the  flow of a Morse gradient vector field on a compact manifold.
In this case, each invariant set is a fixed point and there are only finitely many of them.
Heteroclinic trajectories of the system determine a relation between invariant sets that we express as ``being downstream from'' and can be encoded into a graph. 

\black We aim \black at providing a natural framework that unifies and widely generalizes the following two fundamental and seminal ideas and results on the subject.
\smallskip\par\noindent
{\bf 1. Smale's Graph of a Dynamical System.} In~\cite{Sma67}, as a generalization of the example above, Smale introduced the concept of {\em graph of a dynamical system} and studied the case of Axiom-A diffeomorphisms on a compact manifold.
In that case, given
an Axiom-A diffeomorphism $F$, the nodes of the associated graph are (finitely many) closed, disjoint invariant subsets $N_i$ of the non-wandering set $NW_F$ such that: i) $F|_{N_i}$ is topologically transitive; ii) $\cup N_i=NW_F$.
There is an edge from node $N_1$ to node $N_2$ if and only if the unstable manifold of $N_1$ has non-empty intersection with the stable manifold of $N_2$.
In our terminology, we say that $N_1$ is downstream from $N_2$.
This decomposition of $NW_F$ in ``nodes'' was called by Smale its {\em spectral decomposition} and it has been widely generalized, e.g. to 
maps on the interval~\cite{dMvS93}, homeomorphisms on compact~\cite{Akin93,AH94} and non-compact~\cite{DLRW13,KNY18} metric spaces and on non-metrizable spaces~\cite{Oh22}.
On the contrary, the edges -- and so the whole graph point of view -- did not receive, in the authors knowledge, much attention by the dynamical system community, 
with the following noticeable exceptions: an article by D. Mizin~\cite{Miz02} (in Russian; see Chap.~9 in Osipenko's book~\cite{Osi06} for a presentation in English of Mizin's results); several works by Fiedler and Rocha on the graphs of the semi-flows of some class of semi-linear parabolic PDEs~\cite{FR96,FR00,FR23}; two recent works by the present authors on the logistic map and the Lorenz system~\cite{DLY20,DLY21}.
The graph introduced by Mizin is the same we introduced independently in~\cite{DLY20}.
\medskip\par\noindent
{\bf 2. Quasi-orders and Auslander's generalized recurrence.} In~\cite{Aus63}, 
as a generalization of the concept of non-wandering point of a semi-flow $F$, 
Auslander studied the smallest quasi-order (i.e. reflexive and transitive) relation $\Ato$ that extends the orbit relation (i.e., for all $x$ in the phase space, $x\Ato y$ for all $y$ in the orbit of $x$) and is topologically closed.
Auslander used $\Ato$ to generalize the set $NW_F$ as follows: $x$ is a {\em generalized recurrent point} if and only if either $x$ is a fixed point or there is a $y$ such that $x\Ato y$ and $y\Ato x$. 
Thanks to the transitivity of the relation, the set of all these generalized recurrent points has the nice property, similar to the spectral decomposition of Smale, to naturally decompose into the disjoint union of closed invariant sets (which we still call {\em nodes}).
Notice that nodes in Smale's spectral decomposition (i.e. ``non-wandering nodes'') can be sometimes not disjoint (e.g. see Example~\ref{ex: physical pendulum} and~\ref{ex:non transitive}) while nodes of Auslander's generalized recurrent set, as mentioned above, are always disjoint. 
More details, with a thorough discussion about relations and generalized recurrence, can be found in a recent work by Akin and Auslander~\cite{AA10}.

\smallskip

In a nutshell, in this work we extend Auslander's idea by considering {\em all} (rather than just the smallest one) closed quasi-orders that extend the orbit relation of $F$ (we call these relations {\em streams}), and by showing that their qualitative features can be encoded into a graph, similarly to how Smale did for Axiom-A systems.
Each stream $S$ has a set $\cR_S$ of $S$-recurrent points, defined in the same way we defined above the generalized recurrent points, and on $\cR_S$ is defined a natural equivalence relation that decomposes it into the disjoint union of closed forward-invariant sets. 
We show that, besides Auslander's generalized recurrent points, also Conley's chain-recurrent points~\cite{Con72b} and Easton's strong-recurrent points~\cite{Eas77} arise naturally as the sets of recurrent points of a stream, that we call respectively the {\em chains} and the {\em summable chains} streams.
In particular, we show that several common properties of these sets derive from general properties of streams.
Finally, we extend Smale's graph idea to streams by associating to each stream $S$ a directed graph $\Gamma_S$ as follows: 
the nodes of $\Gamma_S$ are the equivalence classes of the recurrent points of $S$ and there is an edge from node $N_1$ to node $N_2$ if and only if $N_1\Sto N_2$ (Defs.~\ref{def: N1>=N2} and~\ref{def graphs}).

We point out that some authors already used similar relations in the study of general properties of dynamical systems.
For instance Akin, inspired by~\cite{Aus63}, based his monography on topological dynamical systems~\cite{Akin93} on streams (although he does not use this terminology).
Duarte and Torres used similar relations in the framework of set-valued dynamical system~\cite{DT06}.
More recently, Akin and Auslander used streams (again, not with this terminology) in their work on compactification of locally compact dynamical systems~\cite{AA10}.
\black

\smallskip
Our main results are the following. Assume that the semi-flow $F$ has a compact global trapping region (Definition~\ref{def: trapping region}; notice that this condition is trivially satisfied when $X$ is compact). Then:
\begin{enumerate}
    \item The graph of the prolongational relation of $F$ is connected (Theorem~\ref{thm:NW}).
    \item The graph of each stream of $F$ is connected (Theorem~\ref{thm: connectedness}).
    \item Given any stream $S$ of $F$, for each node $N$ there is at least a top node $T$ (i.e. $T$ has no incoming edges) and a bottom node $B$ (i.e. $B$ has no outgoing edges) such that $T\Sto N\Sto B$ (Theorem~\ref{thm:topnode}).
    \item The nodes and graph of the chains stream of a continuous-time semi-flow coincide with those of its time-1 map (Theorem~\ref{thm: cont=disc}).
    This extends a result by Hurley~\cite{Hur95} that the set of the recurrent points of the two systems coincide.
    \item If the smallest (i.e. Auslander's) stream has at most countably many nodes, then the chains streams and Auslander's stream (and so every stream in between, including all summable chains streams)  are equal to each other (Theorem~\ref{thm: CF=AF}). 
    In particular, their nodes and graphs are equal.
\end{enumerate}
Moreover, in Section~\ref{sec:T-unimodal} we classify the prolongational graph of T-unimodal maps (Definition~\ref{def:T-unimodal}), namely continuous unimodal maps without wandering intervals and without non-topological attracting periodic orbits.
This extends and clarifies some aspect of our classification of the graph of the logistic map in~\cite{DLY20} and of the graph of T-unimodal maps with a chaotic attractor in~\cite{ADL23}.




%
\section{Setting, main definitions and tools}
\medskip\noindent
Throughout the article, {\BF$X$} will denote a {\sl metrizable} connected topological space. 
\blue
We will usually denote by $x,y,z$ points in $X$ and by $d$ any distance function compatible with the topology of $X$. 
\black

\medskip
The main object of interest of this work is a discrete-time or continuous-time semi-flow, as defined below.
%
\begin{definition}\label{semi-flow} 
    A {\em semi-flow} on a topological space $X$ is a continuous map 
    {\EM$F:\T\times X\to X$}, where either $\T=0,1,2,\dots$ (discrete time) or $\T=[0,\infty)$ (continuous time), satisfying the following properties:
    \begin{enumerate}
        \item 
    $F^0(x)=x\text{ for each }x\in X$;
    \item
    $F^{t_1+t_2}(x)=F^{t_2}(F^{t_1}(x))\text{ for each }x\in X\text{ and }t_1,t_2\geq0$.
    \end{enumerate}
\end{definition}
\noindent
Notice that the discrete case consists in the iterations of the time-1 map $F^1$.
\black
Discussions here will usually focus on discrete time semi-flows, and we leave it to the reader to adapt the notation to the continuous-time case.
\black

Our motivational example is the (continuous-time) semi-flow induced by a (finite- or infinite-dimensional) ODE
\beq\label{ODE}
    \dot x = f(x),
\eeq
whose solutions are unique and defined for all times $t\geq0$.
In this case, $F^{t_0}(x_0)$ is defined as the value at time $t=t_0$ of the solution of the ODE that passes through $x_0$ at time $t=0$.

The following definitions are crucial for this article.
\begin{definition}
    \label{def:orbit}
    Given a semi-flow $F$, we write {\EM $x\Fto y$} if $y=F^t(x)$ for some $t\ge0$ and we say that $y$ is {\EM $F$-downstream from} $x$.
    We write {\EM $x\Feq y$} if $x\Fto y$ and $y\Fto x$.
    We call {\em orbit space} of $F$ the set 
    \beq
    \text{\EM$\cO_F$} = \{(x,y): x\Fto y\},
    \eeq
    namely the binary relation on $X$ whose symbol is $\Fto$.
    The orbit of any given point $x$ under $F$ is
    \beq
    \text{\EM$\cO_F(x)$} = \{y: (x,y)\in\cO_F\}.
    \eeq

    In analogy with the orbit space, we call {\em limit space} of $F$ the binary relation
    \beq
    \text{\EM$\Om_F$} = \{(x,y):\text{ there is }t_n\to\infty\text{ as }n\to\infty\text{ such that }F^{t_n}(x)\to y\}.
    \eeq
    The limit set of a point $x$ under $F$ is the set
    \beq
    \text{\EM$\Om_F(x)$}=\{y: (x,y)\in \Om_F\}.
    \eeq
    
    We say that $x$ is {\em fixed} for $F$, or that $x$ is a {\em fixed-point} of $F$, if $\cO_F(x)=\{x\}$; that $x$ is {\em periodic} if either $x$ is fixed or there is a $y\neq x$ such that $x\Feq y$; that $x$ is {\EM recurrent} (or, to emphasize the map, $F$-recurrent) if $x\in\Om_F(x)$.
    Finally, we say that a function $L\in C^0(X)$ is a {\em Lyapunov function} for $F$ if $x\Fto y$ implies $L(x)\geq L(y)$.
\end{definition}
%

\subsection{
Compactness and Trapping regions} 
%
In this article, we focus on the  case of semi-flows with ``compact dynamics''. 
By this, we mean semi-flows such that each of their orbit are contained within a compact set.
Of course we could just assume $X$ to be compact.
Such a strong assumption, though, would rule out cases important in applications, as examples below will show.
In fact, even local compactness would be too strong because it would rule out infinite-dimensional dynamical systems (e.g. parabolic PDEs).
Hence, from now on, we do not make any extra assumption on $X$ and we rather focus on the case when all orbits of $F$ are attracted by some compact set $Q\subset X$, as illustrated below.
%
\begin{definition}
    \label{def: trapping region}
    Let $Q\subset X$.
    We say that $Q$ is a {\em trapping region} for $F$ if it is closed and forward-invariant under $F$.
    Set ${N_\eps(Q)}=\{y:d(y,Q)<\eps\}$.
    A trapping region $Q$ is:
    \begin{itemize}
        \item {\em stable} if, for every $\eps>0$, there is a $\delta>0$ such that \\
        $F^t(N_\delta(Q))\subset N_\eps(Q)$ for all $t\geq0$;
        \item {\em uniformly asymptotically stable} if there is an $\eta>0$ such that, for every $\eps,\delta\in(0,\eta)$, there is $\tau>0$ such that\\ $F^t(N_\delta(Q))\subset N_\eps(Q)$ for all $t\geq\tau$;
        \item {\em globally attracting}  
        \blue if, for every $x$ and $\eps>0$, there is a $T$ such that $x\in N_\eps(Q)$ for all $t\geq T$.
    \end{itemize}
    We say that a trapping region $Q$ is a {\em global trapping region} if it is globally attracting and uniformly asymptotically stable.
    Finally, we say that $F$ is a {\em semi-flow with compact dynamics} if $F$ has a compact global trapping region.
\end{definition}
Notice that $X$ is, trivially, a global trapping region for every semi-flow on $X$.
Hence, when $X$ is compact, each semi-flow on $X$ has compact dynamics.
We leave to the reader the proof of the next statement.
\begin{proposition}
  Let $F$ be a semi-flow with compact dynamics and let $Q$ be a compact global trapping region for $F$.
  \black
  Then, for each $x$, $\Om_F(x)$ is a non-empty subset of $Q$.
\end{proposition}
\begin{proposition}
    \label{prop:recurrent pt}
    Let $F$ be a semi-flow with compact dynamics.
    Then, for every $x$, there is an $F$-recurrent point in $\Omega_F(x)$.
\end{proposition}
\begin{proof}
    Let $y\in\Om_F(x)$. 
    If $y$ is not $F$-recurrent, then $y\not\in\Om_F(y) $ and so $\Om_F(y)$ is a proper subset of $\Om_F(x)$.
    Let $\Lambda$ be the set of all subsets of $\Om_F(x)$ that are the limit set of some point in $\Om_F(x)$. 
    The collection of such sets is partially ordered by inclusion. 
    By Zorn's lemma, there is a maximal ordered collection $\cC$ of these sets. 
    They are all compact, so their intersection $\Omega^*$ is compact and non-empty. 
    Since the intersection of invariant sets is invariant, for each $z\in \Omega^*$ we have that $\Om_F(z)\subset \Omega^*$.
    By maximality, $\Om_F(z)$ must be in $\cC$ and it cannot be strictly smaller than $\Omega^*$. 
    Hence, $z\in\Om_F(z)=\Omega^*$.
\end{proof}
\blue
Finally, we present a technical lemma that will be relevant in Section~\ref{sec:CR}.
\begin{definition}
    Given any set $A\subset X$, we use the notation
    \BF $$
    \cI_F(A) = \bigcup_{t\in[0,\infty)}F^t(A).
    $$
\end{definition}
The reader can verify the following:
\begin{proposition}
    $\cI_F(A)$ is the smallest $F$-invariant set containing $A$.
\end{proposition}
\begin{lemma}
    \label{lemma: invariant set}
    Assume that $X$ is completely metrizable and let $d$ be a complete metric for $X$.
    Let $F$ have compact dynamics and let $Q$ be a compact global trapping region for $F$.
    Let $\eta>0$ be small enough that, for every $\eps\in(0,\eta)$, there exists $\tau_\eps>0$ such that 
    $F^t(Q_\eps)\subset Q_{\eps/2}$ for all $t\geq\tau_\eps$.
    Then
    $$
    \bigcap_{\eps\in(0,\eta)}
    \cI_F(Q_\eps)
    =Q.
    $$
    In other words, for every $\delta>0$, there is an $\eps>0$ such that $\cI_F(Q_\eps)\subset Q_\delta$.
\end{lemma}
\begin{proof}
    Let $y\in \cap_{\eps\in(0,\eta)}
    \cI_F(Q_\eps)\setminus Q$.
    Then we can find sequences $z_i\in X$ with $d(z_i,Q)\to0$ and $t_i>0$ such that $F^{t_i}(z_i)=y$.
    For each $z_i$ there is at least a $q_i\in Q$ such that $d(z_i,Q)=d(z_i,q_i)$.
    Since $Q$ is compact, the $q_i$ have a converging subsequence. 
    Hence, we can assume without loss of generality that $q_i\to\bar q\in Q$.
    
    Since $d(z_i,q_i)\to0$ and $q_i$ converges, the sequence $z_i$ is Cauchy and so it converges to $\bar q$.
    If the sequence of the corresponding $t_i$ were bounded, we could extract a converging subsequence $t_i\to \tau$ and, by continuity, $F^{\tau}(\bar q)=y$.
    This is impossible because $Q$ is forward-invariant, so necessarily $t_i\to\infty$.
    Therefore, there are points $z_i$ arbitrarily closed to $\bar q$ whose orbit passes through $y$ after an arbitrarily long time.   
    
    We show now that this leads to a contradiction.
    Indeed, set an $\eps_0\in(0,\eta)$ such that $\eps_0<2d(y,Q)$.
    Then $F^t(Q_{\eps_0})\subset Q_{\eps_0/2}$ for all $t\geq\tau_{\eps_0}$.
    Now, take a $z_i$ with $d(z_i,\bar q)<\eps_0$ such that $t_i>\tau_{\eps_0}$ (since $t_i\to\infty$, almost all $z_i$ satisfy this condition).
    Then $F^{t_i}(z_i)\in Q_{\eps/2}$, so that $d(F^{t_i}(z_i),Q)<\eps_0/2<d(y,Q)$, against the assumption that $F^{t_i}(z_i)=y$.
\end{proof}
\black
While having compact dynamics is a quite restrictive condition, there are many important classes of semi-flows that satisfy it. 
We present some of them in the examples below.
\begin{example}[\em Logistic map]
    Consider the logistic map $$\ell_\mu(x)=\mu x(1-x),$$ where $\mu\in(2,4)$ and $X=(0,1)$, and set $c_k(\mu)=\ell^k_\mu(1/2)$. 
    Then $Q=[c_2(\mu),c_1(\mu)]$ is a compact global trapping region for $\ell_\mu$. 
\end{example}
\begin{example}[\em Lorenz system]
    Consider the Lorenz ODE system in $X=\R^3$
    \beq
    \begin{cases}
    \label{eq:Lor}
    \dot x &=-\sigma x +\sigma y  \\
    \dot y &=-xz + r x -y\\
    \dot z &=\phantom{+}xy - \beta z
    \end{cases}.    
    \eeq
    The convex function 
    $$
    L(x,y,z) = r x^2 + \sigma y^2 + \sigma(z-2r)^2
    $$
    is a Lyapunov function for the flow $F^t$ of the Lorenz system far enough from the origin.
    More precisely, 
    $$
    \frac{d}{dt}L(F^t(x_0,y_0,z_0))\Big|_{t=0}<0
    $$
    for all $(x_0,y_0,z_0)$ far enough from the origin.
    Hence, every ball of large enough radius centered at the origin is a compact global trapping region for $F$.
\end{example}
\begin{example}[\em Reaction-Diffusion PDEs]
    One of the simplest infinite-dimensional dynamical systems to which our machinery applies is the reaction diffusion scalar PDE
    \beq
    u_t = u_{xx} + f(x,u,\nabla u),\; u(0,x)=u_0(x),
    \eeq
    where $x\in(0,1)$.
    PDEs of this type arise from many phenomena in chemistry reactions, population dynamics, biophysics and more (see~\cite{FR23} for a large number of references).
    It turns out that the most convenient setting for these PDEs is the functional space $X=L^2(0,1)$ of all square-summable functions on the unit interval.
    It can be proved, indeed, that the weak version of the PDE above admits a continuous semi-flow 
    $F:\R\times L^2(0,1)\to L^2(0,1)$.
    Moreover, for some choice of the right-hand-side $f$, the corresponding semi-flow $F$ has a compact global trapping region.
    This is the case, for instance, when $f(x,u,0)\cdot u<0$ for large $|u|$ and $f(x,u,p)$ has subquadratic growth in $|p|$.
    In this case, indeed, one can prove that there is a ball $Q$ in the subset of $H^1_0$ functions (namely $L^2$ functions whose weak first derivative is also $L^2$ and that are limit of continuous functions on $(0,1)$ with compact support) such that, for each $x$, there is a $T$ such that $F^t(x)\in Q$ for each $t\geq T$.
    Since $H^1_0(0,1)$ is a compact subset of $L^2(0,1)$, then $Q$ is compact as a subset of $L^2$ (see~\cite{Rob01}, Chapter 11).
    Hence, Q is a compact global trapping region for this system.
    The graphs of many of this type of PDEs are under investigation by Fedler and Rocha~\cite{FR96,FR00,FR23}.
\end{example}
\begin{example}[\em Navier-Stokes]
    Consider the 2D Navier-Stokes PDEs on the 2-torus $\T^2$, given by
    $$
    \begin{cases}
    \dot u-\nu\Delta u + (u\cdot\nabla)u+\nabla p=f&\\
    \nabla\cdot u =0&\
    \end{cases}, 
    $$
    where $\nu$ is a positive parameter and the unknowns $u$ and $p$ are a square-integrable vector field and a square-integrable function on $\T^2$.
    As in the example above, this PDE defines a continuous semi-flow on $X=L^2$ and one can prove that there is a ball $Q$ in the $H^1$ norm that each trajectory
    enters in finite time. Notice that a set that is bounded with respect to the $H^1$ norm is compact in the $L^2$ topology.
    This ball $Q$ is a compact trapping region for the semi-flow  (see~\cite{Rob01}, Chapter 12).
\end{example}
\begin{example}[\em Ricci flow]
    V. and T. Ivancevic showed in~\cite{II11} that every reaction-diffusion model in literature can be seen as a special case of the celebrated Ricci PDE on manifolds
    $$
    \dot g = Ricci_g,
    $$
    where $g$ is a smooth metric on a manifold $M$ and $Ricci_g$ is the corresponding Ricci manifold, a $(0,2)$ tensor obtained as contraction of the curvature tensor of $g$.
    This shows that Ricci flow has in many cases a compact global trapping region.
\end{example}
%
\subsection{Chains} 
Since $X$ is metrizable, given a semi-flow $F$ we can use distances to define finite sequences of points called chains, as defined below.
Chains play a major role in defining the main types of recurrence.
Below we define chains only for discrete semi-flows.
In Section~\ref{sec:CR} we show that this is enough for our goals (see Theorem~\ref{thm: cont=disc}).
\begin{definition}
    \label{def:chains}
    A {\EM $(F,d,\eps)$-bchain} (or simply {\EM $\eps$-bchain}, when there is no ambiguity) from point $x$ to point $y$ \blue of length $k$ \black is a finite sequence $C$ of $k$ points $z_1,\dots,z_k$, with $k>1$, such that:
    \begin{enumerate}
    \item $d(z_1,x)\leq\eps$;
    \item $d(z_k,y)\leq\eps$;
    \item $d(F(z_i),z_{i+1})\leq \eps$ for $i=1,\dots,k-1$.
    \end{enumerate}
    We say that an $\eps$-bchain $C$ is:
    \begin{enumerate}    
    \item an {\EM $\eps$-chain} if $z_1=x$ and $z_k=y$.
    \item an {\EM $\eps$-$\alpha$chain}  if $z_k=y$ and $F(z_i)=z_{i+1}$ for all $i=2,\dots,k-1$.
    \item an {\EM $\eps$-$\omega$chain}  if $z_1=x$ and $F(z_i)=z_{i+1}$ for all $i=1,\dots,k-2$.
    \item an {\EM $\eps$-$\alpha\omega$chain}  if $F(z_i)=z_{i+1}$ for all $i=1,\dots,k-2$.
    \end{enumerate}
    Finally, we say that $C$ is an {\EM $(F,d,\eps)$-$\Sigma$chain} if $$\sum_{i=0}^{k-1} d(x_i,F(x_i))<\eps.$$
    \blue
    Sometimes, when $x=y$, we refer to an $\eps$-chain as an {\BF $\eps$-loop based at $x$}.
\end{definition}
In other words, an $\eps$-bchain is like an $\eps$-chain where it is allowed to make a jump {\em before} the first application of $F$; an $\eps$-$\alpha$chain is an $\eps$-bchain with a single jump, the very first one; an $\eps$-$\omega$chain is an $\eps$-bchain with a single jump, the very last one; an $\eps$-$\alpha\omega$chain is an $\eps$-bchain with just two jumps: the very first one and the very last one.

The concept of chains for maps goes back at least to Birkhoff~\cite{Bir26}. To the authors' knowledge, chains for (semi-)flows (see Definition~\ref{def: continuous-time chains}) were introduced by Conley~\cite{Con72} to define the chain-recurrent set.
Later Bowen~\cite{Bow75}, while studying the shadowing phenomenon, considered also the case of
infinite chains and called them {\em pseudo-orbits}. Some authors still use the term chains for finite truncations of pseudo-orbits~\cite{Med21,BNP24}.
Nowadays, though, the terms {\em chains} and {\em pseudo-orbits} are often used as synonyms~\cite{Has08,Tho82}.
Finally, $\Sigma$chains were introduced by Easton, under the name of {\em strong chains}, in~\cite{Eas77}.

As a first application of chains, we provide a characterization of recurrent points in metric spaces.
We leave its proof to the reader.
\begin{proposition}
    \label{prop:recurrent}
    A point $x$ is recurrent under $F$ if and only if, for every $\eps>0$, there is an $\eps$-$\omega$chain from $x$ to itself.
\end{proposition}
\section{Quasi-orders and ``being downstream''}
\label{sec:quasi-orders}
The relation $\Fto$ is clearly reflexive and transitive.
It is therefore natural generalizing the concept of 
{\em being downstream}
by considering the following class of relations.
\begin{definition}
    \label{def: QO}
    A quasi-order on $X$ is a binary relation $D\subset X\times X$ that is reflexive and transitive, namely:
    \begin{enumerate}
        \item $(x,x)\in D$ for every $x$;
        \item if $(x,y)\in D$ and $(y,z)\in D$, then $(x,z)\in D$.
    \end{enumerate}
    We sometimes use the notation $x\Dto y$ to say that $(x,y)\in D$.
    We call $\Dto$ the symbol of $D$. 
    Then reflexivity means that $x\Dto x$ for all $x$ and transitivity that, if $x\Dto y$ and $y\Dto z$, then $x\Dto z$.
\end{definition}
\begin{example}
  The orbit space $\cO_F$ of a semi-flow $F$ is a quasi-order on $X$.
  The symbol $\Fto$ introduced in Definition~\ref{def:orbit} is precisely the symbol of this quasi-order.
\end{example}
The following definition generalizes Definition~\ref{def:orbit} to quasi-orders.
\begin{definition}
    Given a relation $D$, we use the notation
    $$
    \Down_D(x) = \{y: x\Dto y\},
    \;\;
    \Up_D(y) = \{x : x\Dto y\},
    $$
    and we say that the points of $\Down_D(x)$ are {\EM $D$-downstream} from $x$ (or simply {\EM downstream} when there is no ambiguity) and points in $\Up_D(x)$ are {\EM $D$-upstream} from $x$ (or simply {\EM upstream} when there is no ambiguity).
    Finally, for a $M\subset X$ we set
    \beq
    \Down_D(M)=\cup_{x\in M}\Down_D(x),
    \;\;
    \Up_D(M)=\cup_{x\in M}\Up_D(x).
    \eeq
    We say that $M$ is {\EM $D$-invariant} if $\Down_D(M)\subset M$.
    We refer sometimes to the set $\Up_D(M)$ as the {\EM $D$-basin} of $M$.
\end{definition}
\begin{proposition}
    Let $D$ be a quasi-order. Then 
    $\Down_D(\Down_D(x))=\Down_D(x)$ and
    $\Up_D(\Up_D(x))=\Up_D(x)$ for every $x$.
    In particular, for every $M\subset X$, $\Down_D(M)$ is $D$-invariant.
\end{proposition}
\begin{proof}
    It is an immediate consequence of transitivity.
\end{proof}
Notice that, in general, the set of all points downstream from a closed set of points is not closed.
\begin{example}
    Let $F$ be the flow of the vector field $\frac{\partial}{\partial x}$ on $R^2$.
    Consider the quasi-order $\cO_F\subset\bR^4$ and set $C=\{xy=-1:y>0\}$.
    The set $C$ is a closed subset of the plane but the sets
    \beq
    \Down_{\cO_F}(C)=\{(x,y):y>0,xy\geq-1\},
    \eeq
    \beq
    \Up_{\cO_F}(C)=\{(x,y):y>0,xy\leq-1\}
    \eeq
    are not.
\end{example}
%
\subsection{Closed quasi-orders and Lyapunov functions}
It turns out that strong properties hold for closed quasi-orders, namely quasi-orders that are closed in the product topology of $X\times X$.

A first important property is that, when $D$ is closed, the sets $\Down_D(C)$ and $\Up_D(C)$ are closed for every compact set $C$.
%
\begin{proposition}
    \label{prop:Down(compact)}
    If $K\subset X$ is compact, then $\Down_D(K)$ and $\Up_D(K)$ are closed.
    \black In particular, \black when \black $X$ is compact, if $C$ is closed then $\Down_D(C)$ and $\Up_D(C)$ are closed.
\end{proposition}
\begin{proof}
    Let $x_n\in\Up_D(K)$ be a sequence converging to a point $\bar x\in X$.
    Then there exists a sequence $y_n\in K$ such that $(x_n,y_n)\in D$. 
    Since $K$ is compact, there is a subsequence $y_{n_k}$ converging to a $\bar y\in K$.
    Hence, in $D$, $(x_{n_k},y_{n_k})\to(\bar x,\bar y)$.
    Since $D$ is closed, $(\bar x,\bar y)\in D$, namely $\bar x\in\Up_D(K)$ .
    Hence, $\Up_D(K)$ is closed.
    Mutatis mutandis, the same argument works for $\Down_D(K)$.
\end{proof}

%
Notice that in the proof we never used the transitivity of the relation.
Hence, this result holds also for the prolongational relation (see Section~\ref{sec:NW}).

A second important property is that every closed quasi-order can be defined through a collection of continuous functions as described below.
\begin{definition}
  \label{def:Lyapunov}
    A function $L\in C^0(X)$ is a {\em Lyapunov function} for a quasi-order $D$ if $x\Dto y$ implies $L(x)\geq L(y)$.    
    We denote by $\cL_D\subset C^0(X)$ the set of all Lyapunov functions of $D$.
\end{definition}  
Notice that constant functions are Lyapunov for every quasi-order and that, if $L_1,L_2\in\cL_D$, then $a_1 L_1+a_2 L_2\in\cL_D$ for all $a_1,a_2\geq0$.
Hence, the proposition below follows.
\begin{proposition}
    For any quasi-order $D$, the set $\cL_D$ is a non-empty closed convex cone.
\end{proposition}
\begin{example}
    The set of Lyapunov functions of the quasi-order $\cO_F$ coincides with the set of Lyapunov functions of the semi-flow $F$.
\end{example}
%
One can also run things backward and use collections of functions to define quasi-orders, as illustrated below.

Given a collection $\cS\subset C^0(X)$, we set
\beq
   \cD_\cS = \{(x,y) : L(x)\geq L(y)\text{ for all }L\in \cS\}.
\eeq
Notice that, if $C_\cS$ is the convex cone generated by $\cS$ (namely the smallest closed convex cone containing $\cS$), then 
$\cD_\cS=\cD_{C_\cS}$.

%
\begin{proposition}
    The set $\cD_\cS$ is a topologically closed quasi-order on $X$. Moreover, if $\cT\subset\cS$, then $\cD_\cS\subset\cD_\cT$.
\end{proposition}
We leave to the proof to the reader.
\begin{example}[The largest and smallest quasi-order.]
    $$
    \cD_\emptyset = X\times X.
    $$
    $$
    \cD_{C^0(X)} = \emptyset.
    $$
\end{example}

Auslander showed that, when $X$ is separable and locally compact (in particular, when $X$ is compact), for every closed quasi-order $D$, the set $\cL_D$ contains enough functions to fully characterize $D$,
as illustrated below.
The proof, by Auslander, that we include for reader's convenience, is a variant of the standard proof of Urhyson's Lemma.

%
%
\begin{lemmA}[Auslander, 1963~\cite{Aus63}, Theorem~4]
    \label{lem:thm4}
    Let $X$ be separable and locally compact and let $D$ be a closed quasi-order on $X$. 
    Then, if $x\not\Dto y$ (equivalently, if $(x,y)\not\in D$), there exists a function $L\in\cL_D$ such that $L(x)<L(y)$.
\end{lemmA}
\begin{proof}
{\black Assume first that $X$ is compact.}
We will show below that one can build an infinite countable family of closed sets $U_r$, labeled by rational (dyadic) numbers $r$ in $[0,1]$ such that:
\begin{enumerate}
    \item $x\in\cap_r U_r$;
    \item $y\not\in\cup_r U_r$;
    \item if $r<r'$, then $U_r\subset interior(U_{r'})$;
    \item $U_r$ is $D$-invariant, namely $\Down_D(U_r)=U_r$.
\end{enumerate}
With the sets $U_r$, one can define the function $L$ in the proposition claim as follows:
$$
L(z)=\inf\{r : z\in U_r\}.
$$
A standard argument shows that $L$, by point (3) above, is continuous. Moreover:
\begin{itemize}
    \item $L(x)=0$ by point (1) above;
    \item $L(y)=1$ by point (2) above;
    \item if $z\Dto z'$, then $L(z)\geq L(z')$ by points (3) and (4) above.
\end{itemize}
Now we will show how to build the sets $U_r$.
Let $A_0$ and $B$ disjoint {\black closed} neighborhoods of, respectively, $x$ and $y$ such that $\Down(A_0)\cap B=\emptyset$.
Such neighborhoods exist because $D$ is closed. Indeed, assume that there is no such neighborhoods $A_0$ and $B$.
This means that there is a sequence $y_i$ converging to $y$ and $x_i$ converging to $x$ such that $x_i\Dto y_i$ for all $i$. 
Since $D$ is closed, this implies that $x\Dto y$, against the hypotheses.

Set $U_0=\Down(A_0)$.
{\black Since $X$ is compact, $U_0$ is closed by Proposition~\ref{prop:Down(compact)}.}
Moreover, we also have that $x\in U_0$, $\Down(U_0)=U_0$ and $U_0\cap B=\emptyset$.

Now, let $A_{\frac{1}{2}}$ be a closed neighborhood of $U_0$ such that $U_0\subset interior(A_{\frac{1}{2}})$ and $\Down(A_{\frac{1}{2}})\cap B=\emptyset$.
As above, set $U_{\frac{1}{2}}=\Down(A_{\frac{1}{2}})$.
Then, {\black by the same argument above,} $U_{\frac{1}{2}}$ is closed, $\Down(U_{\frac{1}{2}})=U_{\frac{1}{2}}$, $U_{\frac{1}{2}}\cap B=\emptyset$ and $U_0\subset interior(U_{\frac{1}{2}})$.

By repeating this construction, we can build sets $U_{\frac{3}{4}},U_{\frac{7}{8}},\dots$.
To build $U_{\frac{1}{4}}$, we proceed as follows.
Let $B_{\frac{1}{2}}=closure(X\setminus U_{\frac{1}{2}})$.
Then $U_0\cap B_{\frac{1}{2}}=\emptyset$ and we can repeat the first step of our construction with $B_{\frac{1}{2}}$ in place of $B$.
We denote by $U_{\frac{1}{4}}$ the compact set we get this way. 
Notice that $U_{\frac{1}{4}}\subset interior(U_{\frac{1}{2}})$.
This way, we can build a set $U_r$ for each dyadic $r$ with the four properties listed above.

{\black
  Consider now, the general case, when $X$ is a non-compact, locally compact and separable space.
  Recall that every separable metric space is Lindelof and that every locally compact Lindelof space is
  hemicompact.
  Hence, there are compact sets $X_i\subset X$, $i\in\bN$, such that $X_i\subsetneq X_{i+1}$ for all $i$ and $X=\cup_{i\in\bN} X_i$.
  Let $i_0$ be any integer such that $x,y\in X_{i_0}$. Then, by the argument above, there is a Lyapunov function $L$ for $D$ in  $X_{i_0}$.
  In order to prove the theorem, it is enough to prove that we can extend $L$ to a Lyapunov function for $D$ on $X_{i_0+1}$. By induction,
  indeed, this allows to extend $L$ to the whole $X$.

  We can accomplish this by noticing that closed neighborhoods $A'_0$ of $x$ and $B'$ of $y$ such that $\Down_D(A'_0)\cap B'=\emptyset$
  exist in $X$ (notice, though, that in this case $A'_0$ is not necessarily compact and so $\Down_D(A'_0)$ is not necessarily closed).
  Hence, in $X_{i_0}$, we can elect to choose $A_0 = A'_0\cap X_{i_0}$ and $B=B'\cap X_{i_0}$. By repeating this argument at every step of the construction
  of the sets $U_r$, we build a Lyapunov function $L_{i_0}$ for $D$ in $X_{i_0}$.
  Then we can repeat the same construction in $X_{i_0+1}$ by setting $A_0 = A'_0\cap X_{i_0+1}$ and $B=B'\cap X_{i_0+1}$ and so on.
  This way, we end up with a Lyapunov function $L_{i_0+1}$ for $D$ in $X_{i_0+1}$ such that, by construction,
  $L_{i_0+1}|_{X_{i_0}} = L_{i_0}$.
}
\end{proof}
%
The following theorem, implicitly contained in~\cite{Aus63}, shows that closed quasi-orders are completely determined by their set of Lyapunov functions.
\begin{theoremX}[Auslander, 1963~\cite{Aus63}]
    \label{thm: Lyapunov characterization}
    Let $X$ be separable and locally-compact and let $D$ be a closed quasi-order on $X$.
    Then 
    $
    D=\cD_{\cL_D}.
    $
\end{theoremX}
\begin{proof}
    Let $L\in\cL_D$.
    By definition, $L(x)\geq L(y)$ for all pairs $(x,y)\in D$, so $D\subset\cD_{\cL_D}$. 
    By Lemma~A above, if $(x,y)\not\in D$, there is some $L\in\cL_D$ such that $L(x)<L(y)$.
    Hence $(x,y)\not\in\cD_{\cL_D}$, namely $\cD_{\cL_D}\subset D$.
\end{proof}
%

%
%
We conclude this discussion on general closed quasi-orders by showing that, for each relation, there is a  smallest closed quasi-order that contains it.
\begin{proposition}
    The set of all quasi-orders on $X$ is closed under intersection.
\end{proposition}
\begin{proof}
    The intersection of closed sets is closed.
    We leave to the reader to verify that 
    the intersection of two quasi-orders is a quasi-order.
\end{proof}
\begin{definition}
    We say that a relation $R$ extends a relation $R'$ if $R'\subset R$.
\end{definition}
\begin{corollary}
    \label{cor:smallest}
    Given any relation $R$ on $X$, there is a smallest closed quasi-order $D$ on $X$ extending $R$.   
\end{corollary}
\begin{proof}
    The set $D$ is the intersection of all closed quasi-orders that contain $R$.
    This set is non-empty because $X\times X$ is a closed quasi-order.
\end{proof}
%
\section{Closing $\cO_F$: the prolongational relation $\cP_F$ and the non-wandering set $NW_F$.}
\label{sec:NW}
The orbit space $\cO_F$ is a quasi-order but, in general, is not closed.
Hence, in general, $\cO_F$ is strictly contained in $\cD_{\cL_F}$, that is the smallest closed quasi-order extension of $\cO_F$.
\begin{example}
    Let $F$ be the flow of the ODE $x'=1$ on the real line.
    Then $\cO_F(x)=[x,+\infty)$ and $\cO_F=\{(x,y):y\geq x\}$.
    In this case, $\cO_F$ is a closed quasi-order.
\end{example}
\begin{example}
    \label{ex:OF open}
    Let $F$ be the flow of the ODE $x'=-\sin(\pi x)$ on the $[0,1]$ segment. 
    Then $\cO_F(0)=\{0\}$, $\cO_F(1)=\{1\}$ and $\cO_F(x)=(0,x]$ for each $x\in(0,1)$ (see Figure~\ref{fig:streams}, left).
    Hence, in this case $\cO_F$ equals the closed triangle with vertices $(0,0),(1,0),(1,1)$ minus the segments $(0,1]\times\{0\}$ and $\{1\}\times[0,1)$.
    In particular, $\cO_F$ is not closed.
\end{example}
Having a topologically closed relation is preferable -- for instance, closedness is essential to prove all general theorems we present in this article.
A natural option is considering the relation obtained by taking the topological closure of $\cO_F$.
\subsection{The prolongational relation}
\begin{definition}
  We denote by {\EM$\NW_F$} the relation $\overline{\cO_F}$, the topological closure of $\cO_F$, and by {\EM$\NWto$} the corresponding symbol. 
  We call this relation the {\em prolongational relation}.
\end{definition}
This relation was first introduced and studied by Auslander and Guerin~\cite{AG97} as a tool to study weak mixing in minimal flows.
By definition, a point belongs to $\NW_F$ if and only if there is a sequence $(x_n,y_n)\subset\cO_F$ such that $(x_n,y_n)\to(x,y)$ in the product topology.
This means that, arbitrarily close to $x$ and $y$, it is possible to find, respectively, points $x'$ and $y'$ such that $y'\in\cO_F(x')$.
The reader can verify the following elementary fact.
\begin{proposition}
    Let $d$ be any metric on $X$ compatible with the topology of $X$.
    Then $(x,y)\in\NW_F$ if and only if, for every $\eps>0$, there is an $\eps$-$\alpha\omega$chain from $x$ to $y$.
\end{proposition}
The set $\Down_{\NW_F}(x)$ coincides with the so-called {\em first prolongation} of $x$~\cite{Ura53,Aus63},
which justifies the name {\em prolongational relation} for $\cP_F$.
Notice that, since $\NW_F$ is not transitive, in general $\Down_{\NW_F}\left(\Down_{\NW_F}(x)\right)$ is strictly larger than $\Down_{\NW_F}(x)$.
\begin{example}
    Let $X=[0,2]$ and let $F$ be the flow of the ODE $$x'=|\sin(\pi x)|.$$
    Notice that $\Down_{\NW_F}(1) = [1,2]$ and, for each $x\in(0,1)$, 
    $$\Down_{\NW_F}(x)=\cO_F(x)\cup\{1\}.$$
    Hence, for each $x\in(0,1)$, 
    $$\Down_{\NW_F}\left(\Down_{\NW_F}(x)\right)=\cO_F(x)\cup[1,2].$$
\end{example}
The following technical elementary lemma, that will be used in several proofs of this section, shows that some weaker notion of transitivity does hold for $\NW_F$.
\begin{lemma}
    \label{lemma: weak transitivity}
    Assume that $x\NWto y$.
    Then $x\NWto\cO_F(y)\cup\Omega_F(y)$.
\end{lemma}
\begin{proof}
    We prove first that, if $x\NWto y$, then $x\NWto F^n(y)$ for every $n\geq0$.
    This is an immediate consequence of the continuity of $F^n$.
    Indeed, fix a $n\geq0$ and let $\eps>0$.
    Then we can find an $\eta>0$ such that $d(z,y)<\eta$ implies $d(F^n(z),F^n(y))<\eps$.
    If $\eta>\eps$, set $\eta=\eps$.
    Let now $c_1,\dots,F^k(c_1)$ be an $\eta$-$\alpha\omega$chain from $x$ to $y$, namely $d(c_1,x)<\eta$ and $d(F^k(c_1),y)<\eta$.
    Then $c_1,\dots,F^{k+n}(c_1)$ is an $\eps$-${\alpha\omega}$chain from $x$ to $F^n(y)$.
    Hence, $x\NWto\cO_F(y)$.

    Let now $z\in\Omega_F(y)$.
    Then, for every $\eps>0$, there is a $N>0$ such that $d(F^N(y),z)<\eps$.
    By continuity, we can also find $\eta>0$ such that $d(w,y)<\eta$ implies $d(F^N(w),F^N(y))<\eps$.
    Now, set $\eta=\eps$ if $\eta>\eps$.
    Then, by the triangular inequality, every $\eta$-$\alpha\omega$chain from $x$ to $y$ can be extended to a $2\eps$-$\alpha\omega$chain from $x$ to $z$.
    Hence, $x\NWto\Omega_F(y)$.
\end{proof}
\blue
\begin{lemma}
    \label{lemma: infinite chains}
    Assume that $x\NWto y$ and $y\not\in\cO_F(x)$. 
    Then, for every $\eps>0$, there is an $N_\eps>0$, with $N_\eps\to\infty$ as $\eps\to0$, such that every $\eps$-$\alpha\omega$chain from $x$ to $y$ has length larger than $N_\eps$.
\end{lemma}
\begin{proof}
    Suppose, by contradiction, that there is an integer $N>0$ such that, for every $\eps>0$, we can find an $\eps$-$\alpha\omega$chain from $x$ to $y$ whose length is not larger than $N$.
    Take a sequence $\eps_n\to0$ and corresponding $\eps_n$-$\alpha\omega$chains $(z_1^{(n)},\dots,z_{k_n}^{(n)})$ from $x$ to $y$.
    There is at least one integer $m$ with $1\leq m\leq N$ such that infinitely many of these chains have length $m$.
    Passing to a suitable subsequence $n_j$, we have a sequence of $\eps_{n_j}$-$\alpha\omega$chains of length $m$ from $x$ to $y$.
    Therefore, as $\eps_{n_j}\to0$, $z_1^{(n_j)}\to x$ and $z_m^{(n_j)}\to F^{m-1}(x)$.
    Hence, it must be that $y=F^{m-1}(x)$, against the theorem's assumption that $y$ is not in the orbit of $x$.   
\end{proof}
\black
\begin{proposition}[\EM Properties of $\NW_F$]
    \label{prop:Down}
    For any semi-flow $F$, $\NW_F$ is forward-invariant under the induced action of $F$ on $X\times X$ given by $F^t(x,y)=(F^t(x),F^t(y))$.
    Moreover, for every $x$, the following holds:
    \begin{enumerate}
        \item $\Down_{\NW_F}(x)$ is forward-invariant under $F$;
        \item $\Down_{\NW_F}(x)
        \;\supset\; 
        \cO_F(x)\cup\Omega_F(x)$;
        \item if $\Omega_F(x)\neq\emptyset$, 
        $\cO_F(x)\cup\Down_{\NW_F}(\Omega_F(x))
        \supset
        \Down_{\NW_F}(x)$;
        \item $\cO_F(x)\cup\Down_{\NW_F}(y)
         \;\supset\; 
        \Down_{\NW_F}(x)
        $ for all $y\in\cO_F(x)$;
        \item if $F^t$ is open for all $t\geq0$ (for instance if $F$ is a flow or an open map), then\\
        $\Down_{\NW_F}(x)
        = \cO_F(x)\cup\Down_{\NW_F}(y)$ for all $y\in\cO_F(x)$.
    \end{enumerate}
\end{proposition}
\begin{proof}
    (1) Let $y\in\Down_{\NW_F}(x)$. 
    We need to show that $F(y)\in\Down_{\NW_F}(x)$.
    This follows immediately from Lemma~\ref{lemma: weak transitivity}.

    (2) This is an immediate consequence of the facts that $\cO_F\subset\NW_F$ and that $\NW_F$ is closed.

    (3) Let $x\NWto z$ and $y\in\Omega_F(x)$. 
    \blue 
    Assume that $z\not\in\cO_F(x)$ and take a sequence $\eps_n\to0^+$ and a corresponding sequence of $\eps_n$-$\alpha\omega$chains $(w_1^{(n)},\dots,w_{N_n}^{(n)})$ from $x$ to $z$.
    Then, by Lemma~\ref{lemma: infinite chains}, the length of these chains diverges with $n$.
    Now, fix an $\eta>0$ and $y\in\Omega_F(x)$.
    By definition of limit set, there is a $k>0$ such that $d(F^k(x),y)<\eta$.
    By the continuity of $F^k$, there is a $\delta>0$ such that $d(x,\zeta)<\delta$ implies that $d(F^k(x),F^k(\zeta))<\eta$.
    For almost all $n$, $\eps_n<\min\{\eta,\delta\}$ and $N_n>k+1$.
    Hence, given the $\eps_n$-$\alpha\omega$chain $(w_1^{(n)},\dots,w_{N_n}^{(n)})$, the sequence $(w_{k+1}^{(n)}=F^{k}(w_1^{(n)}),\dots,w_{N_n}^{(n)})$ is a $2\eta$-$\alpha\omega$chain from $y$ to $z$. 
    Hence, $y\NWto z$. 
    Since this holds for al $y\in\Omega_F(x)$, we get that $\Omega_F(x)\NWto z$.
    
    \black   
    (4) This can be proven by the same argument of point (3).  

    (5) By (4), it is enough to prove the inverse inclusion.
    Let $y=F^n(x)$, $z\in\Down(y)$ and fix an $\eps>0$.
    By continuity and since $F^n$ is open, there is an $\eta>0$ such that, for every $y'$ such that $d(y,y')<\eta$, there is an $x'$ with $d(x,x')<\eps$ and $F(x')=y'$.
    We can assume without loss of generality that $\eta\leq \eps$.
    Hence, every $\eta$-$\alpha\omega$chain from $y$ to $z$ can be extended to an $\eps$-$\alpha\omega$chain from $x$ to $z$.

    Finally, we prove the forward-invariance of $\NW_F$.
    We need to prove that $x\NWto y$ implies $F^t(x)\NWto F^t(y)$ for every $t\geq0$.
    Notice first that $x\Fto y$ implies trivially that $F^t(x)\Fto F^t(y)$.
    Suppose now that $y$ is not on the orbit of $x$.
    Then, by point (4) above, $F^t(x)\NWto y$ and furthermore, by Lemma~\ref{lemma: weak transitivity}, $F^t(x)\NWto F^s(y)$ for every $s\geq0$.
    Hence, in particular, $F^t(x)\NWto F^t(y)$.
\end{proof}
The following examples show that there is no way to improve the inclusions in the proposition above.
%
\begin{example}[$\Down_{\NW_F}(x)$ can be larger than $\cO_F(x)\cup\Omega_F(x)$]
    \label{ex: saddle point 1}
    Consider a flow $F$ on $\R^2$ with a single recurrent point $f$, which is a saddle point of the type shown in Figure~\ref{fig:homoclinic}(right).
    We denote by $C$ the ``critical loop'' equal to the union of $f$ with the homoclinic trajectory $h$. 
    In this example we assume that, outside of the portion displayed, each non-critical trajectory coincides with some horizontal line.
    Hence, if $x$ lies on a non-critical open trajectory, $\Omega_F(x)=\emptyset$.
    In case of the closed trajectories, all of which are inside $C$, $\cO_F(x)=\Om_F(x)$ for each $x$ on them.
    Moreover, the reader can verify that there is no point $\NW_F$-downstream from a point $x$ on a non-critical trajectory other than the points on its trajectory, namely $\Down_{\NW_F}(x)=\cO_F(x)$ for all such points.
    The situation is different for points on some critical trajectory.
    Notice that the saddle consists of the fixed point $f$, its stable manifolds $g_s$ and $h$ and its unstable manifold $g_u$ and $h$.
    All points on $g_s$ asymptote to the fixed point $f$, so that $\Down_{\NW_F}(x)\supset\cO_F(x)\cup\{f\}$ for all $x\in g_s$. 
    Moreover, by hopping on trajectories above $g_s$, from any $x\in g_s$ one can reach every other point on $h$ and $g_u$, so that
    $$
    \Down_{\NW_F}(x) = \cO_F(x)\cup\Omega_F(x)\cup h\cup g_u\;\text{ for all }\;x\in g_s,
    $$
    where all these sets are pairwise disjoint.
    Similarly, the reader can verify that $\cO_F(f)=\Omega_F(f)$ and
    $$
    \Down_{\NW_F}(f) = \cO_F(f)\cup h\cup g_u
    $$
    and that
    $$
    \Down_{\NW_F}(x) = \cO_F(x)\cup\Omega_F(x)\cup g_u\;\text{ for all }\;x\in h,
    $$
    where all unions above are union of pairwise disjoint sets.

    In the saddle displayed in Figure~\ref{fig:homoclinic}(left), the only difference is that the trajectories inside $C$ other than the fixed point are not closed but rather spiral outwards.
    Consequently, for all non-fixed points inside $C$, $\Omega_F(x)=h\cup\{f\}$ and 
    $$
    \Down_{\NW_F}(x) = \cO_F(x)\cup\Omega_F(x).
    $$
    Finally, $\cO_F(c)=\Omega_F(c)=\{c\}$ while
    $$
    \Down_{\NW_F}(c) = \cO_F(c)\cup D,
    $$
    where $D$ is the closed disc having $C$ as its boundary.
\end{example}
\begin{figure}
    \begin{tikzpicture}[scale=1.275]
      \draw[myred,yshift=0 em, thick,ar5] (0,0) -- (2,0) [point={thick, fill=green, point name=A}] -- (4,0);
      \draw[mypurple,yshift=0 em, thin,ar4] (0,-0.15) -- (4,-0.15);
      \draw[mypurple,yshift=0 em, thin,ar4] (0,-0.3) -- (4,-0.3);
      \node  at ([shift={(90:0.18)}]A) {$\scriptstyle f$};
      \node  at (-.15,0) {$\scriptstyle g_s$};
      \node  at (4.15,0) {$\scriptstyle g_u$};
      \node at (2.3875,0.252) {$\scriptstyle h$};
      \draw[mypurple,thin,->-]
      (0,0.15) .. controls (2.5,0.15) and (1, 0.3) .. (1.08, 1.55) .. controls (1.2,2.15) and (1.7,2.2)
      .. (2, 2.22) ..
      controls (2.3,2.2) and (2.8,2.15) .. (2.92,1.55) .. controls (3,0.3) and (1.5,0.15) .. (4,0.15);

      \draw[mypurple,thin,ar1]
      (0,0.3) .. controls (2.3,0.3) and (0.7, 0.45) .. (1, 1.7) .. controls (1.2,2.3) and (1.7,2.35)
      .. (2, 2.35) ..
      controls (2.3,2.35) and (2.8,2.3) .. (3,1.7) .. controls (3.3,0.45) and (1.7,0.3) .. (4,0.3);
      \draw[mypurple,thin,ar2]
      (0,0.45) .. controls (2.,0.45) and (0.4, 0.6) .. (1, 2) .. controls (1.2,2.45) and (1.7,2.5)
      .. (2, 2.5) ..
      controls (2.3,2.5) and (2.8,2.45) .. (3,2) .. controls (3.6,0.6) and (2,0.45) .. (4,0.45);

      \Loop[dist = 4cm, dir = NO, style={myred,thick,-<-}](A.north)

      \draw [myblue,domain=0:37,variable=\t,smooth,samples=200,xshift=56.5,yshift=33,-stealth]
      plot ( {3 * \t *\t/ pi) * cos(\t r)/2000}, {-3.8 * \t *\t/ pi) * sin(\t r)/2000} );
    
    \end{tikzpicture}
    \hskip -1cm
    \begin{tikzpicture}[scale=1.275]
      \draw[myred,yshift=0 em, thick,ar5] (0,0) -- (2,0) [point={thick, fill=green, point name=A}] -- (4,0);
      \draw[mypurple,yshift=0 em, thin,ar4] (0,-0.15) -- (4,-0.15);
      \draw[mypurple,yshift=0 em, thin,ar4] (0,-0.3) -- (4,-0.3);
      \node  at ([shift={(90:0.18)}]A) {$\scriptstyle f$};
      \node  at (-.15,0) {$\scriptstyle g_s$};
      \node  at (4.15,0) {$\scriptstyle g_u$};
      \node at (2.3875,0.252) {$\scriptstyle h$};
      \node  at (2,1.1) {$\scriptstyle c$};
      \draw[mypurple,thin,->-]
      (0,0.15) .. controls (2.5,0.15) and (1, 0.3) .. (1.08, 1.55) .. controls (1.2,2.15) and (1.7,2.2)
      .. (2, 2.22) ..
      controls (2.3,2.2) and (2.8,2.15) .. (2.92,1.55) .. controls (3,0.3) and (1.5,0.15) .. (4,0.15);

      \draw[mypurple,thin,ar1]
      (0,0.3) .. controls (2.3,0.3) and (0.7, 0.45) .. (1, 1.7) .. controls (1.2,2.3) and (1.7,2.35)
      .. (2, 2.35) ..
      controls (2.3,2.35) and (2.8,2.3) .. (3,1.7) .. controls (3.3,0.45) and (1.7,0.3) .. (4,0.3);
      \draw[mypurple,thin,ar2]
      (0,0.45) .. controls (2.,0.45) and (0.4, 0.6) .. (1, 2) .. controls (1.2,2.45) and (1.7,2.5)
      .. (2, 2.5) ..
      controls (2.3,2.5) and (2.8,2.45) .. (3,2) .. controls (3.6,0.6) and (2,0.45) .. (4,0.45);

      \Loop[dist = 4cm, dir = NO, style={myred,thick,-<-}](A.north)
      
      \draw[myblue,ar3] plot[smooth cycle,tension=0.7] coordinates{ (2,0.3) (1.6,0.7) (1.3,1.4) (1.5,1.85) (2,2) (2.5,1.85) (2.7,1.4) (2.4,0.7)  };
      \draw[myblue,ar4] plot[smooth cycle,tension=0.8] coordinates{ (2,0.6) (1.5,1.3) (2,1.75) (2.5,1.3)  };
      \draw[myblue,ar4] plot[smooth cycle,tension=0.8] coordinates{ (2,0.9) (1.7,1.3)(2,1.5)(2.3,1.3)  };
      \filldraw [myblue] (2,1.2) circle (0.5pt);
      
    \end{tikzpicture}  
    \vskip .3cm
  \caption{
    {\em Two topologically inequivalent homoclinic saddle points on a surface.}
    The prolongational relation for the corresponding flows is discussed in Example~\ref{ex: saddle point 1}.    
  }
    \label{fig:homoclinic}
\end{figure}
\begin{example}[$\cO_F(x)\cup\Down_{\NW_F}(\Omega_F(x))$ can be larger than $\Down_{\NW_F}(x)$]
    Let $X=[0,2]$ and let $F$ be the flow of the ODE $x'=|\sin(\pi x)|.$
    Then $\Down_{\NW_F}(1)=[1,2]$ while, for any $x\in(0,1)$, $\Omega_F(x)=\{1\}$ and $Down_{\NW_F}(x) = \cO_F(x)\cup\{1\}$.
    Hence $\cO_F(x)\cup\Down_{\NW_F}(\Omega_F(x))=\cO_F(x)\cup[1,2]$ is strictly larger than $Down_{\NW_F}(x)$.
\end{example}
\begin{example}[$\cO_F(x)\cup\Down_{\NW_F}(y)$ can be larger than $\Down_{\NW_F}(x)$ for $y\in\cO_F(x)$]
    Let $F:[0,2]\to[0,2]$ be the continuous function
    $$
    F(x)=
    \begin{cases}
        1,&x\leq1\\
        \sqrt{x-1}+1,&x\geq1\\
    \end{cases}
    $$
    Then $\Down_{\NW_F}(1)=[1,2]$ while, for every $x\in(0,1)$, $$\Down_{\NW_F}(x)=\cO_F(x)=\{x,1\}.$$
    Hence, for each such $x$, $1\in\cO_F(x)$ and $\cO_F(x)\cup\Down_{\NW_F}(1)$ is strictly larger than $\Down_{\NW_F}(x)$.
    
    
\end{example}
%

%
Although $\NW_F$ is not, in general, transitive, the two examples below show that sometimes it is, even in non-trivial cases. 
When this happens, then $\NW_F$ is necessarily the smallest closed quasi-order containing $\cO_F$ and so it equals $\cD_{\cL_F}$. 
%
\begin{example}
    Let $F$ be the flow in Example~\ref{ex:OF open}.
    Then $\NW_F\subset\R^2$ is the closed triangle with vertices $(0,0)$, $(1,0)$, $(1,1)$.
    A direct check shows that this relation is transitive.
\end{example}
\begin{example}
    \label{ex:lm}
    Consider the restriction of the logistic map $\ell_\mu$, $\mu\in(1,3)$, to its image $X=[0,\ell_\mu(1/2)]$.
    The dynamics of $\ell_\mu$ in this range of $\mu$ is very simple: there are two fixed points, 0 and $p_\mu=1-1/\mu$, and every point but 0 and 1 converges to $p_\mu$.
    Hence, $\Down_{\NW_{\ell_\mu}}(x)=\cO_{\ell_\mu}(x)\cup\{p_\mu\}$ for all $x\neq0,1$.
    Moreover, $\Down_{\NW_{\ell_\mu}}(0)=[0,\ell_\mu(1/2)]$, since every point in the image of $\ell_\mu$ has a backward trajectory asymptoting to $0$.
    Then $\NW_{\ell_\mu}$ is transitive because 
    is transitive on its subset $\cO_{\ell_\mu}$ and it contains all possible pairs $(0,x)$, $x\leq\ell_\mu(1/2)$, and $(x,p_\mu)$, $x\in[0,1)$.
    Notice that $\NW_{\ell_\mu}$ is not fully transitive on $X=[0,1]$ because $1\NWto0$ and, by transitivity, 1 is upstream from all points $x\in X$, but $(1,x)\in\NW_{\ell_\mu}$ only for $x=0$.
\end{example}
%

%
The lack of transitivity of $\NW_F$ comes from the fact that, in general, the concatenation of two $\eps$-$\alpha\omega$chains is not an $\eps$-$\alpha\omega$chain.
Linkable pairs, defined below, are precisely those ones for which this happens.
\begin{definition}
    \label{def:linkable}
    Let $C$ be an $\eps$-$\alpha\omega$chain from $x$ to $y$ with points $$c_1,F(c_1),\dots,F^k(c_1)$$ and let $C'$ be an $\eps$-$\alpha\omega$chain from $y$ to $z$ with points $$c'_1,F(c'_1),\dots,F^{k'}(c'_1).$$
    We say that $C$ and $C'$ are {\em linkable} if $c'_1 = F^k(c_1)$.
    
    If $x\NWto y$ and $y\NWto z$ and, for every $\eps>0$, there is an $\eps$-$\alpha\omega$chain from $x$ to $y$ and one from $y$ to $z$ that are linkable, we say that $\NW_F$ is {\em transitive} on the pairs $(x,y)$ and $(y,z)$.
\end{definition}
In the definition below we define non-wandering points, first introduced in literature by Birkhoff~\cite{Bir27}, using the prolongational relation.
\begin{definition}
    \label{def:NW}
    We say that a point $x$ is {\em non-wandering} for $F$ if 
    either $x$ is fixed or there exists a $y\neq x$ such that:
    \begin{enumerate}
        \item $x\NWto y$ and $y\NWto x$;
        \item $\NW_F$ is transitive on the pairs $(x,y)$ and $(y,x)$. 
    \end{enumerate}
    In this case, we say that $x$ and $y$ are {\em $\NW_F$-equivalent} and write that $x\NWeq y$.
    We denote by {\EM$NW_F$} the set of all non-wandering points of $F$.
    We say that a set $M\subset NW_F$ is $\NW_F$-equivalent if all points of $M$ are $\NW_F$-equivalent to each other.
\end{definition}
\begin{proposition}
    \label{prop:NW_F in Q}
    Let $F$ be a semi-flow with compact dynamics and let $Q$ be a global trapping region for $F$.
    Then $NW_F\subset Q$.
\end{proposition}
\begin{proof}
  Since $F$ has compact dynamics, $\Omega_F(x)\neq\emptyset$ {\black for every $x$} and so, by Proposition~\ref{prop:Down}(3),
  all points downstream from $x$ either belong to $\cO_F(x)$ or are downstream from $\Omega_F(x)$.
  Moreover, $\Omega_F(x)\subset Q$.
  Assume that there are $x\in Q$ and $y\not\in Q$ such that $x\NWto y$.
  Then, for every $\eps>0$, there is a trajectory starting from some $z$, with $d(x,z)<\eps$, and a $N>0$ such that $d(y,F^N(z))<\eps$.
  Since $y\not\in\cO_F(x)$, because $Q$ is a trapping region for $F$, then $N\to\infty$ as $\eps\to0$.
  This leads to a contradiction. 
  Indeed, let $\eta=d(y,Q)$ and set $Q_\delta=\{w:d(w,Q)<\delta\}$.
  Then, by hypothesis, there is a $T>0$ such that $F^t(Q_\eta)\subset Q_{\eta/2}$ for all $t\geq T$.
  In particular, for small enough $\eps$, $d(y,F^t(z))>\eps$ for all $t\geq T$.
  So, it is impossible that, at the same time, for diverging values of $N$ we have that $d(y,F^N(z))<\eps$.
\end{proof}

The following proposition shows, in particular, that the definition above of non-wandering point is equivalent to the standard definition, introduced by Birkohff in 1927.
\begin{proposition}
    Given a point $x$, the following are equivalent:
    \begin{enumerate}
        \item $x$ is non-wandering;
        \item for each $\eps>0$, there is an $\eps$-$\alpha\omega$chain from $x$ to itself;
        \item for every neighborhood $U$ of $x$, there is a $t>0$ such that $F^t(U)\cap U\neq\emptyset$.
    \end{enumerate}
\end{proposition}
\begin{proof}
    (a) Let us show that $1\implies 2$.
    Suppose first that $x$ is fixed. 
    Then, by continuity of $F^1$, for every $\eps>0$, there is a $\delta\in(0,\eps)$ such that $d(x,z)<\delta$ implies that $d(x,F^1(z))<\eps$.
    Hence, the sequence $z,F^1(z)$ is an $\eps$-$\alpha\omega$chain from $x$ to itself.
    Suppose now that $x$ is not fixed.
    Then, there is a $y\neq x$ such that, for every $\eps>0$, there is an $\eps$-$\alpha\omega$chain from $x$ to $y$ and one from $y$ to $x$ and these chains are linkable. 
    The concatenation of these two chains is an $\eps$-$\alpha\omega$chain from $x$ to itself.

    (b) Let us show that $2\implies 3$.
    Let $U$ be a neighborhood of $x$ and let $\eps>0$ so small that every point within $\eps>0$ from $x$ is in $U$. 
    By hypothesis, there is an $\eps$-$\alpha\omega$chain from $x$ to itself, namely there is $y\in U$ such that, for some $t>0$, $F^t(x)\in U$.
    Hence, $U\cap F^t(U)\neq\emptyset$.

    (c) The fact that $3\implies 2$ is essentially a tautology. 
    We will now show that $2\implies 1$.
    By continuity of $F^1$, for every $\eps>0$, there is a $\delta\in(0,\eps)$ such that $d(x,z)<\delta$ implies that $d(F^1(x),F^1(z))<\eps$.
    Hence, for any $\delta$-$\alpha\omega$chain $z,F^1(z),F^2(z),\dots F^N(z)$ from $x$ to itself, the sequence $z,F^1(z)$ is an $\eps$-$\alpha\omega$chain from $x$ to $F^1(x)$ and the sequence $F^1(z),F^2(z),\dots F^N(z)$ is an $\eps$-$\alpha\omega$chain from $F^1(x)$ to $x$.
%
\end{proof}
The fact that $x\NWto y$ and $y\NWto z$ are not enough to grant $x\NWto z$ is illustrated in the following elementary example.
%
\begin{example}
    Consider a circle with a fixed point $N$ and every other point that rotates clockwise and asymptotes to $N$.
    Then, for every $x\neq N$, we have that $x\NWto N$, since $x$ asymptotes forward to $N$, and $N\NWto x$, because $x$ asymptotes backward to $N$.
    On the other side, every $\eps$-$\alpha\omega$chain from $x$ cannot move beyond $N$ by more than $\eps$.
    Hence, for $\eps$ small enough, no such chain can reach $x$.
\end{example}

A particularly important type of $\NW_F$-equivalent sets is limit sets.
\begin{proposition}
    \label{prop:Om(x)}
    For every $x$, the set $\Omega_F(x)$ is $\NW_F$-equivalent.
\end{proposition}
\begin{proof}
  For every $y\in\Omega_F(x)$, the orbit $\cO_F(x)$ can be used to build, for every $\eps>0$, $\eps$-$\alpha\omega$chains from $y$
  \black
  to any other point in $y\in\Omega_F(x)$. Hence, $y\in NW_F$ and $\Omega_F(x)$ is $\NW_F$-equivalent.
\end{proof}
Recall that recurrent points and $\omega$-limit points are non-wandering (e.g. compare Proposition~\ref{prop:recurrent} with point (2) of the proposition above) but the converse is not true, as next examples illustrates.
%
%
\begin{example}[Not all non-wandering points are limit points]
    Consider the two flows depicted in Figure~\ref{fig:homoclinic}.
    In the left picture, the limits sets are the following: 
    the fixed points $f$ and $c$ and the critical loop $C=h\cup\{f\}$.
    As pointed out in the previous example, $NW_F=\{f,c\}\cup h$ and so every non-wandering point is a limit point.
    On the contrary, in the right picture, while $NW_F$ is the same, the limit sets are the following: the fixed points $f$ and $c$ and each periodic orbit inside $C$.
    Hence, each point of $h$ is non-wandering but not a limit point.
\end{example}

\subsection{Nodes and Spectral Decomposition of \EM$NW_F$.}
The $\NW_F$-equivalence induces a decomposition of $NW_F$ as follows.
\begin{definition}
    We call {\em nodes} of $NW_F$ \blue (or, equivalently, of $\cP_F$) \black the maximal $\NW_F$-equivalent subsets of $NW_F$.    
\end{definition}
\begin{example}
    \label{ex:8}
    Let $F$ be the flow displayed in Figure~\ref{fig:homoclinic}(left).
    Then $NW_F=\{f\}\cup h\cup\{c\}$ and there are two nodes: the critical loop $C=\{f\}\cup h$ and the fixed point $c$.
    Notice that, unlike in the logistic map case above,  node $C$ contains no dense orbit but rather is a limit set (of any non-fixed point inside $C$).
\end{example}
\begin{example}
    Let $G$ be the flow displayed in Figure~\ref{fig:homoclinic}(right).
    Then $NW_G=\{f,c\}\cup h\cup D$, where $D$ is the open disk that has $C$ as boundary, and there are
    uncountably many nodes: the critical loop $C=\{f\}\cup h$ and each periodic trajectory of $G$ (whose union is equal to $D$).
    Recall that $C$ contains no dense orbit and, in this case, it is not even a limit set.
\end{example}
\begin{example}
    Let $H$ be the restriction of either $F$ or $G$ above to their closed invariant set $C$. Then $NW_H=\{f\}$, so there is a single node. 
\end{example}
Notice that the example above shows also the well-known fact that the non-wandering set is not natural with respect to restrictions.
\blue
\begin{definition}
    We say that a semiflow $F$ with compact dynamics has {\em strong compact dynamics} if it has a compact global trapping region that is a neighborhood of $NW_F$.
\end{definition}
\black
\begin{proposition}
    \label{prop:NW}
    For any semi-flow $F$, 
    $NW_F$ satisfies the following:
    \begin{enumerate}
        \item for each $x\in NW_F$, $\cO_F(x)$ is a $\NW_F$-equivalent set; 
        \item $NW_F$ and each of its nodes are forward-invariant under $F$;
        \item $NW_F$ and each of its nodes are closed.
    \end{enumerate}
    Assume now that $F$ has compact dynamics. 
    Then $NW_F$ satisfies the following further properties:
    \begin{enumerate}
    \setcounter{enumi}{3}
        \item each node of $NW_F$ contains at least a recurrent point $x$ (and so the whole set $\Omega_F(x)$);
        \item each node of $NW_F$ contains a limit set and each limit set is contained in at least a node of $NW_F$;
    \end{enumerate}
    \blue
    Finally, if $F$ has strong compact dynamics and every $F^t$ is an open map,
    then $NW_F$ satisfies the following further property:
    \begin{enumerate}
    \setcounter{enumi}{5}
        \item \blue $NW_F$ and its nodes are $F$-invariant.
    \end{enumerate}
\end{proposition}
\begin{proof}
    (1) 
    Since $x\in NW_F$, for every $\eps>0$ there is an $\eps$-$\alpha\omega$chain from $x$ to itself. 
    Let $y=F^k(x)$ for some $k\geq0$ and let $\eps>0$.
    Since $F^k$ is continuous, there are $\eta>0$ such that $d(x,c)<\eta$ implies $d(y,F^k(c))<\eps$ and $\eta'>0$ such that $d(x,c)<\eta'$ implies $d(y,F^k(c))<\eta$.
    Let $C$ be an $\eta'$-$\alpha\omega$chain from $x$ to itself with points $c,F(c),\dots,F^r(c)$.
    We can assume without loss of generality that $\eps\geq\eta\geq\eta'$.
    Then the chain $C'$ with points $F^k(c),\dots,F^r(c)$ is an $\eta$-$\alpha\omega$chain from $y$ to $x$ and $C''$ with points $F^k(c),\dots,F^{r+k}(c)$ is an $\eps$-$\alpha\omega$chain from $y$ to itself.
    Hence, $y\in NW_F$.
    Moreover, our argument shows that $y\NWto x$, so $x\NWeq y$.

    (2) This is an immediate consequence of (1).


 
    
    
    (3) Let $x_1,x_2,\dots$ be a sequence of points in $NW_F$ converging to some $\bar x\in X$.
    If $\bar x$ is fixed, the claim is trivially true.
    If $\bar x$ is not fixed, then we can assume without loss of generality that the sequence $x_i$ has no fixed points since,
    by continuity, the limit of fixed points would be fixed.
    Hence, for each $i$, $F^1(x_i)\neq x_i$ but, by Proposition~\ref{prop:Om(x)}, $F^1(x_i)\NWeq x_i$.
    Then, by the closure of $\NW_F$ and continuity of $F$, $F^1(\bar x)\NWeq\bar x$.
    Hence, $\bar x\in NW_F$.
    Now, suppose that all of the $x_i$ lie in a node $N$.
    Then $x_n\NWeq x_m$ for all $n,m$ and so, by the closure of $\NW_F$, $x_n\NWeq\bar x$ for all $n$, so $\bar x\in N$.
    Hence, both $NW_F$ and its nodes are closed.
    
    (4) This is an immediate consequence of point (2) together with Proposition~\ref{prop:recurrent pt}. 

    (5) By (4), each node contains at least a limit set. 
    By Proposition~\ref{prop:Om(x)}, every limit set is $\NW_F$-equivalent and so it is contained in a node.

    (6)
    \blue
    Let $N$ be a node and $x\in N$.
    Let $C_i=(c_i,F(c_i),\dots,F^{n_i}(c_i))$, $i=1,2,\dots$, be an $\eps_i$-$\alpha\omega$chain from $x$ to itself of length $n_i+1$, with $\eps_i\to 0^+$.
    Set $y_i = F^{n_i-1}(c_i)$.
    Then 
    $d(x,F(y_i))<\eps_i$, namely
    $F(y_i)\to x$ as $i\to\infty$.
    Since $Q$ is a forward-invariant neighborhood of $NW_F$, for $i$ large enough $C_i$ lies entirely in $Q$, which is compact, and so  we can assume without loss of generality that $y_i\to\bar y$.
    By continuity, $F(\bar y)=x$.
    
    We now prove that $\bar y\in NW_F$.
    First of all, notice that, since $c_i\to x$ and $F^{n_i-1}(c_i)\to\bar y$, we have that $x\NWto\bar y$.
    Moreover, since $F^1$ is an open map, for every $\eps>0$ there is a $\delta>0$ such that, for every $z\in N_\delta(\{x\})$, there is a $w\in N_\eps(\bar y)$ such that $F(w)=z$.
    Notice that we can assume without loss of generality that $\delta\leq\eps$.
    Now, fix an $\eps>0$, let $\delta>0$ be the $\delta$ mentioned above and let $(z_0,\dots,F^n(z_0))$ be a $\delta$-$\alpha\omega$chain from $x$ to $\bar y$.
    Then there is $w_0\in N_\eps(\{\bar y\})$ such that $F(w_0)=z_0$ and $(w_0,z_0,\dots,F^n(z_0))$ is an $\eps$-$\alpha\omega$chain from $\bar y$ to itself.
    Hence, $\bar y\in NW_F$.
    Moreover, it follows from our argument that $x$ and $\bar y$ are $\NW_F$-equivalent, so $\bar y\in N$.
    Hence, each node is invariant and so is the whole non-wandering set.
\end{proof}
\blue
Next two examples show that case 6 of the proposition above is sharp.
The first is an example of a semiflow $F$ on a two-dimensional non compact topological manifold with a non-invariant non-wandering set.
\begin{figure}
    \centering
    \includegraphics[width=0.5\linewidth]{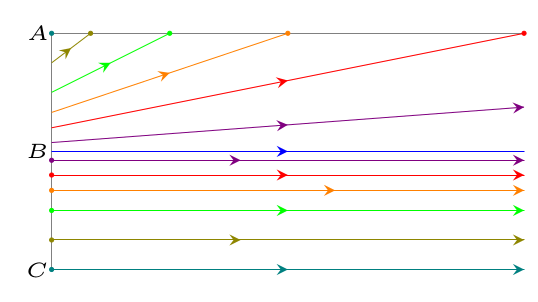}
    \caption{\blue{\bf A semiflow with a non-invariant non-wandering set.}
    The picture shows several orbits of a semiflow $F$ on the non compact space $X$ equal to the unbounded strip shown in figure where we identify points on the horizontal half-line $h$ passing through $A$ with points on the vertical segment $BC$ so that $A$ is identified with $C$ and points going to infinity on $h$ are identified with points going to $B$ on $BC$.
    Several orbits of $F$ are shown, each one painted in a different color.
    As the picture suggests, $\cap_{t\geq0}F^t(X)=\emptyset$, i.e. no subset of $X$ is $F$-invariant.
    The non-wandering set coincides with the blue orbit.
    }
    \label{fig: nw1}
\end{figure}
\begin{example}
    Let $X$ be the quotient of the rectangle $R=[1,\infty)\times [-1,1]$ under the identification of the halfline $[1,\infty)\times\{1\}$ with the segment $\{1\}\times(0,-1]$ given by $(x,1)\sim(1,-1/x)$.
    We define on $X$ a continuous-time semiflow 
    as follows.
    For each $p\in X$, there is a point $q=(1,y)$, $y\in[0,1]$, such that $p\in\cO_F(q)$.
    For $y\in[0,1]$,
    $$
    F^t(0,y) = 
    \begin{cases}
    (1+t)(1,y),& t\leq 1/y-1;\\
    (1+t,-y),& t > 1/y.
    \end{cases}
    $$
    The action of $F$ on any other point of $X$ can be obtained from the formula above using the fact that
    $F^s(p) = F^s(F^t(q)) = F^{s+t}(q)$.
    Several orbits are illustrated in Figure~\ref{fig: nw1}, where distinct orbits are painted in different colors.
    The reader can verify that 
    $$
    NW_F=[1,\infty)\times\{0\}
    $$ 
    and that 
    $$
    F^t(NW_F)=[1+t,\infty)\times\{0\}\neq NW_F
    $$
    for every $t>0$.
    Consider now the time-1 map $f=F^1$.
    Then $NW_f=\{1,2,\dots\}\times\{0\}$ and $f(NW_f)=\{2,3,\dots\}\times\{0\}$.
    The map $f$ is open but, thanks to the fact that $X$ is not compact, not necessarily each point of $NW_f$ has a preimage. 
    In this concrete case, the point $(1,0)$ has no preimage.
\end{example}
The example above can be slightly modified to a semiflow on a compact space with non invariant non-wandering set. 
\begin{figure}
    \centering
    \includegraphics[width=0.5\linewidth]{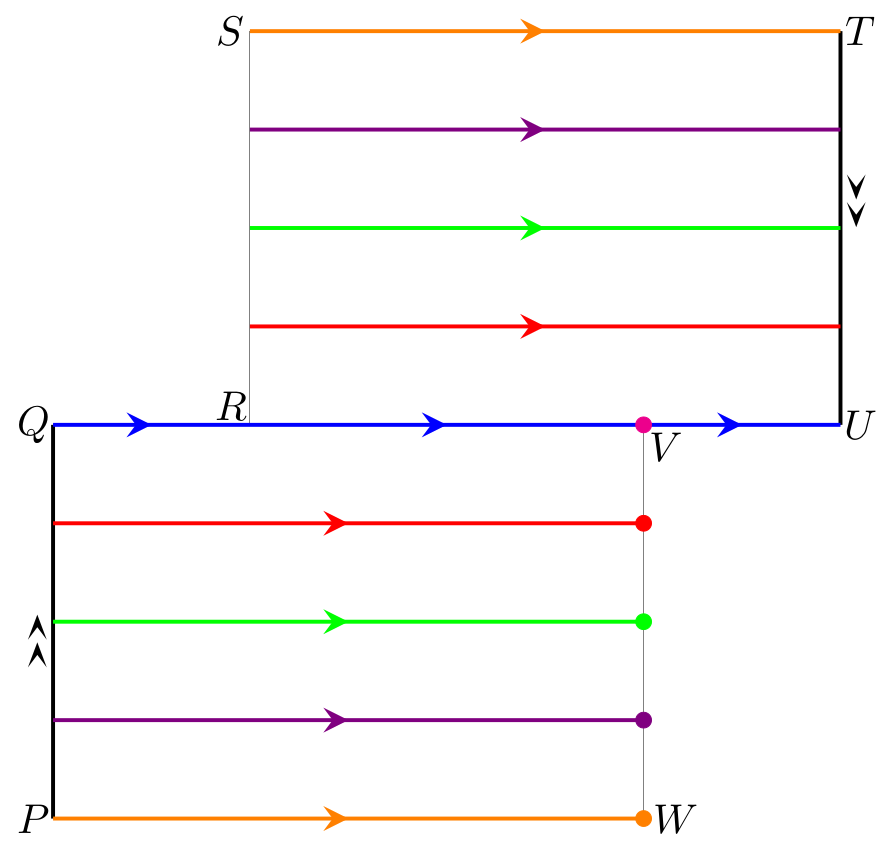}
    \caption{{\bf A semiflow with a non-invariant non-wandering set.}}
    \label{fig: nw3}
\end{figure}
\begin{example}
    Let $X$ be the Mobius strip realized as shown in Figure~\ref{fig: nw3}.
    We show several orbits of the continuous-time semiflow $F$ on $X$
    whose action on $X$ is defined as follows.
    Every orbit of $F$ is horizontal and each point, except for points on the blue one, lies on the orbit of a point on the $RS$ segment and each of these orbits asymptotes to a fixed point on the $VW$ segment. 
    The blue orbit has forward and backward limit sets equal to the magenta fixed point $V$.
    The only $F$-invariant sets of $X$ are the fixed points and the blue orbit.
    The reader can verify that each point in $RV$, including the endpoints of the segment, is non-wandering and that $NW_F=AE\cup EF$.
    Since $A$ is not fixed, $F^t(NW_F) = F^t(A)E\cup EF\neq NW_F$ for every $t>0$.
    Notice that no map $F^t$ is open at any point $p$ such that $F^t(p)=R$.
\end{example}
\black
\medskip
Since every point of $NW_F$ belongs to a maximal $\NW_F$-equivalent subset of $NW_F$ (i.e. to  a node), $NW_F$ decomposes naturally as the union of nodes.

The first result about decomposing $NW_F$ into ``elementary pieces'' goes back to Smale~\cite{Sma67}, that  proved that the non-wandering set of an Axiom-A diffeomorphism $f$ on a compact manifold is the finite union of disjoint closed invariant sets, on each of which $f$ is transitive (i.e. has a dense orbit).
Smale called this the {\em spectral decomposition} of $NW_f$.
%
%
%
Due to the presence of a dense orbit, clearly each set in the spectral decomposition is a $\NW$-equivalent set and so it is entirely contained in a single node of the non-wandering set.
In case of Axiom-A diffeomorphisms and all generalization of Smale's spectral decomposition known to the authors (maps on the interval~\cite{dMvS93}, homeomorphisms on compact~\cite{Akin93,AH94} and non-compact~\cite{DLRW13,KNY18}  metric spaces and on non-metrizable spaces~\cite{Oh22}), the spectral decomposition coincides with the decompostion of $NW_F$ into nodes of $\NW_F$.
%


Notice that, 
due to the fact that $\NW_F$ is not transitive,
non-wandering nodes are not, in general, disjoint from each other, as illustrated in next examples.

\begin{example}
    \label{ex: physical pendulum}
    Consider the physical pendulum system $q''=-\sin q$, whose phase-space portrait is shown in Figure~\ref{fig:physical pendulum}.
    Here $X$ is the $(q,p)$ plane and $F$ is the flow of the vector field $\eta=(p,-\sin q)$.
    In the figure, we labeled by $s_i$ the fixed saddle points, by $c_i$ the fixed centers and by $h_i$ (resp. $g_i$) the heteroclinic trajectory from $s_i$ to $s_{i+1}$ and lie above (resp. below) the center $c_i$.
    We denote by $D_i$ the open disc bounded by the critical loop $C_i=\{s_i\}\cup g_i\cup\{s_{i+1}\}\cup h_i$.
    In this case, 
    $$
    NW_F=\cup_i D_i\cup_i\{s_i\}\cup_i h_i\cup_i g_i.
    $$
    The reader can verify that the nodes are the following: each periodic trajectory plus each critical loop.
    Each pair of consecutive critical loop nodes $C_i,C_{i+1}$ has in common the saddle point $s_{i+1}$.
    Notice that each point of $h_i$ is upstream from all points of $C_i$ and all points of the $h_j$ with $j\geq i$ but the only $g_j$ it is upstream from is $g_i$.
    A corresponding property holds for the points of $g_i$.
\end{example}


    
%
\begin{example} 
    \label{ex:non transitive}
    Let $\mu_r\simeq 3.868...$ be the parameter value of the right endpoint of the period-3 window {\black of the logistic map's bifurcation diagram
    (see Figure~\ref{fig:per3})}.
    Then {\black the logistic map} $\ell_{\mu_r}$ has precisely three closed forward invariant indecomposable sets:
    the repelling fixed point $0$, a repelling Cantor set $C$ {\black (painted in red in Figure~\ref{fig:per3})} and an attracting period-3 cycle of intervals $A$.
    Each of these sets has an $F$-dense orbit and so is a node of $NW_{\ell_{\mu_r}}$, namely the spectral decomposition
    of the non-wandering set of $\ell_{\mu_r}$ is  $NW_{\ell_{\mu_r}}=\{0\}\cup C\cup A$. 
    
    The Cantor set and the attractor have in common a period-3 repelling periodic orbit $\gamma$.
    It is this non-empty intersection that makes $\NW_{\ell_{\mu_r}}$ non transitive.
    Take indeed a $x\in C\setminus A$, $y\in\gamma$ and $z\in A\setminus C$. 
    Then $(x,y)\in\NW_{\ell_{\mu_r}}$ and $(y,z)\in\NW_{\ell_{\mu_r}}$ but, for $\eps>0$ small enough, there is no
    $\eps$-$\alpha\omega$chain from $x$ to $z$ since any chain from $x$ to $z$ needs an intermediate jump to pass from $C$ to $A$.
    Hence, $(x,z)\not\in\NW_{\ell_{\mu_r}}$.
\end{example}
%
    %
%
\begin{figure}
    \centering
    \includegraphics[width=13cm]{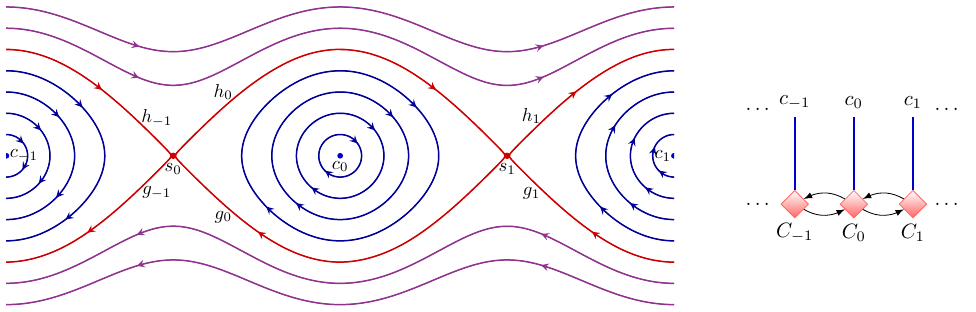}
    \\
    \phantom{a}
    \\
    \includegraphics[width=13cm]{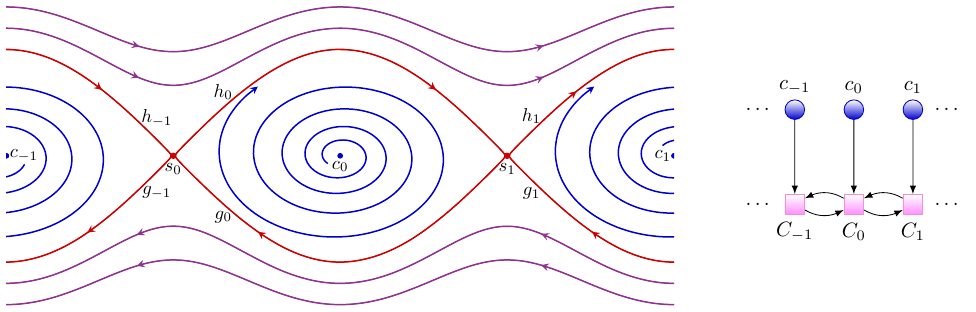}
    \caption{{\em Two examples of bidimensional dynamical systems and their prolongational graphs.} (Above) Phase-portrait of the physical pendulum on $\R$. In this case there are no attractors. Each periodic orbit (painted in blue) is a non-wandering node. Besides these nodes, the only other nodes are the boundaries $C_i$ of each disc foliated by concentric periodic orbits. For instance, $C_0$ is the disjoint union of the heteroclinic orbits $h_0$ and $g_0$ and the two saddle fixed points $s_0$ and $s_1$. Notice that this system, being Hamiltonian, has no attractor nor repellor and has a set of nodes with the power of continuum. The loops in the graph are due to the fact that each red node intersects its nearest neighborhoods. (Below) After replacing the periodic orbits with spirals spiralling outward, the fixed points $c_i$ become repelling and each interval of nodes above gets replaced by a repellor/attractor pair.}
    \label{fig:physical pendulum}
\end{figure}
%
\subsection{The Graph of \EM$\NW_F$} 
The qualitative properties of $\NW_F$ can be encoded into a graph as follows.
\begin{definition}
    The graph of $\NW_F$ is the directed graph   $\Gamma_{\NW_F}$ having the nodes of $NW_F$ as its vertices and such that there is an edge from node $N$ to node $M$ if and only if there is an $x\in N$ and a $y\in M$ such that $x\NWto y$.
\end{definition}
%

%
\begin{example}
  \label{ex:axiomA}
  In case of an Axiom-A diffeomorphism $F$, $\Gamma_{\NW_F}$ coincides with Smale's graph of $F$ described in the introduction.
\end{example}
%
%
\begin{example}
    \label{ex:circle}
    Consider a continuous-time semi-flow $F$ on the unit circle $C$ such that its north pole $N$ and  south pole $S$ are fixed while every other point rotates clockwise. 
    Then 
    $\Gamma_{\NW_F}$ 
    has two nodes, namely the fixed poles $N$ and $S$, and two edges, one from $N$ to $S$ and one from $S$ to $N$.
\end{example}
\begin{example}
    Extend the example above to a flow $\tilde F$ on the closed unit disc $D$ so that $\tilde F|_C=F$, the origin $c$ is fixed and every point in the interior of $D$ belongs to a bi-infinite spiral that asymptotes to $c$ for $t\to-\infty$ and to $C$ for $t\to\infty$. 
    Then $NW_{\tilde F}=\{c\}\cup C$ and there are two nodes: $C$ and $p$.
    The graph $\Gamma_{\NW_{\tilde F}}$ has a single edge, going from $p$ to $C$.
\end{example}
%
%
%
\begin{example}
    Consider the phase-space dynamics of the physical pendulum in Example~\ref{ex: physical pendulum}.
    In this case $\Gamma_{\tilde F}$ consists in countably many infinite segments, say $I_i=\{i\}\times[0,1]$, where $(i,0)$ represents the critical loop $C_i$ and each other point of $I_1$ represents a periodic orbit in the disc $D_i$. 
    There is a loop between each pair of nodes $(i,0)$ and $(i+1,0)$.
    As discussed in the example above, if we replace the dynamics in the $D_i$ by spirals, then there is a pair of nodes $C_i$, $c_i$ for every $i$, namely each $I_i$ gets replaced by a pair of nodes $C_i,c_i$ and an edge from one to the other (depending on whether the spiral runs outward or inward).
\end{example}
\begin{example}
  \label{ex:Klein}
    Let $F$ be the continuous-time semi-flow whose dynamics is sketched in Figure~\ref{fig:klein}(left).
    Here $X$ is the Klein bottle, seen as the unit square with boundary conditions:
    the opposite vertical sides of the square are glued after a twist, while the horizontal ones are glued straight.
    \black
    The set of fixed points of $F$ consists in all points painted in green plus the points $p_1$ and $p_2$.
    There is a homoclinic orbit based at the middle point of the segment of green fixed points, painted in teal.
    \black
    Each other orbit is heteroclinic, joining either $q$ with $p_1$ or $p_2$ to $r$ or two points of the green segment symmetric with respect to its center.
    As shown in the picture, all points on the horizontal segment between $p_1$ and $p_2$ are non-wandering but are not  limit-points nor limits of limit-points. 
    Hence, $NW_F$ is the union of the closed horizontal segment $G$ between $q$ and $r$ (painted in green) and of the closed horizontal segment $S$ between $p_1$ and $p_2$ (painted in red).
    While $S$ is a  node, each green fixed point is a node in itself. The prolongational graph of $F$ is shown in Figure~\ref{fig:klein}(right). The red saddle represents the node $S$.
    The central point $c$ of $G$ has a loop edge, due to the trajectory painted in teal in the left picture.
    Every other point of $G$ has an edge towards its symmetric with respect to $c$ and viceversa.
    Namely, for each such pair $x,y\in G$, we have that $x\NWto y$ and $y\NWto x$.
    Notice that, nevertheless, $x$ and $y$ do not belong to the same node because the chains from $x$ to $y$ and back from $y$ to $x$ are not linkable to each other.
\end{example}
\begin{figure}
    \centering
    \includegraphics[width=12cm]{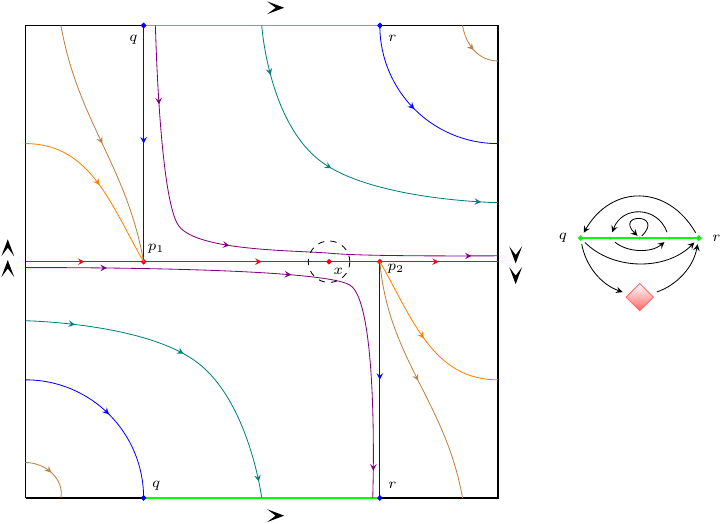}
    \caption{{\em A semi-flow with ``purely non-wandering'' points.}
      (Left) In the picture we show the dynamics of the semi-flow $F$ in Example~\ref{ex:Klein}, where $NW_F$ contains points that do not belong to the closure of the set of all limit points of $F$. 
      The phase space $X$ is the Klein bottle, the arrows next to the sides show the way the opposite sides are glued. 
      (Right) prolongational graph of $F$.
      The red saddle is the node of all points on the horizontal closed segment between $p_1$ and $p_2$.
      Each of the green fixed points is a node in itself.
      The point $c$ at the center of the green segment has a self-edge, each other point $x$ on it has an edge going to the point symmetric with respect
      to $c$ -- only two pairs of these edges are plotted.
    }
    \label{fig:klein}
\end{figure}
\begin{figure}
    \centering
    \includegraphics[width=10cm]{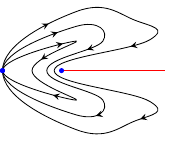}
    \caption{{\em A semi-flow with ``purely non-wandering'' edges.}
      \black
      (Left) In the picture we show the dynamics of a semi-flow with two fixed points $M$ and $N$ and where every other point lie on homoclinic trajectories
      based at $N$.
      In this case, there are non-wandering edges from $M$ to $N$ and viceversa but no bitrajectory between the two.
      (Right) Prolongational graph of $F$.
    }
    \label{fig:NWedges}
\end{figure}
\begin{example}
  \label{ex:NWedges}
    Let $F$ be the continuous-time semi-flow whose dynamics is sketched in Figure~\ref{fig:NWedges}.
    Here $X=\{(x,y):\;x\geq0\}\setminus\{(x,0):\;x\geq1\}$, $F$ has two fixed nodes $N=(0,0)$ and $M=(1,0)$ (painted in blue) and every other point with
    $x>0$ lies on a homoclinic orbit based at $N$.
    In this case, the two fixed points are the only nodes of $F$ and the graph of $F$ has {\black an edge from $M$ to $N$ and an edge from $N$ to $M$}.
\end{example}
\begin{definition}
  We call {\em bitrajectory} of $F$ a sequence of points $b=(\dots,b_{-1},b_0,b_1,\dots)$ such that $F(b_i)=b_{i+1}$ for every $i\in\bZ$.
  We denote by $\alpha(b)$ (resp. $\omega(b)$) the set of limit points of $b$ for $n\to-\infty$ (resp. $n\to\infty$).
  Given two nodes $M,N$, we say that $b$ runs from $M$ to $N$ if $\alpha(b)\subset M$ and $\omega(b)\subset N$.
\end{definition}
In Examples~\ref{ex:axiomA}--~\ref{ex:Klein}, when there is an edge from a node $M$ to a node $N$, there is a bitrajectory running from $M$ to $N$.
As Example~\ref{ex:NWedges} shows, though, this is not true in general.
Yet, next proposition shows that sometimes this is the case. 
\black
\begin{definition}
  \black
  Let $M,N$ be nodes of $\Gamma_{\NW_F}$ with $M\NWto N$, i.e. there is an edge from $M$ to $N$.
  We say $M$ and $N$ are {\EM adjacent} if, given a node $K$ with $M\NWto K$ and $K\NWto N$, either $K=M$ or $K=N$.
\end{definition}
\blue
\begin{lemma}
    \label{lemma: let's restrict}
    Let $Q$ be a global trapping region of $F$ which is a neighborhood of $NW_F$.
    Then the prolongational graph of $F$ coincides with the one of $F|_Q$.
\end{lemma}
\begin{proof}
    This is an immediate consequence of the fact that, since $Q$ is a forward-invariant neighborhood of $NW_F$, for $\eps>0$ small enough each $\eps$-$\alpha\omega$chain from $x\in NW_F$ to $y\in NW_F$ lies entirely in $Q$.
\end{proof}
\black
\begin{lemma}
  \label{lemma:NW adjacent}
  Let $M,N$ be non-wandering nodes of $F$. The following holds:
  \begin{enumerate}
    \item If there is a bitrajectory running from $M$ to $N$, then $M\NWto N$.
    \item \blue When $F$ has strong compact dynamics,\black\ the following inverse holds: if $M,N$ are adjacent,
      then there is a bitrajectory running from $M$ and $N$.
  \end{enumerate}
\end{lemma}
\begin{proof}
  (1) Via $b$, one can build, for every $\varepsilon>0$, a $\eps$-$\alpha\omega$chain from $M$ to $N$. Hence, $M\NWto N$.

  (2) 
  \blue
  We restrict $F$ to a compact global trapping region $Q$ which is a neighborhood of $NW_F$.
  \black  
  Assume first that $M\cap N\neq\emptyset$ and let $x\in M\cap N$. Then $\Omega_F(x)\subset M\cap N\subset N$.
  Since non-wandering nodes of $F$ are fully $F$-invariant, $x$ has a backward trajectory inside $M$.
  Hence, there is a bitrajectory $b$ with $\alpha(b)\subset M$ and $\omega(b)\subset N$.

  Assume now that $M$ and $N$ are disjoint. Since $Q$ is compact, $d(M,N)>0$.
  We claim that, for every $t>0$, there exist disjoint open sets $U\supset M$ and $V\supset N$ such that $F^t(U)\cap V=\emptyset$.
  If it were not so, for every $\eps>0$ there would be a $x$ with $d(x,M)<\eps$ such that $d(F^t(x),N)<\eps$.
  By taking $\eps_n=1/n$, we can build sequences $x_n$ such that $d(x_n,M)<\eps_n$ and $d(F^{t}(x_n),N)<\eps_n$.
  We can assume without loss of generality that $x_n\to y$. Clearly $y\in M$, so that $F^t(y)\in M$.
  Hence, by continuity, we should have $d(F^{t}(y),N)=0$, which is not possible since $d(M,N)>0$.

  An important consequence of the existence of such $U$ and $V$ is that every $\eps$-$\alpha\omega$chain has at least a point that does not
  belong to $U\cup V$. Take again $\eps_n=1/n$ and let $C_n$ be $\eps_n$-$\alpha\omega$chains from $M$ to $N$.
  Since $Q$ is compact, the points of these chains have at least an accumulation point $z$ not belonging to $U\cup V$.
  Since $z$ is limit of a sequence of points whose orbit passes within $\eps_n$ from $N$, by continuity $\Omega_F(z)\subset N$.

  Now, denote by $E$ the set of all accumulation points of all $\eps$-$\alpha\omega$chains from $M$ to $N$ for all $\eps>0$.
  Then $E$ is $F$-invariant.   
%
%
  Indeed, let $x_i\to z$, with $x_i\in C_i$. Then, for each $t>0$ and $x_i\in C_i$, there is a predecessor $y_i$ in $C_i$ such that $F^t(y_i)=x_i$.
  We can assume without loss of generality that $y_i$ converge to some $w$ in $E$, so that $F^t(w)=z$.

  Hence, for every $z\in E$ we can build a backward trajectory of $z$ in $E$ and therefore a bitrajectory $b$ based at $z$.
  Let $K$ be the node such that $\alpha(b)\subset K$. By construction, each point $x\in K$ is a limit of points belonging to $\eps$-$\alpha\omega$chains
  from $M$ to $N$ for $\eps\to0$. Hence, we can break each such chain into a $\eps$-$\alpha\omega$chain from $M$ to $K$ and another $\eps$-$\alpha\omega$chain
  from $K$ to $N$, so that $M\NWto K$ and $K\NWto N$. Since $M$ and $N$ are adjacent, this means that $K=M$. Hence, $b$ is a bitrajectory that runs from $M$ to $N$.
\end{proof}
\black
We conclude this subsection with the first of the main results of this work.
\begin{definition}
    A prolongational graph $\Gamma_{\NW_F}$ is {\em connected} if, whenever $NW_F=C_1\cup C_2$ with $C_1$ and $C_2$ closed and disjoint sets forward-invariant under $F$, there is an edge from a node of $C_1$ to a node of $C_2$ or viceversa.
\end{definition}
Notice that, since $C_1$ is $F$-forward-invariant, if a point of a node lies in $C_1$ then the whole node does, and similarly for $C_2$.
\begin{theorem}
    \label{thm:NW}
    Assume that $F$ has a connected global trapping region $C$ that is a neighborhood of $NW_F$.
    Then, if $F$ has compact dynamics, $\Gamma_{\NW_F}$ is a connected graph.
\end{theorem}
Notice that $C$ does not need to be globally attracting.
\begin{proof}
    \blue
    By Lemma~\ref{lemma: let's restrict}, the restriction of $F$ to $C$ has the same prolongational graph than $F$.
    Hence, we can assume without loss of generality that $X$ is connected.
    We denote by $Q$ a compact global trapping region of $F$.
    \black
    Let $C_1,C_2$ be two $F$-forward-invariant, closed and disjoint sets such that $NW_F=C_1\cup C_2$.
    Since $NW_F$ and all of its nodes are closed and $NW_F\subset Q$, then they are also compact.
    Hence, by Proposition~\ref{prop:Down(compact)} (and the comment underneath), $B_1=\Up_{\NW_F}(C_1)$ and $B_2=\Up_{\NW_F}(C_2)$ are closed.
    {\black Since $Q$ is globally attracting, every $x\in X$ has a non-empty limit set $\Om_F(x)\subset Q$.
      By case (5) of Proposition~\ref{prop:NW}, $\Om_F(x)\subset NW_F$ and therefore every $x$ is upstream from some non-wandering node.
      Hence, $B_1\cup B_2=X$.}
    Since $X$ is connected, there must be a point $x_0\in B_1\cap B_2$.
    Therefore, there exist points $x_i\in C_i$, $i=1,2$, such that $x_0\NWto x_i$.

    \black
    Denote by $N_i\subset C_i$ the node containing $x_i$, $i=1,2$ and by $N_0$ the node containing $\Omega_F(x_0)$. 
    By case (3) of Proposition~\ref{prop:Down}, either $x_i\in\cO_F(x_0)$ or $\Omega_F(x_0)\NWto x_i$, $i=1,2$.
    In the first case, by case (1) of Proposition~\ref{prop:NW}, $\cO_F(x_i)$ is a $\NW_F$-equivalent set.
    Hence $\cO_F(x_i)\subset\Omega_F(x_i)=\Omega_F(x_0)$, so that $x_i\in\Omega_F(x_0)$.
    In any case, therefore, we have that $\Omega_F(x_0)\NWto x_i$, $i=1,2$.
    \black
    Then, 
    $
    N_0\NWto N_1\subset C_1\text{ and }N_0\NWto N_2\subset C_2.
    $
    Since $N_0$ belongs to either $C_1$ or $C_2$, this means that there is either an edge from $C_1$ to $C_2$ or viceversa.
    Hence, $\Gamma_{\NW_F}$ is connected.
\end{proof}

\subsection{The prolongational graph of T-unimodal maps}
\label{sec:T-unimodal}
In this final subsection we discuss in some detail the graph of some special classes of unimodal maps.
While doing so, we sharpen the main result in~\cite{DLY20} in the context of the prolongational relation. 
\begin{definition}
    \label{def:T-unimodal}
    A continuous map $f:[a,b]\to[a,b]$ is {\em unimodal} if there exists a unique point $c\in(a,b)$ such that $f$ is monotonically increasing (resp. decreasing) in $[a,c]$ and monotonically decreasing
    (resp. increasing) in $[c,b]$.
    Point $c$ is called the {\em critical point} of $f$.
    We say that an attracting periodic orbit $\gamma$ of $f$ is {\em non-topological} if $c$ is not in the basin of attraction of $\gamma$.
    \black
    We say that an interval $J\subset[a,b]$ is a {\em wandering interval} if $f^m(J)\cap f^n(J)=\emptyset$ for every $n>m\geq0$ and $f^n(J)$ does not lie in the basin of a periodic orbit.
    \black
    We say that $f$ is {\em T-unimodal} if $f$ has no wandering intervals and no non-topological attracting periodic orbits. 
\end{definition}
\black
The ``T'' in {\em T-unimodal} alludes to the fact that the graph of all these unimodal maps is a tower (see Definition~\ref{def:tower}).
T-unimodal maps include the tent map family~\cite{ADL23} as well as every S-unimodal map~\cite{DLY20,DL22}.
See Example~5.5 in~\cite{ADL23} for an example of a non-trivial T-unimodal map that is not topologically conjugate neither to a tent map neither to a S-unimodal map. 
\black

The theorem below summarizes the most relevant well-known results on invariant sets of unimodal maps.
%
\begin{theoremX}[Gukenheimer, 1979~\cite{Guc79}, Jonker and Rand, 1980~\cite{JR80}]
  \label{thm:JonkerRand}
  Let $f$ be a unimodal map with critical point $c$. 

  \noindent 
  The repellors of $f$ can be only of the following two types:
  \begin{enumerate}
    \item a periodic orbit;
    \item a Cantor set on which $f$ acts as a subshift of finite type with a dense orbit.
  \end{enumerate}
  
  \noindent
  The attractors of $f$ can be only of the following three types:
  \begin{enumerate}
      \item a periodic orbit;
      \item a cycle of intervals containing $c$; 
      \item a Cantor set on which $f$ acts as an adding machine. 
  \end{enumerate}
  The map $f$ is transitive, i.e. has a dense orbit,  in each of the repellors and attractors types above. 
  Attractors of type 2 and 3 contain $c$.
  If there is no attractor of type 2 or 3, then one of the attracting periodic orbits has $c$ in its basin of attraction.
  When the attractor is of type 3, and only in this case, $NW_f$ has (countably) infinitely many nodes, whose unique accumulation point is the attractor.
\end{theoremX}
For a thorough discussion and proofs of the statement above, see also Thm~4.1 and~4.2 in~\cite{dMvS93}.

Theorem~\ref{thm:JonkerRand} has the following two fundamental consequences.
\begin{corollary}
    Every T-unimodal map $f$ has a unique attractor.
\end{corollary}
\begin{proof}
    By definition, $f$ has no non-topological attracting periodic orbit.
    Hence, by Theorem~\ref{thm:JonkerRand}, $c$ must be in the basin of attraction of the attractor, so there can be only one attractor.
\end{proof}
\begin{corollary}
    The spectral decomposition of $NW_{\ell_\mu}$ in Theorem~\ref{thm:JonkerRand} coincides with the decomposition of $NW_{\ell_\mu}$ into nodes of the prolongational relation $\NW_{\ell_\mu}$.
    In other words, every node of $\NW_{\ell_\mu}$ is of one of the five types listed in Theorem~\ref{thm:JonkerRand}.
\end{corollary}
\begin{proof}
    Each of the invariant non-wandering sets in Theorem~\ref{thm:JonkerRand} has a dense orbit.
    Hence, each of them is a node of $\NW_{\ell_\mu}$.
\end{proof}

In~\cite{DLY20} we studied, in case of the logistic map family $\ell_\mu$, the graph $\Gamma_{\ell_\mu}$
defined below (see also~\cite{DL22} for more general statements on S-unimodal maps). 
We did not give a particular name to it, so in this article we use the name given to it by Mizin.
%
\begin{definition}[Mizin, 2002~\cite{Miz02}; De Leo \& Yorke, 2020~\cite{DLY20}]
  \label{def:graph}
  Given a map $f:X\to X$, the {\em structural graph} of $f$ is the graph $\Gamma_f$ whose nodes are the chain-recurrent nodes of $f$, namely sets $N$ of points $x$ such that,
  for each $\eps>0$, there are $\eps$-chains from $x$ to each other point of $N$ (for more details on chain-recurrence see Section~\ref{sec:CR}).
  There is an edge from node $M$ to node $N$ if and only if there is a bi-infinite trajectory $b$ such that $\alpha(b)\subset M$ and $\omega(b)\subset N$.    
\end{definition}
Notice that $\eps$-$\alpha\omega$chains are $\eps$-chains and so each non-wandering point is also chain-recurrent and each non-wandering node is contained in a chain-recurrent node.

\begin{definition}
  \label{def:tower}
    %
    A directed graph $\Gamma$ is a {\em tower} if it has no loops and there is an edge between any two nodes. 
\end{definition}
The following is our main result on CR graphs of the logistic map.
\begin{theoremX}[De Leo, Yorke 2021~\cite{DLY20}]
  \label{thm:tower}
    The structural graph of the logistic map $\ell_\mu$ is a tower for each $\mu\in[0,4]$.
\end{theoremX}
In other words, chain-recurrent nodes of $\ell_\mu$ can be sorted in a linear order $N_0,N_1,\dots,N_p$ (where $p$ can be infinite) so that there is an edge from $N_i$ to $N_j$ if and only if $j>i$.
In this ordering, node $N_0$ is always the repelling fixed point 0 and $N_p$ is always the attractor.
Notice that the cases $\mu\leq2$ and $\mu=4$ are trivial because, for those values, there is a single node (the fixed point $0$ for $\mu\leq2$ and the whole interval $[0,1]$ for $\mu=4)$.

Below we translate these result in terms of properties of $\Gamma_{\NW_{\ell_\mu}}$:

\noindent
{\em 1) Nodes.}
In~\cite{DLY20}, we provided the classification below of chain-recurrent nodes of the logistic map family.
Notice that by ``periodic window'' we indicate the maximal range of $\mu$ for which exists a given non-wandering repelling Cantor set (see Figure~\ref{fig:per3} for a picture of the logistic map family's period-3 window).
\begin{theoremX}[De Leo, Yorke 2021~\cite{DLY20}]
  \label{thm:attracting nodes}
    Let $\ell_\mu$ be a logistic map.
    Then:
    \begin{enumerate}
        \item If $\mu$ is not an end-point of a periodic window of the logistic map family, then its chain-recurrent nodes coincide with its non-wandering nodes.
        \item If $\mu$ is the left end-point of a periodic window, then the (one-sided) attractor is a periodic orbit belonging to a repelling Cantor set.
          The attracting chain-recurrent node is equal to this Cantor set.
        \item If $\mu$ is the right end-point of a periodic window, then the (one-sided) attractor is a cycle of intervals having a non-empty intersection with a repelling Cantor set. In this case, the chain-recurrent node contains the attractor, part of its basin of attraction and the repelling Cantor set.
 \end{enumerate}
\end{theoremX}

\begin{figure}
    \centering
    \includegraphics[width=13cm]{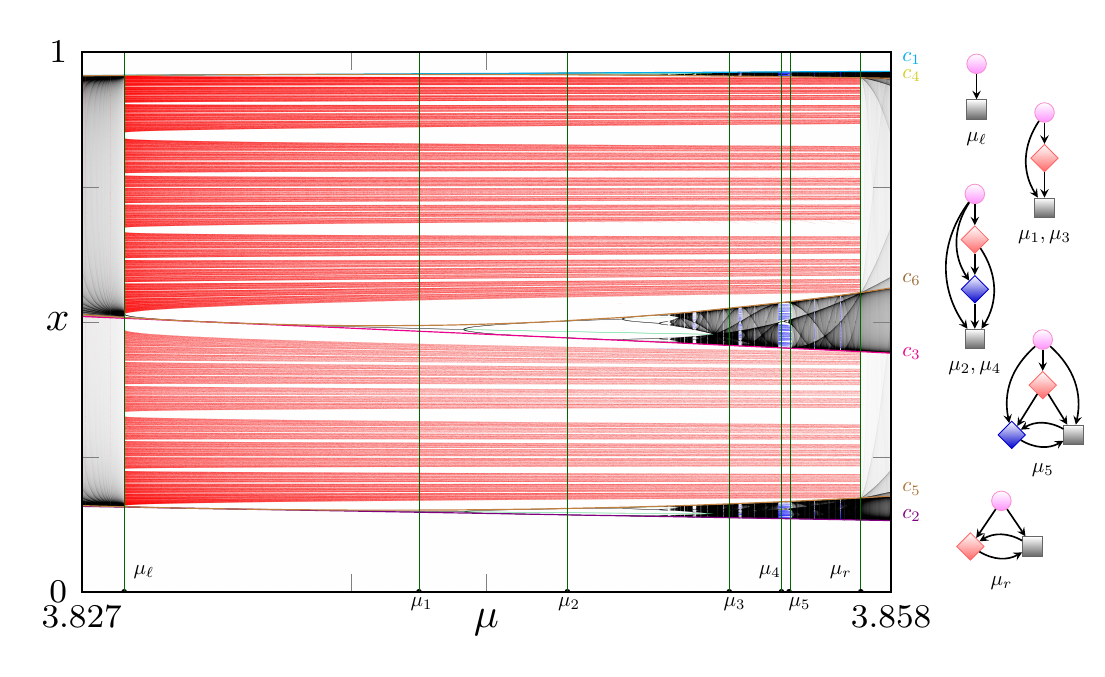}
    \caption{{\em Prolongational graphs of the logistic map for seven parameter values within the period-3 window.}
    The graph corresponding to each of the seven parameters marked on the bifurcation diagram is shown at the right of the diagram. 
    Circles represent repelling non-wandering nodes, diamonds represent saddle ones and squares represent attractors.
    Each node, excepts for the repellor (the fixed point $x=0$), is painted in the same color it appears with in the bifurcation diagram.
    The curves $c_1,\dots,c_6$ are the first 6 iterates of the critical point.
    This picture is discussed in detail at the bottom of Section~\ref{sec:T-unimodal}.
    }
    \label{fig:per3}
\end{figure}

\noindent
    {\em 2) Edges.} As shown in next proposition, when $X$ is compact, as is the case for the logistic map family, \black there is a strict relation between \black the edges
    of the structural graph 
    and the edges of the prolongational graph.
 \begin{proposition}
     Let $\Gamma_F$ be the structural graph of a map $F$, let $M,N$ be distinct nodes of $\Gamma_F$, each of which contains a \black single \black non-wandering node $\hat M,\hat N$. Then:
     \begin{enumerate}
         \item If there is an edge \black in $\Gamma_F$ \black from $M$ to $N$, then there is an edge \black in $\Gamma_{\NW_F}$ \black from $\hat M$ to $\hat N$.
         \item \black Assume that $X$ is compact and let $\hat M,\hat N$ be adjacent.
           Then there is an edge in $\Gamma_F$ from $M$ to $N$.
     \end{enumerate}
 \end{proposition}
 \black
 \begin{proof}
   1) Under the hypotheses, there is a bitrajectory $b$ with $\omega(b)\subset N$ and $\alpha(b)\subset M$.
   By Proposition~\ref{prop:NW}, $\omega(b)\subset\hat N$ and $\alpha(b)\subset\hat M$.
   Via $b$, we can build $\varepsilon$-$\omega\alpha$chains from $\hat M$ to $\hat N$ for every $\varepsilon>0$.
   Hence $\hat M\NWto \hat N$.

   2) This is an immediate consequence of Proposition~\ref{lemma:NW adjacent}.
 \end{proof}
 \black

In~\cite{ADL23}, we extended all main results of~\cite{DLY20} to T-unimodal maps whose attractor is a cycle of intervals by showing that our proofs ultimately only rely on \black purely \black
topological properties, \black first and foremost \black the absence of wandering intervals and of non-topological attracting periodic orbits.
The same ideas apply to the case of the other two types of attractors, so that Theorem~\ref{thm:tower},~\ref{thm:attracting nodes} actually apply more generally to T-unimodal maps. 
Hence, we can reformulate the results above in the following way.
\begin{proposition}
  \label{prop:Tunimodal}
  \black 
    Let $f$ be a T-unimodal map and let $N_0,\dots,N_p$ be its non-wandering nodes sorted according to their distance from $c$, so that $N_0=\{a\}$ and $N_p$ is the unique attractor.
    Then:
    \begin{enumerate}
    \item $N_i\cap N_j=\emptyset$ for each $i\neq j$ with $i,j<p$;
    \item each $N_i$, $i=0,\dots,p-2$, is also a chain-recurrent node;
    \item if $N_{p-1}\cap N_p=\emptyset$, then every non-wandering node of $f$ is also a chain-recurrent node and viceversa and $\Gamma_{\NW_f}$ coincides with the structural graph $\Gamma_f$ (in particular, $\Gamma_{\NW_f}$ is a tower);
    \item if $N_{p-1}\cap N_p\neq\emptyset$, then $N_{p-1}$ is a repelling Cantor set and $N_p$ is a cycle of intervals.
    In this case, $\Gamma_{\NW_f}$ is equal to a tower with $p+1$ nodes with the addition of an extra edge from $N_{p}$ to $N_{p-1}$.    
%
%
    \end{enumerate}
\end{proposition}
\noindent
\subsubsection{Some prolongational graphs of the logistic map}
The content of Proposition~\ref{prop:Tunimodal} is illustrated in Figure~\ref{fig:per3}, where are shown the prolongational graphs of seven logistic maps within the period-3 window.
The picture shows several invariant sets for each $\mu$. 
The attractor is painted in shades of gray.
The red Cantor set is the one that defines the periodic window; it arises at $\mu_\ell$, when a pair of attracting/repelling period-3 periodic orbit is created, and it changes continuously up to $\mu_r$, where it collides with the attractor. 
This is precisely the hyperbolic subshift of finite type described by Smale and Williams in~\cite{SW76}.
The attracting periodic orbit of the pair mentioned above is shown first in gray and then, after it biforcates and becomes repelling, in green, until it plunges in a chaotic attractor. 
The Cantor set painted in blue corresponds to the logistic map's period-9 window.

Each of the graphs corresponding to the seven parameters has one repellor (the fixed point $x=0$), one attractor and either one or two saddles.
Parameters $\mu_1,\dots,\mu_4$ are not end-points of a periodic window, so the prolongational graph of the corresponding logistic map is a tower, identical to the structural graph discussed in full detail in~\cite{DLY20}.
\blue
Parameter $\mu_\ell$ is the left end-point of the period-3 window.
Even in this case, the prolongational graph and the structural graph coincide.
Parameter $\mu_r$ is the right end-point of the window and its graph is a tower consisting of the three nodes  with the addition of an extra edge from the attractor back to the red Cantor set. 
\black
Finally, parameter $\mu_5$ is the right end-point of the period-9 window, so its prolongational graph is a tower with four nodes with the addition of an extra edge from the attractor to the blue Cantor set.
\section{Streams}
%
The prolongational relation $\NW_F$ is closed and extends $\cO_F$ but is not, in general, a quasi-order (see Example~\ref{ex:non transitive}).
The main goal of the present article is to show that 
closed quasi-order extensions $D$ of $\cO_F$, which extend at the same time the roles of $\cO_F$ (as a quasi-order) and $\NW_F$ (as a closed relation), play a particularly important role in dynamics.
\begin{definition}
    We call {\EM $F$-stream} (or simply {\em stream}, when there is no ambiguity) on $X$ a closed quasi-order $S$ that is an extension of $\cO_F$.
    A {\em substream} of $S$ is an $F$-stream $S'$ such that $S'\subset S$.
\end{definition}   
\begin{proposition}
    \label{prop:smallest stream}
    Let $S$ be an $F$-stream on space $X$. 
    Then $\cL_S\subset\cL_F$. 
    Hence, in particular:
    \begin{enumerate}
        \item $S\supset\NW_F$;
        \item $\cR_S\supset NW_F$;
        \item $\cD_{\cL_F}$ is the smallest $F$-stream on $X$.
    \end{enumerate}
\end{proposition}
\begin{proof}
    Let $L\in\cL_S$.
    Since $\cO_F\subset S$, $x\Fto y$ implies $x\Sto y$ which, in turn, means that $L(x)\geq L(y)$.
    Hence, $L\in\cL_F$.
\end{proof}
The smallest stream of a semi-flow $F$ was introduced and studied first by Auslander in 60s~\cite{Aus63} and so we denote it, in the remainder of the article, by $\cA_F$.
\begin{proposition}
    \label{prop:Down_S}
    Let $S$ be an $F$-stream. Then:
    \begin{enumerate}
        \item $\Down_S(x)\supset\cO_F(x)\cup\Down_S(\Omega_F(x))$;
        \item $\Down_{S}(x)
         \supset \cO_F(x)\cup\Down_{S}(y)
        $ for all $y\in\cO_F(x)$.
    \end{enumerate}
\end{proposition}
\begin{proof}
    (1) If $x\Fto y$, then $x\Sto y$ because $S$ is an extension of $\cO_F$.
    If $y\in\Omega_F(x)$, then $x\Sto y$ because $y$ is the limit of a sequence of points on $\cO_F(x)$ and $S$ is closed.
    If $y\in\Down_S(\Omega_F(x))$, then $x\Sto y$ by transitivity.

    (2) This is an immediate consequence of transitivity.
\end{proof}

\subsection{Recurrent points and nodes of a stream}
Recurrent and non-wandering points are generalizations of periodic points.
Similarly, nodes of $\NW_F$
are a generalization of periodic orbits of $F$ (which are the nodes of $\cO_F$).
Streams allow to define a natural and quite general analogue of periodic orbits in the following way.
\begin{definition}
    Given a stream $S$ of $F$, we write {\EM$x\Seq y$}, and we say that $x$ and $y$ are {\EM$S$-equivalent}, if $x\Sto y$ and $y\Sto x$.
    
    We say that $x$ is {\EM $S$-recurrent} if either $x$ is fixed or there is a $y\neq x$ such that $x\Seq y$.
    We denote by $\cR_S$ the set of all $S$-recurrent points.
    We say that a set $M\subset\sR_{S}$ is $S$-equivalent if all points of $M$ are mutually $S$-equivalent.
\end{definition}
The set $\cR_S$ comes with a natural equivalence relation, unavailable in $NW_F$ due to the lack of transitivity of $\NW_F$.
%
\begin{proposition}
    The relation $\Seq$ is an equivalence relation on $\cR_{S}$.
\end{proposition}
\begin{proposition}
    \label{prop:R_S monotonic}
    Let $S,S'$ be two $F$-streams and assume $S\subset S'$. 
    Then $\cR_S\subset\cR_{S'}$.
\end{proposition}
We leave the proof to the reader.

{

\begin{definition}
    Let $D'$ be a sub-quasi-order of $D$ and let $N$ and $N'$ nodes of, respectively, $D$ and $D'$.
    We say that $N$ extends $N'$ if $N\supset N'$.
\end{definition}
}
\begin{definition}
    We call {\em nodes} the equivalence classes of $\cR_{S}$ with respect to $\Seq$.
    If $x\in\cR_S$, we denote by $\Node_S(x)$ the node containing $x$.
\end{definition}
Note that, unlike the nodes of $\NW_F$, nodes of a stream are all mutually disjoint because they are defined via an equivalence relation.
\begin{example}
    Every point in $\cR_{\cO_F}$ is $F$-periodic 
    and the nodes of $\cR_{\cO_F}$ are precisely the periodic orbits of $F$. 
    Not all recurrent points of $F$, though, are in $\cR_{\cO_F}$ since, in general, there are recurrent points that are not periodic.
    For instance, let $F$ be a rotation on the circle by an irrational angle. 
    Then the orbit of any point is dense, namely every point is recurrent for $F$, but no point is actually periodic.
\end{example}
\begin{proposition}
    \label{prop:streams}
    Let $S$ be an $F$-stream. Then:
    \begin{enumerate}
        \item For each $x$, $\Omega_F(x)$ is a $S$-equivalent set. In particular, there is a node $N$ of $S$ such that $\Omega_F(x)\subset N$.
        \item If $x$ is $F$-recurrent, then $\cO_F(x)\subset\Omega_F(x)\subset\Node_S(x)$.
    \end{enumerate}
\end{proposition}
\begin{proof}
    Since $S\supset\NW_F$, the claim is an immediate consequence of Proposition~\ref{prop:Om(x)} and Proposition~\ref{prop:NW}. 
\end{proof}
\begin{proposition}
    \label{prop:Lyap nodes}
    Let $S$ be a stream.
    Then each Lyapunov function $L\in\cL_S$ is constant on each node of $S$.
\end{proposition}
\begin{proof}
  Let $N$ be a node. If $N$ is a fixed point, there is nothing to prove.
  Otherwise, let $x,y\in N$ with $x\neq y$. Then $x\Sto y$ and $y\Sto x$, so that, by Definition~\ref{def:Lyapunov},
  $L(x)\geq L(y)$ and $L(y)\geq L(x)$.
  Hence, $L$ is constant on $N$.
\end{proof}
\begin{proposition}
    \label{prop:streams intersection}
    Any intersection of $F$-streams is an $F$-stream.
\end{proposition}
\begin{proof}
    Quasi-orders and closed sets are closed under intersection.
\end{proof}
\begin{corollary}
    Given any set $R\subset X\times X$ and semi-flow $F$, there is a smallest $F$-stream containing $R$.
    \black
    In particular, $\cA_F$ coincides with the smallest $F$-stream containing $\NW_F=\overline{\cO_F}$.
\end{corollary}
%

\subsection{The Graph of a stream}
The qualitative properties of a stream can be encoded in a graph as follows.
\begin{definition}
    \label{def: N1>=N2}
    Let $S$ be an $F$-stream.
    Given two sets $A,B\subset X$, we write $A\Sto B$ if and only if $x\Sto y$ for each $x\in A$ and $y\in B$.
\end{definition}
%
%
\begin{definition}[\em Graph of a stream]
    \label{def graphs}
    Given an $F$-stream, the graph of $S$, denoted by {\em\EM$\Gamma_S$}, is the directed graph whose nodes are the nodes of $\sR_S$ and such that there is an edge from a node $N_1$ to a node $N_2\neq N_1$ if and only if $N_1\leadsto N_2$.
\end{definition}
Notice that, due to transitivity, if there is an edge from $N_1$ to $N_2$, then $x\leadsto y$ for every $x\in N_1$ and $y\in N_2$.
Moreover, again by transitivity, edges are themselves transitive, in the sense that $N_1\leadsto N_2$ plus $N_2\leadsto N_3$ implies
$N_1\leadsto N_3$. 
Hence, when we draw the graph of a stream, we usually draw it with only a minimal set of edges, where the rest can be inferred.

\black\medskip\noindent
%
\begin{proposition}
    The graph of a stream $S$ has no cycles.
\end{proposition}
\begin{proof}
    If $x_1\leadsto\dots\leadsto x_n\leadsto x_1$ for some integer $n>1$, then all these $x_i$ are mutually $S$-equivalent and so they all belong to the same node. 
\end{proof}
\begin{figure}
 \centering
 \includegraphics[width=13cm]{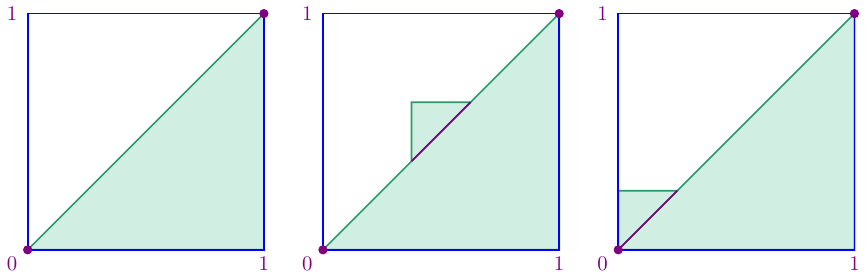}
 \caption{%
   {\em Three examples of stream structures on the unit interval for the flow $F$ of the ODE $x'=-\sin(\pi x)$.}
   (Left) The the smallest stream of $F$ is the closed triangle with vertices $(0,0),(1,0),(1,1)$. 
   This stream has only two nodes: the fixed points 1 and 0. 
   Both nodes are invariant under $F$.
   The graph of this stream is a tower with two nodes (see Figure~\ref{fig:graphs}(a)).
   (Center) An $F$-stream that is not an $\Omega$-stream. 
   In this case there is a third node containing the point 0.5. 
   Notice that this node is not forward-invariant under $F$.
   The graph of this stream is a tower with three nodes (see Figure~\ref{fig:graphs}(b)).
   (Right). 
   In this stream there are just two nodes but, unlike the smallest stream, the node containing the fixed point 0,
   is forward-invariant but not fully-invariant under $F$.
   The graph of this stream is identical to the graph of the smallest stream.
 }
 \label{fig:streams}
 \end{figure} 
\begin{example}
    \label{ex:streams}
    In Figure~\ref{fig:streams} we show three different $F$-streams on $X=[0,1]$, where $F$ is a continuous-time flow that keeps 0 and 1 fixed and moves every other point leftward so that it asymptotes to 0. For instance, this is the case of the flow of the ODE $x'=-\sin(\pi x)$. 

    The stream on the left is the closure of $\cO_F$, so it is the smallest $F$-stream $\cA_F$ (which, in this case, also coincides with $\NW_F$). 
    Its nodes are the repelling fixed points $N_0=\{1\}$ and the attracting fixed point $N_1=\{1\}$.
    Indeed, for any non-fixed $\cA_F$-recurrent point $x$ there should be a $y\neq x$ such that both $(x,y)\in\cA_F$ and $(y,x)\in\cA_F$ but in $\cA_F$, that is entirely contained on and below the diagonal of $X\times X$, there are no pairs with this property.
    Finally, there is an edge from 1 to 0 since $(x,0)\in\cA_F$ for all $x\in[0,1)$ and so, since streams are closed, $(1,0)\in\cA_F$.
    Hence, $\Gamma_{\cA_F}$ is a tower with two nodes.
    In this case, both nodes are $F$-invariant.
    
    The stream $S$ in the center contains extra points in a neighborhood of $(1/2,1/2)$. 
    These extra points are responsible for the presence of extra $S$-recurrent points that constitute a single new node $N_{1/2}$, which is a segment centered at $x=1/2$.
    By closedness, as above, one can see that there are edges from $N_1$ to $N_{1/2}$ and from $N_{1/2}$ to $N_0$.
    By transitivity, this means that there is also an edge from $N_1$ to $N_0$.
    Hence, $\Gamma_S$ is a tower with three nodes.
    Notice that this extra node is not $F$-invariant nor even just forward-invariant under $F$; in particular, it does not contain any limit set of its points (they all converge to 0 under $F$).

    Similarly to the case above, the extra points (with respect to $\cA_F$) of the stream $S'$ in the right give rise to new $S'$-recurrent points but these ones are all $S'$-equivalent to 0 and so, rather than creating a new node, they just make the node $N_0$ thicker.
    Hence, $\Gamma_{S'}=\cA_F$.
    In this case, though, node $N_0$ is not $F$-invariant but only forward-invariant.
\end{example}
Recall that the smallest stream $\cA_F$ is the smallest transitive extension of $\NW_F$.
Hence, $\Gamma_{\cA_F}$ can be obtained from $\Gamma_{\NW_F}$ by replacing recursively each loop in the prolongational graph by a single $\cA_F$-node.
\begin{figure}
    \centering
    \includegraphics[width=12cm]{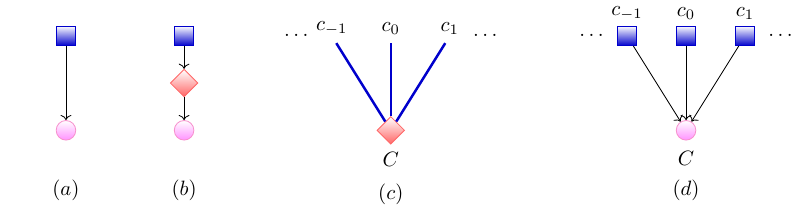}
    \caption{{\em Examples of graphs of the smallest stream of a semi-flow.}
      \black
          {\em (a) A tower with two nodes.} This is the graph of the streams in Fig.~\ref{fig:streams}(left,right) and of the smallest stream of the logistic map for $\mu\in(1,3]$.
            {\em (b) A tower with three nodes.} This is the graph of the stream in Fig.~\ref{fig:streams}(center) and of the smallest stream of the logistic map for
            infinitely many intervals of values of $\mu$, e.g. for $\mu\in(3,1+\sqrt{6}]$ (in which case the saddle is a repelling fixed point and the attractor a period-2 orbit)
    or for $\mu$ close enough from the right to the left endpoint of the period-3 window (in which case the saddle is a repelling Cantor set and the attractor a period-3 orbit).
      {\em (c,d) Graphs of the smallest streams of the systems described in Examples~\ref{ex: physical pendulum} and~\ref{ex:physpenstream}.}
    }
    \label{fig:graphs}
\end{figure}
\begin{example}
  \label{ex:physpenstream}
    Consider the physical pendulum system $F^t$ in Figure~\ref{fig:physical pendulum}(above).
    Every red point (i.e. every point on a saddle or on a heteroclinic orbit) can by joined to any other red point,
    for every $\eps>0$, by a finite sequence of $\eps$-$\alpha\omega$chains.
    Hence, every red point is $\cA_F$-equivalent to every other red point, so that the set of all red points is a single $\cA_F$-node.
    The graph $\Gamma_{\cA_F}$ is shown in \black Figure~\ref{fig:graphs}(c). \black
    Similarly happens in case of the system in Figure~\ref{fig:physical pendulum}(below), whose graph is shown in \black Figure~\ref{fig:graphs}(d).
\end{example}
\begin{example}
    Let $f$ be a T-unimodal map.
    We discussed in Section~\ref{sec:T-unimodal} the structure of $\Gamma_{\NW_f}$.
    When $\Gamma_{\NW_f}$ has no loops, then $\Gamma_{\cA_f}=\Gamma_{\NW_f}=G_f$.
    When it does have loops, after eliminating the only loop that can arise in those prolongational graphs, the graph gets again back to $G_f$, the structural graph of $f$.
    In particular, the graph of the smallest stream of a T-unimodal map is always a  tower. 
\end{example}
\begin{definition}
    A stream graph $\Gamma_S$ is {\em connected} if, whenever $S=C_1\cup C_2$, with $C_1$ and $C_2$ closed and disjoint sets each of which is union of nodes of $S$, there is an edge from a node of $C_1$ to a node of $C_2$ or viceversa.
\end{definition}
%
    
%
\black
\begin{theorem}
    \label{thm: connectedness}
    Let $S$ be a $F$-stream and assume that $F$ has a connected global trapping region $C$ that is a neighborhood of $NW_F$.
    Then, if $F$ has compact dynamics, $\Gamma_{S}$ is a connected graph.
\end{theorem}
\begin{proof}
    Each $F$-stream $S$ has two types of nodes: those that are an extension of non-wandering nodes of $F$ and those that are not. 
    In this proof, we will refer to the first type as ``$\Omega$-nodes''.
    Notice that $S$-nodes that are not $\Omega$-nodes cannot be forward-invariant under $S$.
    If they were, indeed, they would contain an $F$-recurrent point, since every node is compact under the theorem's hypotheses, and every
    $F$-recurrent point is in some non-wandering node.
    Ultimately, each $S$-node either is an $\Omega$-node or has an edge from itself to at least one $\Omega$-node.

    Suppose now that $S=C_1\cup C_2$ with $C_1,C_2$ closed and disjoint sets each of which is a union of nodes.
    If either one of the two, say $C_1$, only contains nodes that are not $\Omega$-nodes, then there is at least an edge from $C_1$ to $C_2$
    because the limit sets of all points of $C_1$ lie in $C_2$.
    Suppose now that both $C_1$ and $C_2$ contain $\Omega$-nodes \blue  and recall that, under the theorem's hypotheses, $\Gamma_{\NW_F}$ is connected (Theorem~\ref{thm:NW}).
    Hence,\black\ there is at least an edge between an $\Omega$-node in $C_1$ and an $\Omega$-node in $C_2$.
    Otherwise, it would be possible to sort the non-wandering nodes into two disjoint closed sets so that there would be no edge between the two sets,
    namely $\Gamma_{\NW_F}$ would not be connected.
    Hence, $\Gamma_S$ is connected.
\end{proof}
%
\subsection{$\Omega$streams.}
An immediate consequence of stream's transitivity is that, for all $y\in\cO_F(x)$,
$$
\Down_{S}(x) \supset \cO_F(x)\cup\Down_{S}(y).
$$
It turns out that streams for which the sets at the left and right hand sides above are equal enjoy rather special properties.
In this section we illustrate some of them.

%
\begin{definition}
    We say that an $F$-stream $S$ is a {\EM $\Omega$stream} if
    \beq
    \Down_S(x) = \cO_F(x)\cup\Down_S(y)
    \eeq
    for every $y\in\cO_F(x)$.
\end{definition}
\begin{proposition}
  \label{prop:Ostream}
  Let $F$ be a semi-flow with compact dynamics. 
    Then $S$ is an $F$-$\Omega$stream if and only if    
    \beq
    \Down_S(x) = \cO_F(x)\cup\Down_S(\Om_F(x)).
    \eeq
\end{proposition}
\begin{proof}
    By Proposition~\ref{prop:Down_S}, it is enough to prove that, if $z\in\Down_S(x)$, then either $x\Fto z$ or $\Omega_F(x)\Sto z$.
    Let $z\in\Down_S(x)$ and suppose that $z\not\in\cO_F(x)$.
    Then, since $S$ is an $\Omega$stream, $y\Sto z$ for all $y\in\cO_F(x)$.
    Since $F$ has a compact global trapping region, $\Omega_F(x)\neq\emptyset$.
    Let $w\in\cO_F(x)$.
    Then there is a sequence $y_n\in\cO_F(x)$ such that $y_n\to w$.
    Since $y_n\Sto z$, then $w\Sto z$ as well since $S$ is closed.
    Hence, $\Omega_F(x)\Sto z$.
\end{proof}
Next example shows that the compactness hypothesis in the previous proposition cannot be dropped.
\begin{example}
    Let $F$ be the flow of the vector field $\eta=(2v,1-v^2)$ in the $(u,v)$ plane.
    This flow is shown in Figure~\ref{fig:H}.
    Notice that the function $H(u,v)=(1-v^2)e^u$ is a first-integral for $\eta$.
    In this case $\Omega_F(x)=\emptyset$ for every $x=(u,v)\in\R^2$ but there are points $x$ for which $\Down_{\NW_F}(x)$, and therefore $\Down_S(x)$ for every $F$-stream $S$, is strictly larger than $\cO_F(x)$.
    Indeed, notice that this flow has two special integral trajectories $s_\pm=\{v=\pm1\}$ that are inseparable in the following sense.
    Say that a neighborhood $U$ of an integral trajectory $t$ is {\em saturated} when, if $x\in U$, the whole integral trajectory through $x$ is contained in $U$.
    Then every saturated neighborhood of $s_+$ has non-empty intersection with every saturated neighborhood of $s_-$.
    The reader can verify that
    $$
    \Down_{\NW_F}(x)=
    \begin{cases}
    \cO_F(x)\cup s_+,& x\in s_-\\
    \cO_F(x),&\text{otherwise}\\
    \end{cases}.
    $$
    In Figure~\ref{fig:H} two downstream sets are shown, one for a point in $s_-$ (in blue) and one for a point not on $s_-$ (in red).
    Notice that, in this case, $\sR_S=NW_F=\emptyset$.
    Results in~\cite{DL11} show a connection between non-trivial downstream sets of a flow and the solvability of the relative cohomological equation.
\end{example}
\begin{figure}[t]
    \centering
    \includegraphics[width=10cm]{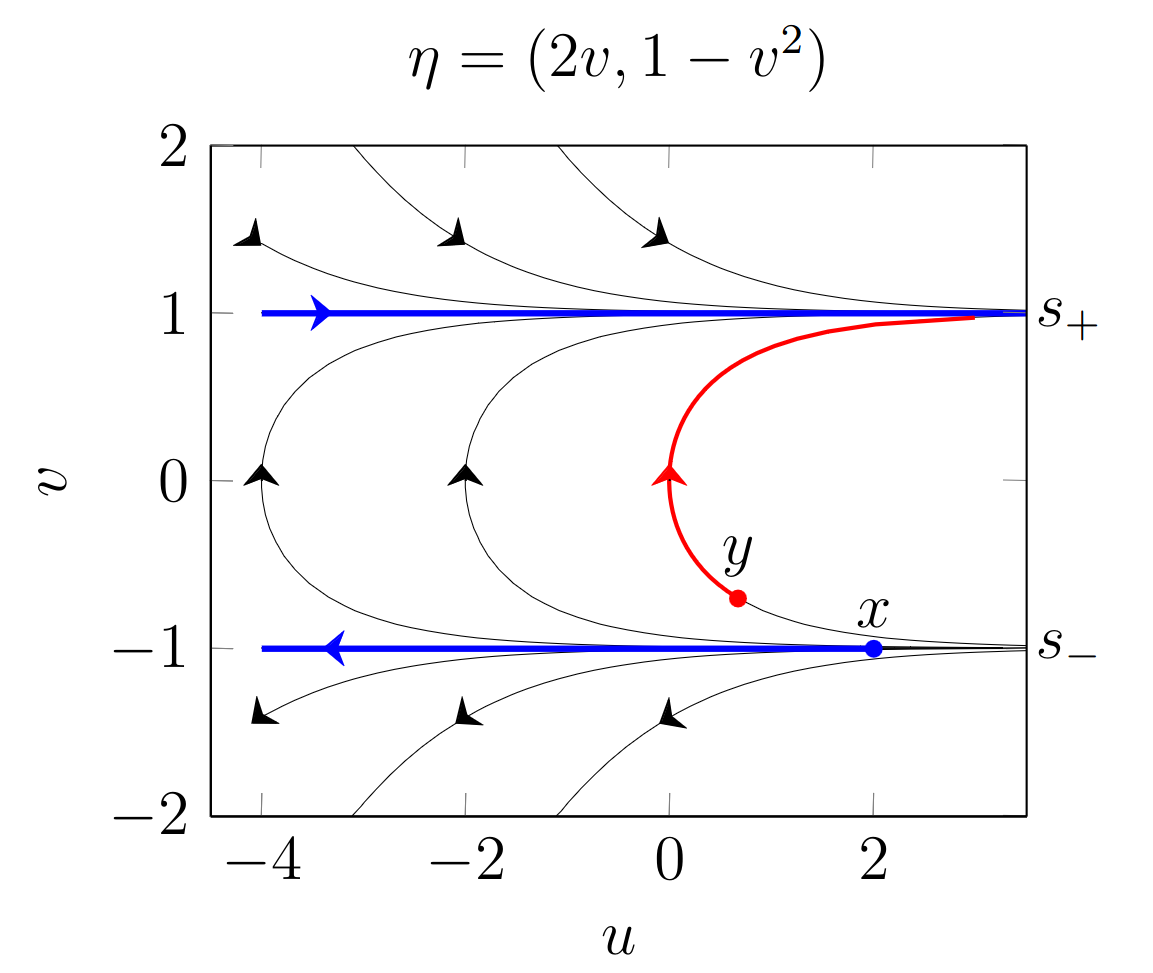}
    \caption{{\EM A semi-flow $F$ such that $\Down_{\NW_F}(x)\supsetneq\cO_F(x)\cup\Down_{\NW_F}\Omega_F(x)$.} 
    Here $X=\R^2$ and $F$ is the flow of the vector field $\eta=(2v,1-v^2)$.
    We painted in black some flow lines of $\eta$ and in blue the points of $\Down_{\NW_F}(x)$.
    The blue points are the points in $\Down_{\NW_F}(x)$.
    The red points are the points in $\Down_{\NW_F}(y)$.
    }
    \label{fig:H}
\end{figure}
\begin{definition}
    We say that a node $N$ of an $F$-stream $S$ is {\em dynamical} if $N$ contains $F$-recurrent points. 
\end{definition}
Next propositions shows that every dynamical node contains non-wandering points.
\begin{proposition}
    \label{prop:forward-invariant}
    Every forward-invariant node of a stream is dynamical. 
\end{proposition}
\begin{proof}
    If a node is forward-invariant, it contains the limit-sets of each of its points. 
    By the compactness hypothesis, every limit-set is non-empty.
    By the Proposition~\ref{prop:recurrent pt}, each limit-set contains $F$-recurrent points.
\end{proof}
\begin{corollary}
    A node $N$ of an $F$-stream is dynamical if and only if it extends a node of $NW_F$.
\end{corollary}
\begin{proof}
    This is an immediate consequence of the fact that non-wandering nodes are forward-invariant.
\end{proof}
\begin{proposition}
    \label{prop:minimal streams}
    Let $S$ be an $F$-$\Omega$stream. Then:
    \begin{enumerate}
        \item If $x$ is $S$-recurrent, then $\cO_F(x)\subset\Node_S(x)$.
        \item $\Down_S(\Om_F(x))\cap\sR_S=\Down_S(x)\cap\sR_S$.
        \item Every node of $S$ is closed and forward-invariant under $F$.
        \item Every node of $S$ is dynamical.
        \item $\cR_S$ is closed and forward-invariant under $F$.
        \item $S$ is forward-invariant under the natural action induced by $F$ on $X\times X$.
        \item $\Down_S(M)$ is forward-invariant under $F$ for each $M\subset X$.
        \item $\Up_S(C)$ is forward-invariant under $F$ for each set $C\subset\cR_S$ that is union of nodes of $S$.
    \end{enumerate}
\end{proposition}
\begin{proof}
    (1)  
    If $x$ is $F$-recurrent, the claim follows immediately from Proposition~\ref{prop:streams}.
        
    Assume now that $x$ is $S$-recurrent but not $F$-recurrent. 
    Then there is a $y\neq x$ that is $S$-equivalent to $x$.
    Now, let $z\in\cO_F(x)$.
    Since $S$ is an $\Omega$stream, $\Down_S(x)=\cO_F(x)\cup\Down_S(z)$.
    Hence, either $y\in\cO_F(x)$ or $y\in\Down_S(z)$.
    When either $y\in\cO_F(z)$ or $y\in\Down_S(z)$, then $y$ (and so $x$) is $S$-equivalent to $z$ and so $z\in\Node_S(x)$.
    When $z\in\cO_F(y)$, namely $y$ is between $x$ and $z$ in the orbit of $x$, we have that
    $\Down_S(y)=\cO_F(y)\cup\Down_S(z)$.
    Since, by hypothesis, $x\in\Down_S(y)$, then either $x\in\cO_F(y)$ or $x\in\Down_S(z)$.
    The first case cannot happen since $x$ would be periodic.
    In the second case, it follows immediately that $z$ is $S$-equivalent to $x$ and so $z\in\Node_S(x)$.

    (2)
    Since $\Down_S(\Omega_F(x))\subset\Down_S(x)$, it is enough to prove that every $S$-recurrent point in $\Down_S(x)$ does belong to $\Down_S(\Omega_F(x))$.
    Since, by Proposition~\ref{prop:streams}, $\Omega_F(x)$ is a $S$-equivalent set, it is actually enough to prove that every $S$-recurrent point in $\Down_S(x)$ is downstream from any point $y\in\Om_F(x)$.
    Notice that, since $S$ is an $\Omega$stream, either $z\in\cO_F(x)$ or $z\in\Down_S(y)$.
    Since $z\in\cR_S$, in both cases we get that $z\in\Down_S(y)$.

    (3)
    We prove first that $N$ is closed.
    Let $x_n$ be a sequence of points of $N$ converging to some $x\in X$.
    The case when $x$ is fixed is trivial, so we assume it is not, namely there exists $t>0$ such that $F^t(x)\neq x$.
    We can assume without loss of generality that also none of the $x_n$ is fixed.
    By point (1) above, each $\cO_F(x_n)$ is  $S$-equivalent, so that $F^t(x_n)$ and $x_n$ are $S$-equivalent for every $n\geq1$.
    By continuity, therefore, since $S$ is closed we have that $x$ and $F^t(x)$ are $S$-equivalent, so that $x\in\sR_S$.
    Moreover, since by construction $x_n\lra x_m$ for each $n,m$, then by continuity also $x_n\lra x$ for all $n$, namely $x\in N$.

    We prove now that $N$ is forward-invariant, namely that $\cO_F(x)\subset N$ for each $x\in N$.
    If $x$ is periodic, then $\cO_F(x)=\Om_F(x)\subset\Node_S(x)=N$ and the claim follows.
    Assume now that $x$ is not periodic.
    Then there is a $y\neq x$ that is $S$-equivalent to $x$.
    Let $z\in\cO_F(x)$.  
    In order to prove that $z\in N$, it is sufficient to show that either $z\in\Down_S(x)$ or $z\in\Down_S(y)$.
    Since $S$ is an $\Omega$stream, then either $y\in\Down_S(z)$ or $y\in\cO_F(x)$.
    The only non-trivial case is when $y$ is between $x$ and $z$ in $\cO_F(x)$. 
    In this case $z\in\cO_F(y)$ and so, again by minimality, we get that 
    $\Down_S(y)=\cO_F(y)\cup\Down_S(z)$.
    Since $x\in\Down_S(y)$, the only non-trivial case is that $x\in\cO_F(y)$.
    But then $\cO_F(x)=\cO_F(y)$, namely the orbit is periodic, against our working assumption.
    Hence, $\cO_F(x)$ is a $S$-equivalent set, namely $\cO_F(x)\subset N$.

    (4) 
    This is a corollary of Proposition~\ref{prop:forward-invariant}.

    (5) 
    That $\cR_S$ is closed follows from the same argument used in (3). 
    Moreover, $\cR_S$ is forward-invariant because it is the union of forward-invariant sets.

    (6)
    We need to show that $F^t(S)\subset S$, namely that, if $(x,y)\in S$, then  $(F^tx,F^ty)\in S$ for each $t\geq0$.
    So, let $(x,y)\in S$.
    By minimality, there are two cases: either $y\in\cO_F(x)$, in which case the result is true by transitivity, or $y$ is downstream from $\cO_F(x)$. In the latter case, by continuity, $$
    F^tx\leadsto\Om(x)\leadsto y\leadsto F^ty.$$

    (7)
    This amounts just to the observation that $\cO_F(x)\subset\Down_S(x)$.
    
    (8)
    Let $x\in\Up_S(C)$, \ie, $x\leadsto C$. 
    Then, there exists $y\in C$ such that $x\leadsto y$.
    Since $S$ is forward-invariant, then, 
    for every $t\geq0$, $F^t(x)\leadsto F^t(y)$. 
    Since $y\in\cR_S$, $F^t(y)\lra y$.
    Hence, $F^t(x)\Sto y$, so that $F^t(x)\in\Up_S(C)$.
\end{proof}
\begin{proposition}
    \label{prop:intersection of minimal}
    Let $S_\alpha$ be a collection of $F$-$\Omega$streams.
    Then:
    \begin{enumerate}
        \item $S_\cap = \bigcap_\alpha S_\alpha$ is an $F$-$\Omega$stream.
        \item $\cR_{S_\cap} = \bigcap_\alpha\cR_{S_\alpha}$.
            \item For every node $N$ of $\cR_{S_\cap}$, denote by $N_\alpha$ the node of $\cR_{S_\alpha}$ that extends $N$.
            Then $N=\bigcap_\alpha N_\alpha$.
    \end{enumerate}
\end{proposition}
\begin{proof}
    (1) 
    $S_\cap$ is a stream by Proposition~\ref{prop:streams intersection}. Moreover, for every $y\in\cO_F(x)$,
    \beq
    \Down_{S_\cap}(x)=
    \bigcap_\alpha\Down_{S_\alpha}(x)=
    \cO_F(x)\cup\bigcap_\alpha\Down_{S_\alpha}(y)=\cO_F(x)\cup\Down_{S_\cap}(y),
    \eeq
    so $S_\cap$ is an $\Omega$stream.
    (2) 
    This is an immediate consequence of the fact that, by Proposition~\ref{prop:minimal streams}, the whole orbit of a point $x\in\cR_S$ is $S$-recurrent and points of the orbit are all mutually $S$-equivalent. 
    In particular, this means that 
    \beq
    x\in\cR_{S_\alpha} \Longleftrightarrow
    \cO_F(x)\subset\Node_{S_\alpha}(x)
    \Longleftrightarrow
    \cO_F(x)\subset\Node_{S_\cap}(x)
    \Longleftrightarrow
    x\in\cR_{S_\cap}.
    \eeq    
    (3)
    Left to the reader.
\end{proof}
%

\subsection{Subgraphs and Towers}
Graphs of streams can be highly complicated.
Even in case of the logistic map, one of the simplest possible non-trivial dynamical system, for certain parameters the non-wandering set has infinitely many nodes, although its phase space $X=[0,1]$ is compact.
Moreover, as shown in the example below, 
one can easily modify locally the logistic map so to obtain a unimodal map without wandering intervals and with an arbitrary (at most countable number) of attractors (all but one consisting of non-topological attracting orbits), saddles and repellors.
\begin{figure}
    \centering
    \includegraphics[width=12cm]{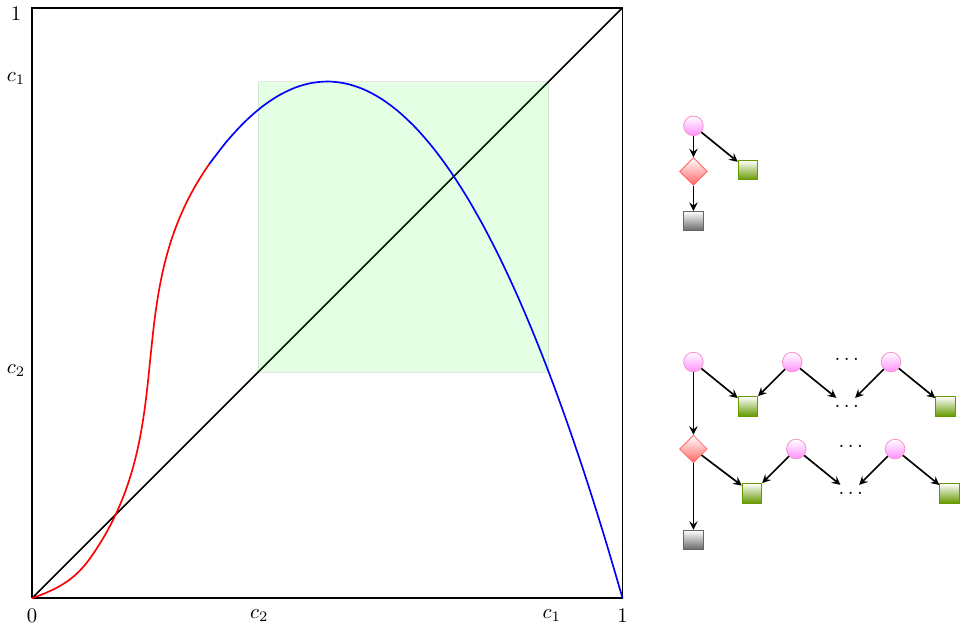}
    \caption{{\em A unimodal map and the graph of its smallest stream.} 
    (left) Graph of a unimodal map $f$. This unimodal map coincides with the logistic map $\ell_\mu$, with $\mu=3.5$, in some neighborhood of the trapping region $[c_2,c_1]$, where $c_1=\ell_\mu(0.5)$ and $c_2=\ell_\mu(c_1)$.
    (right, top)
    Graph of the smallest stream of $f$, namely the chains stream (see Section~\ref{sec:countably many nodes}).
    Notice that the top node of the graph is 
    the repelling fixed point corresponding to the internal intersection with the diagonal of the red portion of the graph of $\ell_\mu$.
    The attracting node in green is the left endpoint $x=0$.
    (right, bottom)
    By repeating the trick, one can create any number of non-topological periodic orbits, resulting in the graph shown in the bottom of the right column.
    In the two stream graphs are shown only the basic edges.
    Every non-displayed edge can be inferred by transitivity.
    }
    \label{fig:unimodal}
\end{figure}
\begin{example}
    Recall that the prolongational graph of a logistic map $\ell_\mu$, $\mu\in[0,4]$, with $\mu$ not at the boundary of a periodic window, is a tower. 
    In this case, it coincides with the graph of the smallest stream of $\ell_\mu$ (see the bottom of Section~\ref{sec:countably many nodes}).
    Now, set $c_k=\ell^k_\mu(1/2)$.
    Notice that the dynamics of $\ell_\mu$ in the interval $J=[0,c_2)$ is trivial:
    $0$ is fixed and every other point enters the trapping region $Q=[c_2,c_1]$ in finite time.
    All non-trivial dynamics takes place inside $Q$.
    Hence, by modifying $\ell_\mu$ in $J$, one obtains a smooth unimodal function $f$ whose graph contains, as a subgraph (see Definition~\ref{def:subgraph}), the graph of $\ell_\mu$.
    In Figure~\ref{fig:unimodal} we show an example of such unimodal maps.
    The map $f$ whose graph is shown in the picture has a single non topological attracting fixed point, obtained by modifying a logistic map $\ell_\mu$ in a neighborhood of 0.
    In this particular case $\mu=3.5$, so that $\ell_\mu$ has three nodes: the  fixed node $0$ (repelling), the other fixed node $p$ (saddle) and a period-2 periodic orbit (attracting).
    The blue part of the graph of $f$ coincides with the graph of $\ell_\mu$.
    Close to zero, we choose $f$ so that  $0<f'(0)<1$ and that $f$ increases strictly monotonically until it meets the blue branch.
    This way, 0 is attracting while the other intersection $\phi$ of the red branch of the graph is repelling.
    Hence, $NW_f$ (and therefore  $\cA_f$) has four nodes: the fixed point 0 (attracting, non-topological), $\phi$ (repelling), $p$ (saddle) and the period-2 orbit (attracting).
    With the same trick, one can get unimodal maps $g$ with
    infinitely many non-topological attracting periodic orbits.
    Assume that one starts with a logistic map $\ell_\mu$ has infinitely many nodes. Then the graph of $\cA_g$ 
    has infinitely many repelling nodes, infinitely many attracting nodes and infinitely many saddle nodes.
    An example of such graph (with a single saddle but potentially infinitely many repelling and attracting nodes) is shown in Figure~\ref{fig:unimodal}(right, bottom).
\end{example}
The situation gets much more complicated in dimension larger than one. 
Indeed, while unimodal maps with infinitely many nodes are non-generic, Sheldon Newhouse showed~\cite{New74} that, on any compact manifold of dimension larger than one and for any $r\geq2$, there are open sets  ${\cal U}\subset\text{Diff}^r(M)$ such that a $C^r$-generic map of $\cal U$ has  
infinitely many attractors.

Hence, it is particularly important to identify important building blocks of stream graphs.
Towers, first introduced in Section~\ref{sec:T-unimodal}, are perhaps the most important one.
In this section, we present a general result about them.
%
\begin{definition}
    \label{def:subgraph}
    Given a directed graph $\Gamma$, the directed graph $\Gamma'$ is a {\em subgraph} of $\Gamma$ if the set of nodes of $\Gamma'$ is a subset of the set of nodes of $\Gamma$ and, if there is an edge from node $N$ to node $M$ in $\Gamma'$, there is an edge from node $N$ to node $M$ in $\Gamma$.
    Given a directed graph $\Gamma$, we call {\em tower} any {\black acyclic} subgraph $\Gamma'\subset\Gamma$ such that there is an edge between each $M,N\in\Gamma'$.
    A tower is {\em maximal} if it is not a subgraph of a larger tower.
\end{definition}
\begin{definition}
    A node $N$ of a stream $S$ is a {\em top node} if there are no $S$-nodes upstream of it and is a {\em bottom node} if there are no $S$-nodes downstream of it. 
\end{definition}
While it is tempting thinking of bottom nodes as attractors, the following example shows that this is not always the case.
\begin{example}
    Consider the scalar ordinary differential equation
\beq
 \frac{dx}{dt} = - x^2\sin\frac{\pi}{x}
\eeq
where the right-hand side is defined to be 0 at $x=0$. For integers $n\neq 0$, the points $x=\pm\frac{1}{n}$ are fixed points and are attracting for $n$ odd and repelling for $n$ even. Notice that all trajectories with $x(0)>1$ are attracted to the point $x=1$ while those with $x(0)<-1$ are attracted to the point $x=-1$. 
Hence there are countably many nodes, and there is an edge between adjacent nodes, one of which is an attractor and a top node and another is a repellor and a bottom node.  
The graph is a connected curve that includes the node $x=0$. However, there are no edges to or from that node. Hence it is both a top and a bottom node. This graph is plotted as a saw tooth plot with top nodes higher than the bottom nodes they connect to. An infinite number of nodes converge to the node at $x=0$.
\end{example}
%
%
\begin{theorem}
    \label{thm:topnode}
    Let $S$ be an $F$-stream and assume that either $X$ is compact or that $S$ is an $\Omega$stream.    
    Then each node $N$ of $S$ is contained in a maximal tower of $\Gamma_S$.    
\end{theorem}

\begin{proof}
    \blue
    In order to prove the claim, we need to work in a compact environment.
    Let $Q$ be a compact global trapping region of $F$.
    If the $F$-stream $S$ is a $\Omega$stream, then every node of $S$ lies in $Q$ (see case 5 of Lemma~\ref{lemma:S trapping region}) and so it is enough to prove the claim for the restriction of $F$ to $Q$. 
    If $S$ is not a $\Omega$stream, then we must assume that $X$ is compact.
    Hence, we can assume without loss of generality that $X$ is compact.
    
    We prove below that there is a top node $T\Sto N$ and we leave the analogous proof for a bottom node to the reader.
    By Zorn's Lemma, there is a maximal tower $\cT$ of nodes such that $N\in\cT$.
    For each node $M\in\cT$, let $\cT(M)=\{A\in\cT: A\Sto M\}$.
    Then $\cT(M)$ is an ordered chain.
    The set $U(M)\bydef\cup_{A\in\cT(M)}A$ is compact and the collection of all $U(A)$ for $A\in\cT$ is a nested set of compact sets.
    Let $V\bydef\cap_{A\in\cT(M)} U(A)$.
    $V$ is non-empty and compact and consists of points that are mutually $S$-equivalent.
    Hence, $V$ is contained in a node $T\bydef\Node(V)\in\cT$ and all points of $T$ are in $V$.
    If there were a node $B\Sto T$, with $B\neq T$, then it would not be in the chain and so the chain would not be maximal, a contradiction.
    Hence, $T$ is a top node.
\end{proof}



\subsection{Trapping regions of streams} 
%
Recall that a trapping region for $F$ is a closed set $Q$ such that $F^t(Q)\subset Q$ for all $t\geq0$.
By analogy, we provide the following definition in case of streams.
\begin{definition}
    Given a stream $S$, we say that a closed set $Q$ is a trapping region for $S$ 
    if $\Down_S(Q)\subset Q$.
\end{definition}
Notice that, since $\cO_F\subset S$, each trapping region for an $F$-stream $S$ is also a trapping region for $F$.

%
\begin{lemma}
    \label{lemma:S trapping region}
    Let $S$ be an $F$-stream on $X$ and let $Q$ be a trapping region for $S$.
    The following hold:
    \begin{enumerate}
        \item if $S'$ is a substream of $S$, then $Q$ is a trapping region for $S'$;
        \item if a node $N$ of $S$ has some point in $Q$, then $N\subset Q$;
        \item there is no edge in $\Gamma_S$ from any node in $Q$ to any node outside $Q$;
        \item if $S$ is an $\Omega$stream \blue and $Q$ is a global trapping region for $F$\black, then $x\Sto y$, with $x,y\not\in Q$, if and only if $x\Fto y$;
        \item if $S$ is an $\Omega$stream \blue and $Q$ is a global trapping region for $F$\black, $\cR_S\subset Q$.
    \end{enumerate}
\end{lemma}
\begin{proof}
    \blue
    (1)
    If $S'$ is a substream of $S$, then $\Down_{S'}(Q)\subset\Down_{S}(Q)\subset Q$.

    (2)
    Let $x\in N\cap Q$ and let $y\in N$.
    Then $y\in\Down_S(x)\subset Q$. 

    (3)
    Let $N$ be a node of $S$ inside $Q$ and let $M$ be another node of $S$.
    Since there is an edge from $N$ to $M$, then $M\subset\Down_S(N)$.
    Hence, $M\subset Q$.

    (4)
    Since $S$ is a $\Omega$stream, $\Down_S(x)=\cO_F(x)\cup\Down_S(\Omega_F(x))$.
    Since $Q$ is globally attracting, either $\Omega_F(x)$ is empty or $\Omega_F(x)\subset Q$.
    Hence, $\Down_S(\Omega_F(x))\subset Q$ and so, since $y\in\Down_S(x)$, necessarily $y\in\cO_F(x)$, namely $x\Fto y$.

    (5)
    Let $x\in\cR_S$.
    If $x$ is periodic, then $\Omega_F(x)=\cO_F(x)$ and, since $\Omega_F(x)\subset Q$, then $x\in Q$.
    If $x$ is not periodic, then there is a $y\in\cR_S$ with $y\neq x$, $x\Sto y$ and $y\Sto x$.
    Assume that $x\not\in Q$. 
    Then it must be also $y\not\in Q$ by case (3) and so, by case (4),  $y\in\cO_F(x)$ and $x\in\cO_F(y)$.
    But then $x$ is periodic and so it must lie in $Q$, against the assumption.
\end{proof}
\begin{proposition}
    Let $\cS_S$ be the set of substreams of an $\Omega$stream $S$ and let $Q$ be a trapping region for $S$ \blue that is globally attracting.
    \black
    Then $S_1,S_2\in\cS_S$ coincide if and only if their restriction to $Q$ coincide.
\end{proposition}
\begin{proof}
    By hypothesis, $S_1\cap(Q\times Q)=S_2\cap(Q\times Q)$.
    We need to prove that $(x,y)\in S_1$ if and only if $(x,y)\in S_2$ when either $x\not\in Q$ or $y\not \in Q$ (or both).
    
    Suppose that $(x,y)\in S_1$.
    By case 4 of the lemma above, when both $x,y$ are outside of $Q$, $(x,y)$ belongs to all substreams of $S$.
    If $x\in Q$, since $Q$ is a trapping region for $S$, we must have that $y\in Q$.
    The last case is when $x\not\in Q$ and $y\in Q$. 
    In this case, since $S$ is an $\Omega$stream, either $(x,y)\in\cO_F$, in which case it belongs to all streams, or $y\in\Down_{S_1}(\Om_F(x))$.
    Since $S_1$ and $S_2$ coincide inside $Q$, then $(x,y)\in S_2$.
\end{proof}


    


\section{Streams of chains}
The most important generalizations of  recurrent points in literature are all based on chains: generalized recurrence, first considered by Auslander in 1963~\cite{Aus63}, is based on $\eps$-$\alpha\omega$chains; chain-recurrence, introduced by Conley in 1972~\cite{Con72b,Moe88}, is based on $\eps$-chains; strong chain-recurrence, introduced in 1977 by Easton~\cite{Eas77}, is based on  $\eps$-$\Sigma$chains (see~\cite{Wis18} for several examples and recent results on these three recurrence generalizations).
In this section we show that each of those generalizations are the set of recurrent points of suitable streams. 
\subsection{The chains streams}
\label{sec:CR}
\blue
Recall that, given a map $f$ on $X$, a $(f,\eps)$-chain from $x$ to $y$ of length $n+1$ is a sequence $(c_0,\dots,c_n)$ such that $c_0=x$, $c_n=y$ and $d(f(c_i),c_{i+1})<\eps$ for all $i=0,\dots,n-1$.

The following technical lemma will be used several times in the rest of the section.
\begin{lemma}[Hurley, 1991~\cite{Hur91}, Lemma 1.2]
    \label{lemma: UC discrete}
    Let $f:X\to X$ be continuous  \blue and let $x\in X$.\black
    Then, for any $\eps>0$ and any integer $p>0$, there exists a $\delta>0$ such that:
    \begin{enumerate}
        \item 
        $
        d(x,y)<\delta \implies d(f^k(x),f^k(y))<\eps\;\text{ for every }\;k=0,\dots,p;
        $
        \item 
        for every $(f,d,\delta)$-chain $C=(c_0,c_1,\dots,c_n)$  \blue based at $x$ \black with $n\leq p$, the chain $C'$ consisting in the pair of points $c_0,c_n$ is a $(f^n,d,\eps)$-chain.
        Equivalently, 
        $$
        d(f^n(c_0),c_n)<\eps.
        $$
        \blue
        If $f$ is uniformly continuous, then $\delta$ can be chosen independently on $x$.
    \end{enumerate}
\end{lemma}
\begin{proof}
    \blue
    (1)
    We leave it to the reader.

    (2)
    The case $p=1$ is a tautology and holds for every $\delta<\eps$. 
    Assume now that the claim holds for all $p\leq p_0$.
    Since we already know that the claim is true for all $n=1,\dots,p_0$, it is enough to prove that, for every $\eps>0$, there is a $\delta>0$ such that $d(f^{p_0+1}(c_0),c_{p_0+1})<\eps$.
    
    Consider a chain $(c_0,\dots,c_{p_0+1})$ and notice that 
    $$
    d(f^{p_0+1}(c_0),c_{p_0+1})
    \leq
    d(f^{p_0+1}(c_0),f(c_{p_0}))
    +
    d(f(c_{p_0}),c_{p_0+1}).
    $$
    By continuity, there is an $\eta>0$ such that $d(y,z)<\eta$ implies $d(f(y),f(z))<\eps/2$.
    By the inductive assumption, 
    there is a $\delta_1>0$ such that, if $(c_0,\dots,c_{p_0})$ is a $(f,d,\delta_1)$-chain, then $d(f^{p_0}(c_0),c_{p_0})<\eta$.
    Therefore, if $(c_0,\dots,c_{p_0})$ is a $(f,d,\delta_1)$-chain, then $d(f^{p_0+1}(c_0),f(c_{p_0}))<\eps/2$.
    
    Assume now that $(c_0,\dots,c_{p_0+1})$ is a $\eps/2$-chain.
    Then $d(f(c_{p_0}),c_{p_0+1})<\eps/2$.
    Finally, let $\delta=\min\{\eps/2,\delta_1\}$.
    Then, for every $\delta$-chain based at $x$ of length at least $p_0+2$, 
    $
    d(f^{p_0+1}(c_0),c_{p_0+1})<\eps/2+\eps/2=\eps.
    $
\end{proof}
\blue
\begin{corollary}[Hurley, 1991]
    \label{cor: chains with infinite length}
    Let $y\in\Down_{\cC_{F,d}}(x)$ and $y\not\in\cO_F(x)$.
    Then, if $\eps_n\to0^+$ and $C_n$ is a $(F,d,\eps_n)$-chain from $x$ to $y$, the length of the $C_n$ diverges as $\eps_n$ goes to 0.
\end{corollary}
\black
\begin{definition}
    Given a discrete-time semi-flow $F$ on $X$ and a metric $d$ compatible with the topology of $X$, we call {\EM $(F,d,\eps)$-chains stream} the relation
    \beq
    \cC_{F,d,\eps} = \{(x,y):\text{ there is a $(F,d,\eps)$-chain from $x$ to $y$}\}.
    \eeq
    We call {\EM $(F,d)$-infinitesimal chains stream} (or simply {\em chains stream}) the relation
    \beq
    \cC_{F,d} = \bigcap_{\eps>0}\cC_{F,d,\eps}.
    \eeq
\end{definition}
Next two propositions grant that the relations defined above are indeed streams.
\begin{proposition}
    $\cC_{F,d,\eps}$ is a stream for every $\eps>0$.
\end{proposition}
\begin{proof}
    {\EM $\cC_{F,d,\eps}$ is a quasi-order.}
    By definition $\cC_{F,d,\eps}$ contains the diagonal of $X\times X$. Moreover, $\cC_{F,d,\eps}$ is transitive because the concatenation of two $F$-$\eps$-chain is an $F$-$\eps$-chain.

    \noindent
    {\EM $\cC_{F,d,\eps}$ is closed.}
    Assume that $(x_n,y_n)\to(x,y)$ for some sequence of pairs belonging to $\cC_{F,d,\eps}$.
    For each $n>0$, denote by $C_n$ any $\eps$-chain from $x_n$ to $y_n$.
    Then, for $n$ large enough, we can assume that $d(x,x_n)<\eps$ and $d(y,y_n)<\eps$ and so the chain $\{x\}+C_n+\{y\}$ is an $\eps$-chain from $x$ to $y$.
    In particular, $(x,y)\in\cC_{\eps,F}$.

    \noindent
    {\EM $\cO_F\subset\cC_{F,d,\eps}.$} Each $y\in\cO_F(x)$ can be joined to $x$ by a $0$-chain (orbit segment).
\end{proof}
%


Notice that, in general, the $\cC_{F,d,\eps}$ are not $\Omega$streams, as the example below shows.
\black
\begin{example}
    Let $F(x)=x^2$ on $X=[0,1]$.
    This is a discrete version of the case discussed in Example~\ref{ex:streams}: the end-points are fixed while every other point moves leftward, asymptoting to 0.
    Since $X$ is compact, $\cC_{F,d,\eps}$ does not depend on $d$ and we denote it by $\cC_{F,\eps}$.
    For every $\eps>0$, every point close enough to 0 (resp. to 1) is $\cC_{F,\eps}$-equivalent to 0 (resp. to 1).
    For $\eps\geq1$, every finite sequence of points in $X$ is an $\eps$-chain and so 
    $\cC_{F,\eps}=X\times X$.
    For $\eps>0$ small enough, points close enough to $1/2$ are further away than $\eps$ from their image, so that the node containing 0 is disjoint form the node containing 1.
    In this case, $\cC_{F,\eps}$ has two nodes ($N_0$, containing 0, and $N_1$, containing 1) and its graph is equal to the one of $\cC_F$, namely a tower with two nodes. 
    Nevertheless, for $x\in N_1$ with $x\neq1$, we have that $1\in\Down_{\cC_{F,\eps}}(x)$ although $1\not\in\cO_F(x)\cup\Down_{\cC_{F,\eps}}(0)$.
    {\black Hence, by Proposition~\ref{prop:Ostream}, $\cC_{F,\eps}$ cannot be a $\Omega$stream.}
\end{example}
\begin{lemma}
    \label{lemma:minimal}
    For every $x$ and $y\in\cO_F(x)$,
    $$
    \cO_F(x)\cup\Down_{\cC_{F,d}}(y)\supset\Down_{\cC_{F,d}}(x).
    $$
\end{lemma}
\begin{proof}
  \blue
  Let $z\in \Down_{\cC_{F,d}}(x)$, $z\not\in\cO_F(x)$.
  Let $n>0$ be such that $y=F^n(x)$.
  By case 2 of Lemma~\ref{lemma: UC discrete}, there is a $\delta>0$ such that, for every $(F,d,\delta)$-chain $(c_0,\dots,c_n)$ based at $x$, $d(y,c_n)=d(f^n(x),c_n)<\eps$.
  Let $(c_0,\dots,c_N)$, $N\geq n+1$, be a $(F,d,\delta)$-chain from $x$ to $z$.
  Such chain exists by Corollary~\ref{cor: chains with infinite length}.
  Its subchain $(c_0,\dots,c_{n+1})$ is a $(F,d,\delta)$-chain based at $x$ and so $d(F^{n+1}(x),c_{n+1})<\eps$.
  Since $F^{n+1}(x)=F(y)$, this means that $(y,c_{n+1},\dots,c_N)$ is a $(F,d,\eps)$-chain from $y$ to $z$.
  Since this holds for every $\eps>0$, then  $z\in\Down_{\cC_{F,d}}(y)$.
\end{proof}
\begin{proposition}
    \label{prop: C is a Ostream}
    $\cC_{F,d}$ is an $\Omega$stream.
\end{proposition}
\begin{proof}
    $\cC_{F,d}$ is a stream because streams are closed under intersections (Proposition \ref{prop:streams intersection}). 
    It is an $\Omega$stream as an immediate consequence of Proposition~\ref{prop:Down_S}(2) and the lemma above.
\end{proof}
\begin{proposition}
    \label{prop:CR trapping region}
    \blue
    Assume $X$ be locally compact and let $F$ have compact dynamics.
    Then each compact global trapping region of $F$ is a trapping region for $\cC_{F,d}$.
\end{proposition}
\begin{proof}
    \blue
    Denote by $Q$ a compact global trapping region of $F$.
    We need to prove that $\Down_{\cC_{F,d}}(Q)\subset Q$.
        
    Let $(x,y)\in\cC_{F,d}$ with $x\in Q$ and $y\not\in Q$ and set $\rho=d(y,Q)$.
    Notice that $\rho>0$ since $Q$ is compact.
    Since $X$ is locally compact, we can assume that $Q_\rho$ has compact closure -- if not, we set $\rho$ to a smaller value for which that holds.
    
    Recall that $\cI_F(Q_\eps)$ denotes the smallest $F$-invariant subset of $X$ containing $Q_\eps$.
    By Lemma~\ref{lemma: invariant set}, there is an $\eps>0$ such that $\cI_F(Q_\eps)\subset Q_{\rho/2}$.
    Hence, if $x\in Q_\eps$, $F^t(x)\in Q_{\rho/2}$ for all $t\geq0$.
    Since $\cI_F(A)\subset\cI_F(B)$ for $A\subset B$, we can assume without loss of generality that $\eps<\rho$.
    Recall also that, by definition, there is a $N>0$ such that $F^t(Q_\eps)\subset Q_{\eps/2}$ for all $t\geq N$.
    Finally, by Lemma~\ref{lemma: UC discrete}, there is an $\eta>0$ such that, for every $(F,d,\eta)$-chain $(c_0,\dots,c_n)$ with $n\leq N$, we have that $d(F^n(c_0),c_n)<\eps/2$.
    Since $\overline{Q_\rho}$ is compact, each $F^t$ is uniformly continuous on it and so the $\eps$ above does not depend on the starting and ending point of the chain as long as they lie in $\overline{Q_\rho}$.

    Let now $(z_0,\dots,z_N,\dots,z_{2N},\dots,z_{kN+q})$ be a $(F,d,\eta)$-chain from $x$ to $y$.
    Since $z_0=x\in Q$, $F^N(z_0)\in Q$ and $d(F^N(z_0),z_N)<\eps/2$, so $z_N\in Q_{\eps/2}$.
    Since $z_N\in Q_{\eps}$, $F^N(z_N)\in Q_{\eps/2}$ and $d(F^N(z_N),z_{2N})<\eps/2$, so $z_{2N}\in Q_{\eps}$.
    By repeating this same argument a finite number of times we get that $z_{kN}\in Q_{\eps}$.
    Finally, since $q\leq N$ and $z_{kN}\in Q_{\eps}$, we have that $F^q(z_{kN})\in\cI_F(Q_\eps)\subset Q_{\rho/2}$ and $d(F^q(z_{kN}),z_{kN+q})<\eps/2$, so that $y=z_{kN+q}\in Q_{\rho/2+\eps/2}$.
    Since $\rho/2+\eps/2<\rho$, we get a contradiction.
\end{proof}
\blue
\begin{corollary}
    Let $X$ be locally compact and assume that $F$ has a compact global trapping region $Q$.
    Then $\cR_{\cC_{F,d}}\subset Q$.
\end{corollary}
\begin{proof}
    This is an immediate consequence of Proposition~\ref{prop:CR trapping region} and case 5 of  Lemma~\ref{lemma:S trapping region}.
\end{proof}
\black
The streams $\cC_{F,d,\eps}$ depend, in general, on the particular metric $d$ used on $X$.
Nevertheless, as proved below, when $X$ is locally compact and $F$ has compact dynamics, $\cC_{F,d}$ is purely topological.
\blue
\begin{lemma}
    \label{lemma: C1=C2}
    Let $X$ be compact and let $d_1,d_2$ be any two metrics generating the topology of $X$.
    Then $\cC_{F,d_1}=\cC_{F,d_2}$ for any discrete-time semiflow $F$ on $X$.
\end{lemma}
\begin{proof}
    Suppose that there are $x,y\in X$ such that $y$ is $\cC_{F,d_1}$-downstream but not $\cC_{F,d_2}$-downstream from $x$.

    Let $C_i$ be a sequence of $(d_1,\eps_i)$-chains from $x$ to $y$, with $\eps_i\to0$. 
    Since $y$ is not $\cC_{F,d_2}$-downstream from $x$, there is a $\delta>0$ such that, for each $i$, there is a point $x_{k_i}$ on $C_i$ such that
    $
    d_2(f(x_{k_i-1}),x_{k_i}) > \delta.
    $

    Since $X$ is compact, we can assume without loss of generality that these $x_{k_i}$ converge to a point $z$. 
    This means that, for every $\eta>0$, we can find an $i$ such that
    $d_1(f(x_{k_i-1}),x_{k_i}) < \eta$ and $d_1(z,x_{k_i}) < \eta$,
    so that $d_1(z,f(x_{k_i-1})) < 2\eta$.
    In particular, also $f(x_{k_i-1}) \to z$ with respect to the $d_1$ distance. 
    Since $d_1$ and $d_2$ are equivalent, $x_{k_i}\to z$ and $f(x_{k_i})\to z$ for both $d_1$ and $d_2$.
    This means that, for every $\eta > 0$, we can find an $i$ large enough such that
    $d_2(z,x_{k_i}) < \eta$ and $d_2(z,f(x_{k_i-1})) < \eta$.
    On the other side, we also have (see above) that
    $d_2(f(x_{k_i-1}),x_{k_i}) > \delta$ for every $i$.
    These three inequalities are incompatible with the triangular inequality for $\eta$ small enough.
    Hence, we must have that $\cC_{F,d_1}=\cC_{F,d_2}$. 
\end{proof}
\begin{proposition}
    \blue
    Assume that $X$ is locally compact and let $d_1,d_2$ be any two equivalent metrics on $X$.
    If $F$ has compact dynamics,
    then $\cC_{F,d_1}=\cC_{F,d_2}$.
\end{proposition}
\begin{proof}
    \blue
    Let $Q$ be a compact global trapping region for $F$.
    By Proposition~\ref{prop:CR trapping region}, $Q$ is a trapping region for both $\cC_{F,d_1}$ and $\cC_{F,d_2}$. 
    By Proposition~\ref{prop: C is a Ostream}, we know that both $\cC_{F,d_1}$ and $\cC_{F,d_2}$ are $\Omega$streams.
    We now show that they are identical.
    
    Consider first the case $x,y\not\in Q$.
    By case 4 of Lemma~\ref{lemma:S trapping region}, $(x,y)\in\cC_{F,d_i}$, $i=1,2$, if and only if $y\in\cO_F(x)$, which is a condition independent on the distance used on $X$.
    Hence, $\cC_{F,d_1}$ and $\cC_{F,d_2}$ agree on pairs of points outside of $Q$.
    The case $(x,y)\in\cC_{F,d_i}$ with $x\in Q$ and $y\not\in Q$ cannot happen for either $i=1$ or $i=2$ because of Proposition~\ref{prop:CR trapping region}.
    Consider now the case where $x\not\in Q$ and $y\in Q$.
    Since $\cC_{F,d_1}$ and $\cC_{F,d_2}$ are $\Omega$streams, either $y\in\Omega_F(x)$ or there are $z_1,z_2\in\Omega_F(x)$ such that $x\CRtoI z_1\CRtoI y$ and $x\CRtoII z_1\CRtoII y$.
    Since each pair $(x,y)$ with $y\in\Omega_F(x)$ is in both $\cC_{F,d_1}$ and $\cC_{F,d_2}$,
    we are now left with the case when $x,y\in Q$.
    The fact that $(x,y)\in\cC_{F,d_1}$ if and only if $(x,y)\in\cC_{F,d_2}$ is proven in Lemma~\ref{lemma: C1=C2}.
\end{proof}
When $\cC_{F,d}$ is the same for all $d$, we just write $\cC_F$.

The following proposition generalizes an analogue result of Douglas Norton in~\cite{Nor95}, stated for $X$ compact. 
\blue
%
\begin{proposition}
    Assume that $X$ is locally compact and let $F$ be a semi-flow with compact dynamics.
    Then each node of $\cC_{F}$ is $F$-invariant.
\end{proposition}
\begin{proof}
    By Proposition~\ref{prop:CR trapping region} and Lemma~\ref{lemma:S trapping region}, $\cR_{\cC_F}\subset Q$.
    Hence, we can assume without loss of generality that $X$ is compact. 
    Then, by Norton's Proposition~3.8 in~\cite{Nor95}, the claim follows.
\end{proof}
\black
\subsection{Continuous-time chains streams.}
So far, we only considered the case of discrete-time chains.
Here, we prove that this can be done without loss of generality because the time-1 map $f=F^1$ of a continuous-time semi-flow $F^t$ completely determines the nodes and edges of the graph of $F^t$.
Our results extend, within a compact dynamics context, the following important result by Mike Hurley:
%
\begin{theoremX}[Hurley, 1995~\cite{Hur95}]
  \label{thm:Hurley}
  Let $F$ be a continuous-time semi-flow on a compact metric space $X$ and let $f=F^1$ be the corresponding time-1 discrete-time semi-flow.    
  Then $\cR_{\cC_{F,d}}=\cR_{\cC_{f,d}}$.
\end{theoremX}
Notice that the result above is not stated explicitly in~\cite{Hur95} but is rather a corollary of a more general result (Theorem~5 in~\cite{Hur95}) that holds, in general metric spaces, 
for a stronger version of chain-recurrence, where the ``$\eps$'' of an $\eps$-chain is not a constant but rather a strictly positive function. 
In case of a compact metric space, this general result reduces to Theorem~\ref{thm:Hurley}.

We start with the following definition.
\begin{definition}
  \label{def: continuous-time chains}
  Given a continuous-time semi-flow $F$ on $X$ and a metric $d$ compatible with the topology of $X$, given $\eps>0$ and $T>0$, a {\EM$(F,d,\eps,T)$-chain} of length $n+1$ from $x$ to $y$ is a sequence of $n+1$ points $c_0,\dots,c_n$ together with a finite sequence of positive real numbers $t_0,\dots,t_{n-1}$ such that:
    \begin{enumerate}
        \item $c_0=x$, $c_n=y$;
        \item $d(F^{t_i}(c_i),c_{i+1})\leq\eps$ for all $i=0,\dots,n-1$;
        \item $t_i\geq T$ for all $i=0,\dots,n-1$.
    \end{enumerate}
\end{definition}
The following technical lemma is a continuous-time analogue of Lemma~\ref{lemma: UC discrete}.
\begin{lemma}[Hurley, 1995]
    \label{lemma: UC cont}
    Let $F$ be a continuous-time semi-flow on $X$.
    Then, for any $\eps>0$, $T>0$ and $p>0$, there exists a $\delta>0$ such that:
    \begin{enumerate}
        \item 
        $
        d(x,y)<\delta \implies d(F^t(x),F^t(y))<\eps\;\text{ for every }\;t\in[0,T];
        $
        \item 
        for every $(F,d,\delta,T)$-chain $C$ with $p+1$ points $c_0,\dots,c_p$ and times $t_0,\dots,t_{p-1}$, the chain $C'$ with points $c_0,c_p$ and time $\tau=\sum_{i=0}^{p-1}t_i$ is a $(F,d,\eps,\tau)$-chain.
        Equivalently, 
        $$d(F^{\tau}(c_0),c_p)<\eps.$$
    \end{enumerate}
    If $f$ is uniformly continuous, then $\delta$ can be chosen independently on $x$.
\end{lemma}
\begin{proof}
    \blue
    (1)
    We leave it to the reader.

    (2)
    The proof goes exactly as for the discrete-time version, case 2 of Lemma~\ref{lemma: UC discrete}.
    %
\end{proof}
\begin{definition}
    We call {\EM $(F,d,\eps,T)$-chains stream} the relation
    \beq
    \cC_{F,d,\eps,T} = \cO_F\cup\{(x,y):\text{ there is a $(F,d,\eps,T)$-chain from $x$ to $y$}\}.
    \eeq
    We call {\EM $(F,d,T)$-infinitesimal chains stream}
    the relation
    \beq
    \cC_{F,d,T} = \bigcap_{\eps>0}\cC_{F,d,\eps,T}.
    \eeq
    We call {\EM $(F,d)$-infinitesimal chains stream} (or simply {\em chains stream}) the relation
    \beq
    \cC_{F,d} = \bigcap_{T>0}\cC_{F,d,T}.
    \eeq
    By analogy, given a discrete-time semi-flow $f$, we call {\EM $(f,d,\eps,N)$-chain} of length $n+1$ from $x$ to $y$ a sequence of $n+1$ points $c_0,\dots,c_n$ together with a finite sequence of positive integers $k_0,\dots,k_{k-1}$ such that:
    \begin{enumerate}
        \item $c_0=x$, $c_n=y$;
        \item $d(f^{k_i}(c_i),c_{i+1})\leq\eps$ for all $i=0,\dots,n-1$;
        \item $k_i\geq N$ for all $i=0,\dots,n-1$.
    \end{enumerate}
    We call {\EM $(f,d,\eps,N)$-chains stream} the relation
    \beq
    \cC_{f,d,\eps,N} = \cO_f\cup\{(x,y):\text{ there is a $(f,d,\eps,N)$-chain from $x$ to $y$}\}
    \eeq
    and {\EM $(f,d,N)$-infinitesimal chains stream}
    the relation
    \beq
    \cC_{f,d,N} = \bigcap_{\eps>0}\cC_{f,d,\eps,N}.
    \eeq
\end{definition}
%
The reader can verify that all the relations above are indeed streams.
\blue
\begin{proposition}
    \label{prop: cont Ostream}
    Let $F$ be a continous-time semiflow.
    Then $\cC_{F,d}$ is an $\Omega$stream and, for every $T>0$, $\cC_{F,d,T}$ is an $\Omega$stream.
\end{proposition}
\begin{proof}
    We start with the case of $\cC_{F,d,T}$.
    Let $y=F^{\bar t}(x)$.
    We need to prove that 
    $$
    \cO_F(x)\cup\Down_{\cC_{F,d,T}}(y)\supset\Down_{\cC_{F,d,T}}(x).
    $$
    Let $z\in\Down_{\cC_{F,d,T}}(x)\setminus\cO_F(x)$ and let $\eps>0$.
    Then there is a $(F,d,T,\eps)$-chain from $x$ to $z$ with points $(c_0,\dots,c_N)$ and times $(t_0,\dots,t_{N-1})$.
    Notice that, since a $(F,d,T,\eps)$-chain is also a $(F,d,T',\eps)$-chain for all $T'<T$, we can assume without loss of generality that $T>\bar t$.
        
    By Lemma~\ref{lemma: UC cont}, there exists $\delta>0$ such that, for every $(F,d,T,\delta)$-chain $(c_0=x,c_1,c_2)$ with times $(t_0,t_1)$, we have that $d(F^{t_0+t_1}(x),c_2)<\eps$, namely
    $$
    d(F^{t_0+t_1-\bar t}(y),c_2)<\eps.
    $$
    We can assume without loss of generality that $\delta\leq\eps$.
    Since $\bar t<T\leq t_0$, then the sequence $(y,c_2,\dots,c_N)$ with times $(t_0+t_1-\bar t,t_2,\dots,t_{N-1})$ is a $(F,d,T,\eps)$-chain from $y$ to $z$.

    Since every $\cC_{F,d,T}$ is an $\Omega$stream, then $\cC_{F,d}=\cap_{t\geq0}\cC_{F,d,T}$ is an $\Omega$stream as well.
\end{proof}
\begin{proposition}
    Let $f$ be a discrete-time semiflow.
    Then, for every $N>0$, $\cC_{f,d,N}$ is an $\Omega$stream.
\end{proposition}
\begin{proof}
    The argument is identical to the one used to prove Proposition~\ref{prop: cont Ostream}.
\end{proof}

%
We now state, for continuous-time chains, some statement we proved in case of discrete-time and that can be proved with the very same proof of their discrete-time version.
\begin{proposition}
    \label{prop: cont-time chains}
    Let $F$ be a continuous-time semiflow on $X$ and set $f=F^1$.
    If $X$ is locally compact and $F$ has a compact global trapping region $Q$, the following holds:
    \begin{enumerate}
        \item $Q$ is a trapping region for $\cC_{F,d}$, for $\cC_{F,d,T}$ for all $T>0$ and for $\cC_{f,d,N}$ for all $N>0$;
        \item $\cR_{\cC_{F,d}}$, each $\cR_{\cC_{F,d,T}}$, $T>0$, and each $\cC_{f,d,N}$ for all $N>0$ are subsets of $Q$;
        \item $\cR_{\cC_{F,d}}$, each $\cR_{\cC_{F,d,T}}$, $T>0$,  and each $\cC_{f,d,N}$ for all $N>0$ are $F$-invariant;
        \item $\cC_{F,d}$, each $\cR_{\cC_{F,d,T}}$, $T>0$, and each $\cC_{f,d,N}$ for all $N>0$ are independent on the metric.
    \end{enumerate}
\end{proposition}
Under assumptions that grant the metric independence of a relation, we will use for them the notation $\cC_F$, $\cC_{F,T}$ and $\cC_{f,N}$.

\black
Next two lemmas show that, for all that concerns infinitesimal chains, it is enough to consider the time-1 map $f=F^1$.
\begin{lemma}
    \label{lemma:long chains}
    For any $\eps>0$ and $x\in\cR_{\cC_{F,d,T}}$, there are $(F,d,\eps,T)$-chains of arbitrarily large length from $x$ to itself.
\end{lemma}
\begin{proof}
    Fix any integer $n>0$.
    By hypothesis, there is at least a $(F,d,\eps,T)$-chain $C$ from $x$ to itself.
    By concatenating $C$ with itself enough times, the result is a $(F,d,\eps,T)$-chain of length larger than $n$.
\end{proof}
In several statements below, starting from next one, we will use the notation ${\BF \lfloor T\rfloor}$ to indicate the largest integer not larger than $T$.
\begin{lemma}
    \label{lemma:TvsN}
    Assume that $X$ is locally compact and let $F$ be a continuous-time semi-flow with compact dynamics.
    Let $x$ and $y$ be $\cC_{F,T}$-equivalent 
    and set $f=F^1$ and $N=\lfloor T\rfloor$.
    Then $x$ and $y$ are $\cC_{f,N}$-equivalent.
\end{lemma}
\begin{proof}
    Let $Q$ be a compact global trapping region for $F$.
    By Proposition~\ref{prop: cont-time chains},
    $Q$ is also a trapping region for both streams $\cC_{F,T}$ and $\cC_{f,N}$.
    Hence, their recurrent points are entirely contained in $Q$ and so it is enough to consider the analogue problem for the restriction of $F$ to $Q$.  
    Hence, in the reminder of the proof we assume, without loss of generality, that $X$ is compact.
    
    We will prove that, for every $\eps>0$, there is a $(f,N,\eps)$-chain from $x$ to $y$. 
    The same argument then can be used to show that there is a $(f,\eps,N)$-chain from $y$ to $x$. 

    Fix an $\eps>0$ and let $\delta>0$ satisfy point (1) of Lemma~\ref{lemma: UC cont} and point (2) of Lemma~\ref{lemma: UC discrete} with $p=N$.
    We can assume without loss of generality that $\delta\leq\eps$.
    Let $C$ be a $(F,T,\delta)$-loop based at $x$ with points $(c_0,\dots,c_r)$ and times $(t_0,\dots,t_{r-1})$ such that 
    $c_i=y$ for some $0<i<r$.
    \blue
    Set 
    $S=t_0+\dots+t_{r-1}$.
    Notice that, if $S/N$ is irrational, since rationals are dense, we can change $t_{r-1}$ to a new time $t'_{r-1}$ so that the new chain $(c_0,\dots,c'_{r-1},c_r)$ is still a $(F,T,\delta)$-loop based at $x$ but this time its period $S'$ is such that $S'/N$ is rational.
    Hence, we can assume without loss of generality that $S/N$ is rational.
    \black
    
    Following Hurley~\cite{Hur95}, we build a $(f,N,\eps)$-loop $C'$ based at $x$ 
    in the following way.
    For every $j$, set $s_j=\sum_{i=0}^jt_j$. 
    The $s_j$ are precisely the times at which, on the chain $C$, there are jumps -- precisely, a jump from $F^{t_j}(c_{j})$ to $c_{j+1}$.
    We start $C'$ by setting $c'_0=c_0$.
    Then, after Hurley, for each $k$, we follow the rule below:
    \begin{enumerate}
        \item if there is no jump in $C$ in the interval $(kN,(k+1)N]$, then we set $c'_{k+1}=f^N(c'_k)=F^N(c'_k)$;
        \item if there is a jump in $C$ at $s_j\in(kN,(k+1)N]$, then we set\\ $c'_{k+1}=F^{(k+1)N-s_j}(c_{j+1})$.
    \end{enumerate}
    The two cases above cover all possible cases because, since in $C$ jumps take place at least $T\geq N$ time units apart, there can be at most one jump in each interval $(kN,(k+1)N]$.

    \blue
    Once the $(F,T,\delta)$-loop gets back to $x$, not necessarily this is the case for the $(f,N,\eps)$-chain.
    Nevertheless, recall that $S/N$ is rational, namely there are integers $m,n>0$ such that $mS=nN$.
    The chain obtained by repeating $m$ times the sequence $(c_0,\dots,c_r)$ and the relative times $(t_0,\dots,t_{r-1})$ is still a $(F,T,\delta)$-loop.
    Hence, after applying the construction above to this new loop, the next-to-last of the $(f,N,\eps)$-chain coincides with the next-to-last point of the $(F,T,\delta)$-loop.
    Since $\delta\leq\eps$, this shows that the $(f,N,\eps)$-chain can be completed to a $(f,N,\eps)$-look based at $x$ by adding $x$ as the last point of the chain.
\end{proof}
\begin{corollary}
    \label{cor:Ff}
    Assume that $X$ is locally compact and $F$ is a continuous-time semi-flow with compact dynamics.
    Set $f=F^1$, fix a $T>0$ and set $N=\lfloor T\rfloor$.
    Then $\cR_{\cC_{F,T}}=\cR_{\cC_{f,N}}$ and each node of $\cR_{\cC_{F,T}}$ is a node of $\cR_{\cC_{f,N}}$ and viceversa.
\end{corollary}
\begin{lemma}
    Assume that $X$ is locally compact and $F$ is a continuous-time semi-flow with compact dynamics.
    Set $f=F^1$, fix a $T>0$ and set $N=\lfloor T\rfloor$.
    Then, if $x$ is $\cC_{F,T}$-upstream of $y\in\cR_{\cC_{F,T}}$,
    $x$ is $\cC_{f,N}$-upstream of $y$.
\end{lemma}
\begin{proof}
    Let $C$ be a $(F,T,\delta)$-chain from $x$ to $y$ and let $D$ be a $(F,T,\delta)$-chain loop from $y$ to itself.
    Let $D^n$ be the concatenation of $D$ with itself $n$ times.
    Then by concatenating $C$ with $D^n$ we can get a chain from $x$ to $y$ of arbitrary length.
    Hence, by using the very same procedure of the previous lemma, we can prove that, for every $\eps>0$, there is a $(f,N,\eps)$-chain from $x$ to $y$.
\end{proof}
\begin{corollary}
    Assume that $X$ is locally compact and $F$ is a continuous-time semi-flow with compact dynamics.
    Set $f=F^1$, fix a $T>0$ and set $N=\lfloor T\rfloor$.
    Then $\Gamma_{\cC_{F,T}}=\Gamma_{\cC_{f,N}}$.
\end{corollary}
The results above show already that all that the qualitative description of the dynamics of a continuous-time semi-flow $F$ with compact dynamics is all encoded in the powers of its time-1 map. 
Below we show that, in fact, the first power of the time-1 map is enough.
\begin{theorem}
    \label{thm: cont=disc}
    Assume that $X$ is locally compact and $f$ is a discrete-time semi-flow with compact dynamics.
    Then, for any integer $N>0$, $\cO_f\cup\cC_{f,N}=\cC_{f}$.
    If $f=F^1$ for some continuous-time semi-flow $F$, then we have also that $\cC_{F}=\cO_F\cup\cC_{f}$.
    In particular, $\Gamma_{\cC_F}=\Gamma_{\cC_f}$.
\end{theorem}
\begin{proof}
    As in the proof of Lemma~\ref{lemma:TvsN}, we can assume without loss of generality that $X$ is compact.

    First notice that, for any $N>0$, $\cC_{f}\subset\cO_f\cup\cC_{f,N}$ because every $(f,N,\eps)$-chain $C$ can be seen as a $(f,\eps)$-chain -- just break each jumpless segment in pieces of length 1.
    To complete the proof, we need to prove that, given any $N>0$ and $\eps>0$, if $x$ can be joined to $y$ by a $(f,\eta)$-chain for every $\eta>0$, then we can join $x$ to $y$ with a $(f,N,\eps)$-chain.

    So, let $\eps>0$, set $p=2N$ and let $\delta>0$ be the $\delta$ whose existence is granted by Lemma~\ref{lemma: UC discrete}(2) and $C$ a $(f,\delta)$-chain from $x$ to $y$.
    Recall that, by possibly concatenating $C$ with some $(f,\delta)$-chain from $y$ to itself, we can assume that $C$ has at least $N$ points.
    Let $c_0,\dots,c_r$, $r\geq N$, be the points of $C$.
    Then, by Lemma~\ref{lemma: UC discrete}(2), each pair $c_{kN},c_{(k+1)N}$ is a $(f,N,\eps)$-chain.
    If $r=q N$ for some integer $q>0$, then $c_0,c_N,\dots,c_{q N}$ is a $(f,N,\eps)$-chain from $x$ to $y$.
    Otherwise, $q N<r<(q+1)N$ for some $q>0$.
    In this case, we use as the final segment of the new chain the pair $c_{(q-1)N},c_r$. 
    Since $2N>r-(q-1)N>N$, even this pair is a $(f,N,\eps)$-chain.
    Hence, the chain $c_0,c_N,\dots,c_{(q-1)N},c_r$ is, in any case, a $(f,N,\eps)$-chain from $x$ to $y$.
    This proves that $\cO_f\cup\cC_{f,N}=\cC_{f}$.

    The second claim of the theorem comes from the fact that $\cC_{F,T} = \cO_F\cup\cC_{f,\lfloor T\rfloor}$ (Corollary~\ref{cor:Ff}) and that $\cO_F\cup\cC_{f,N}=\cO_F\cup\cC(f)$ for every integer $N>0$ (by the first claim of this theorem).
\end{proof}
\subsection{The {\EM$\Sigma$}chains streams}
{\black This subsection is based on Easton's definition of $\Sigma$chains given in Definition~\ref{def:chains}.}
\begin{definition}
    Given a metric $d$ compatible with the topology of $X$, we call {\EM $(F,d)$-$\Sigma$chains stream} the relation
    \beq
    \Sigma_{F,d}  = \{(x,y):\text{for every $\eps>0$, there is a $(F,d,\eps)$-$\Sigma$chain from $x$ to $y$}\}.
    \eeq
\end{definition}
\begin{proposition}
    $\Sigma_{F,d} $ is an $\Omega$stream.
\end{proposition}
\begin{proof}
    {\EM $\Sigma_{F,d} $ is a quasi-order.}
    By definition $\Sigma_{F,d} $ contains the diagonal of $X\times X$. 
    Now, suppose that $(x,y),(y,z)\in\Sigma_{F,d} $ and fix an $\eps>0$. 
    By choosing an $\eps/2$-$\Sigma$chain from $x$ to $y$ and another one from $y$ to $z$, their concatenation is an $\eps$-$\Sigma$chain from $x$ to $y$.
    Hence, $\Sigma_{F,d} $ is transitive.
    
    \noindent
    {\EM $\Sigma_{F,d} $ is closed.}
    Assume that $(x_n,y_n)\to(x,y)$ for some sequence of pairs belonging to $\Sigma_{F,d} $.
    For each $n>0$, denote by $C_n$ any $\eps/2$-$\Sigma$chain from $x_n$ to $y_n$.
    Then, for $n$ large enough, we can assume that $d(x,x_n)<\eps/4$ and $d(y,y_n)<\eps/4$ and so the chain $\{x\}+C_n+\{y\}$ is an $\eps$-$\Sigma$chain from $x$ to $y$.
    In particular, $(x,y)\in\Sigma_{F,d} $.

    \noindent
    {\EM $\cO_F\subset\Sigma_{F,d}$}. Each $y\in\cO_F(x)$ can be joined to $x$ by a $0$-chain (orbit segment).
    
    \noindent
    {\EM $\Sigma_{F,d} $ is an $\Omega$stream.} The same argument used to prove that $\cC_F$ is an $\Omega$stream applies to this case, since every $(F,d,\eps)$-$\Sigma$chain from $x$ to $y$ is also a $(F,d,\eps)$-chain between the same points.
\end{proof}
%
%
Notice that, unlike chains streams, $\Sigma$chains streams do depend on the distance $d$ used in their definition.
\begin{example}
    Let $F$ be the flow of an ODE $x'=g(x)$ on $X=[0,1]$, where $g$ is a non-positive Lipschitz function on $X$ whose set of zeros coincide with the middle-third Cantor set $C$. 
    Since the measure of $C$ with respect to the Euclidean distance $e$ on $X$ is zero, there are $(F,e,\eps)$-$\Sigma$chains from 0 to 1 for every $\eps>0$, so that all points of $X$ are $\Sigma_{F,e}$-equivalent to each other.
    Now, let $d$ be a distance topologically equivalent to $e$ but for which the measure of $C$ is strictly positive.
    Then, for $\eps$ small enough, no $(F,d,\eps)$-$\Sigma$chain can join 0 to 1, so that $\Sigma_{F,e}\neq\Sigma_{F,d}$ (see~\cite{Wis18} for several other examples).
\end{example}


\subsection{The smallest stream}
By Proposition~\ref{prop:smallest stream}, we know that each semi-flow $F$ has a smallest stream, namely the smallest closed quasi-order containing $\cO_F$ (and, therefore, containing $\NW_F$).
This relation, that we denote by $\cA_F$, has been first introduced and studied by Auslander in~\cite{Aus63}.
In that article, under the assumption that $X$ is separable and locally compact, Auslander introduced $\cA_F$ as the set
of all pairs $(x,y)$ such that $L(x)\geq L(y)$ for all Lyapunov functions $L\in\cL_F$.
Hence, $\cA_F=\cD_{\cL_F}$.
In the proposition below, we show that $\cA_F$ can also be seen as a stream of chains, in particular a $\Sigma$chains stream.
We also drop Auslander's assumptions on $X$ and we rather transfer them to $F$.

    

%
\begin{proposition}
    Assume that $F$ has a separable and locally compact global trapping region $Q$.
    Denote by $\cM_X$ the set of all distances on $X$ compatible with its topology.
    Then the following holds:
    \begin{enumerate}
        \item $\cA_F=\displaystyle\bigcap_{d\in\cM_X}\Sigma_{F,d}$;
        \item $\cR_{\cA_F}=\displaystyle\bigcap_{d\in\cM_X}\cR_{\Sigma_{F,d} }$;
        \item $\cA_F$ is an $\Omega$stream.
        \item $\cA_F=\Sigma_{F,d} $ for some $d\in\cM_X$.
    \end{enumerate}
\end{proposition}
\begin{proof}
    (1)
    Fix a metric $d$ on $X$ compatible with its topology.
    It is enough to prove that, for each $h\in\cL_F$, there exists a metric $d_h$ compatible with $d$ such that $h(x)\geq h(y)$ for each $(x,y)\in\Sigma_{F,d_h}$.
    Indeed, let $\cL_{\Sigma_{F,d} }$ be the set of all Lyapunov function for $\Sigma_{F,d} $. 
    Then $\cap_d\Sigma_{F,d} =\cD_{\cup_d\cL_{\Sigma_{F,d} }}$.
    Hence, if such $d_h$ exists, then $h\in\cL_{\Sigma_{F,d_h}}$ and so $\cup_d\cL_{\Sigma_{F,d} }=\cL_F$.
    Therefore, ultimately, $\cap_d\Sigma_{F,d} =\cD_{\cL_F}=\cA_F$.
    
    We claim that $d_h(x,y)\bydef d(x,y)+|h(y)-h(x)|$ does the job.
    Indeed, suppose that $(x,y)\in\Sigma_{F,d_h}$. 
    Then, for every $\varepsilon>0$, there exists a chain $x_0,\dots,x_k$, with $x_0=x$ and $x_k=y$, such that 
    $$
    d_h(f(x_0),x_{1})+\dots+ d_h(f(x_{k-1}),x_{k})< \varepsilon.
    $$
    This means that
    $$
    d(f(x_0),x_{1})+\dots+ d(f(x_{k-1}),x_{k})+
    $$
    $$
    +|h(f(x_0))-h(x_1)|+\dots+|h(f(x_{k-1}))-h(x_k)|< \varepsilon.
    $$
    Hence, for every $\varepsilon>0$, we have that
    $$
    h(y) = h(x_k) \leq h(x_k) \overbrace{- h(f(x_{k-1})) + h(x_{k-1})}^\text{non-negative since $h\in\cL_F$} - \dots \overbrace{- h(f(x_0)) + h(x_0)}^\text{non-negative since $h\in\cL_F$} \leq
    $$
    $$
    \leq |h(x_k)-h(f(x_{k-1}))|+\dots+|h(x_1)-h(f(x_{0}))|+h(x_0)\leq
    $$
    $$
    \leq \varepsilon + h(x),
    $$
    so that
    $$
    h(y)\leq h(x).
    $$

    \noindent
    (2)
    It follows immediately from point (1) above and Proposition~\ref{prop:intersection of minimal}(2).

    \noindent
    (3)
    It follows immediately from point (2) above and Proposition~\ref{prop:intersection of minimal}(1).

    \noindent
    (4) 
    By Theorem~2 in~\cite{Aus63}, there is a function $L\in\cL_F$ such that $\cA_F=\cD_{\{L\}}$.
    In other words, $(x,y)\in\cA_F$ if and only if $L(x)\geq L(y)$ and $x\in\cR_{\cA_F}$ if and only if $L$ is constant on the orbit of $x$.
    Now, let $d$ be any distance compatible with the topology of $X$ and define 
    $$
    d_L(x,y) = d(x,y) + |L(x)-L(y)|.
    $$
    Then the same argument used in point (1) of this proposition shows that $(x,y)\in\Sigma_{F,d_L}$ implies that $L(x)\geq L(y)$ which, in turn, implies that $(x,y)\in\cA_F$.
    Since $\cA_F$ is contained in every stream, it follows that $\cA_F=\Sigma_{F,d_L}$.
\end{proof}
\subsection{Chains streams with countably many nodes}
\label{sec:countably many nodes}
\begin{definition}
    Assume $N_1,N_2$ are distinct nodes of a stream $S$.
    We say they are {\EM adjacent} if $N_1\Sto N_2$ and, whenever $N_1\Sto N\Sto N_2$, then either $N=N_1$ or $N=N_2$.
\end{definition}
\black
The following lemma generalizes to streams the content of Lemma~\ref{lemma:NW adjacent} and it is proved using the very same argument.
\black
\begin{lemma}
    \label{lemma:streams adjacent}
    Let $F$ be a semi-flow with compact dynamics 
    and let $S$ be a substream of $\cC_F$.
    Then, there is a bitrajectory between every pair of adjacent nodes of $S$. 
\end{lemma}

\black
The following is the last of our main results and shows that, when the smallest stream has only countably many nodes,
it coincides with the chains stream, that is usually easier to deal with.
\black
\begin{theorem}
    \label{thm: CF=AF}
    Let $F$ be a semi-flow with compact dynamics.
    Then, if $\cA_F$ has countably many nodes, 
    $\cA_F=\cC_F$. 
    In particular, $\Sigma_{F,d}=\cC_F$ for every metric $d$
    compatible with the topology of $X$.    
\end{theorem}
\begin{proof}
  \blue Recall that $\cA_F=\cD_{\cL_F}$. Hence, it is enough to show \black that $\cC_F=\cD_{\cL_F}$, namely that there is an $\eps$-chain from $x$ to $y$ for every $\eps>0$
  if and only if $L(x)\geq L(y)$ for each Lyapunov function $L$ of $F$.
  \black
  Since $\cD_{\cL_F}$ is the smallest stream of $F$, we already know that $\cC_F\supset\cD_{\cL_F}$ and so it is enough to prove the inverse inclusion.
  In turn, this amounts to show that $L(x)\geq L(y)$ for every $(x,y)\in\cC_F$.
%
  If $y\in\cO_F(x)$ the claim is trivially true, so let us assume that $y\not\in\cO_F(x)$.
  \black
  Since $F$ has a compact globally attracting trapping region, $\Omega_F(x)$ is not empty and is contained in a node $T$ of $\cC_F$.
  Similarly, since $y\in\Down_{\cC_F}(\Omega_F(x))$, there is at least a node $B$, possibly equal to $T$, such that $y\in\Down_{\cC_F}(B)$.
  By Zorn's Lemma, there is at least a tower $\cT$ of $\Gamma_{\cC_F}$ having $T$ as a top node and $B$ as a bottom node.
    
  Notice that, by hypothesis, $\cT$ has at most countably many nodes.
  Denote by $\cT'$ the set of all nodes of $\cT$ that are not isolated. 
  These nodes subdivide the set of all isolated nodes of $\cT$ into at most countably many sequences $\cT_i$ ordered with respect to $\cC_F$.
  In each $\cT_i$, every pair of consecutive nodes is a pair of adjacent nodes.
  Hence, any $L\in\cL_F$ is non-decreasing on each $\cT_i$.
    
  Now, consider the set $\cT^*$ of all nodes of $\cT$ minus all nodes contained in the $\cT_i$.
  This set is at most countable and we can repeat the same argument above as follows. 
  Denote by $(\cT^*)'$ the set of all nodes of $\cT^*$ that are not isolated. 
  These nodes subdivide the set of all isolated nodes of $\cT^*$ into at most countably many sequences $\cT^*_i$ ordered with respect to $\cC_F$.
  In each $\cT^*_i$, between every pair $N_1$ and $N_2$ of consecutive nodes is either no node, in which case $N_1$ and $N_2$ are a pair of adjacent nodes, or there is one of the $\cT_i$ sets of nodes.
  In either case, by continuity, $L(N_1)\geq L(N_2)$ and so, more generally, any $L\in\cL_F$ is non-decreasing on each $\cT^*_i$ and so, by continuity, on the union of the nodes in all $\cT_i$ and $\cT^*_i$.

  Now, recall that a closed set is countable if and only if its Cantor-Bendixson rank is finite, namely if there is an ordinal $\alpha_0$ such that 
  there is a finite number of accumulation points of order $\alpha$.
  Repeating the argument above, we can prove that, at every order $\alpha$, any $L$ is non-decreasing on the set of isolated limit points of $\cT$ of order $\alpha$.
  Since there is a finite number of limit points of order $\alpha_0$, the argument stops at that point. 
  Hence, ultimately, $L$ is non-decreasing on the whole $\cT$.
  By continuity, this means that $L(x)\geq L(y)$.
%
    %
    %
\end{proof}
\medskip\noindent
\subsection{The chains graph of T-unimodal maps.}
%
We conclude this section by mentioning an immediate consequence of our Theorem~\ref{thm:tower}, namely that the structural graph $\Gamma_{\ell_\mu}$ of the logistic map $\ell_\mu$
is a tower (see Section~\ref{sec:T-unimodal}).
By definition, the nodes of the structural graph of a T-unimodal map $f$ are the nodes of $\cC_{f}$.  
There is an edge from $M$ to $N$ in $\Gamma_{f}$ if there is a bitrajectory $b$ with $\alpha(b)\subset M$ and $\omega(b)\subset N$. 
In turn, this immediately implies that $M$ is $\cC_{f}$-upstream from $N$.
Hence, we have the following:
\begin{proposition}
    For each T-unimodal map $f$, $\cA_f=\cC_f$ and $\Gamma_{\cC_{f}}=\Gamma_{f}$.
    Hence, the smallest stream of $f$ is the chains stream and its graph is a tower.
\end{proposition}
%
%

%
\section{Final comments}
\medskip\noindent
{\bf An interpretation of streams in terms of modeling.}
No model of a natural phenomenon can be able to reproduce exactly its dynamics.
In other words, we can safely assume that, given a semi-flow $F$, the dynamics of $F$ represents the real behavior only within some $\eps>0$, depending on the degree of accuracy of the model.
The smaller $\eps$, the more accurate the model is, but we can safely assume that $\eps>0$.
This means that what gives us the most reliable information on the dynamics of the system we are modeling is, more than the qualitative dynamics of $F$ (that would be described by the graph of the prolongational relation), the graph of some $F$-stream.

For instance, suppose we are modeling some phenomenon using a discrete-time model and that we do not have precise information on the amount of the error of our model at every step, except perhaps that the error is ``not too large''. 
Then the graph of the chains stream $\cC_F$ is one of the most suitable graphs to describe the qualitative behavior of the system. 
Indeed, by its very definition, the nodes and edges of the chains graph describes the behaviour of a point under $F$ in presence of arbitrarily small perturbations.
Recall that, when the number of nodes is countable, as happens in most applied cases, the chains stream actually coincides with the Auslander stream $\cA_F$.
In case some quantitative information on the error is available, then streams such as $\cC_{F,d,\eps}$ might be more suitable to describe the actual dynamics of the system described by $F$. For instance, if we model a motion in the $[0,1]$ segment as in Figure~\ref{fig:streams}(left) and we know that the model suffers by some ``large'' error for $x$ close to $0.5$, then, depending on the amount of the error, it could happen that the corresponding stream looks like the one in 
Figure~\ref{fig:streams}(center), where the new node means that, due do the large errors, the true dynamics of the natural system is compatible with the presence of a third node about $0.5$.
Possibly a more accurate model could tell whether such a node exists or not.

\medskip\noindent
{\bf An interpretation of streams in terms of controlled systems.}
Given an $F$-stream $S$, our interpretation of $x\Sto y$ is that, under possibly some kind of control, it is possible to start at $x$ and end up at $y$.

The relation $\cO_F$ is the one that represents the system in absence of controls: it is possible to end up in finite time in $y$ when starting at $x$ if and only if $y$ is in the orbit of $x$.
Any stream, including $\cA_F$, describes the behavior of the system under some kind of control.

\black
For instance, consider the three streams shown in Figure~\ref{fig:streams}.
\black
The first (left) describes what happens in the system when arbitrarily small controls can be applied at any point.
There are two effects.
The first is that it is possible to leave from the fixed point 1 and reach any other point of the interval -- this is why $1\Sto x$ for every $x\in[0,1]$.
The second is that it is possible to jump on the fixed point 0 in finite time from the orbit of any point $x<1$.

The second (center) describes what happens in the system when, besides the ``infinitesimal control'' above, one applies some finite control in a region in the middle.
Although the flow tends to move all points towards 0, with this finite control one can keep indefinitely a trajectory within the intermediate node and can reach any point within it.

The third describes what happens when one applies some finite control close enough to a pre-existing node. 
In this case, rather than having a new node, a previously existing node ``extends'' into a larger node.
Notice that this node is forward-invariant but not invariant.

In this context, also the prolongational relation plays a role. 
Its graph shows from which node to which node it is possible to get by using only two controls, a first (arbitrarily small) ``push'' at the very beginning of the motion and a second (arbitrarily small) ``push'' at its very end.
\section*{Acknowledgments}
The authors are in debt with Joe Auslander for several stimulating discussions on the matter.
The authors are also thankful to the anonymous referees for helping improving the quality of the article.
The first author was partially supported by NSF grant \# 2308225.

\bibliographystyle{amsplain}  
\bibliography{refs}  
\end{document}